\documentclass[12pt]{amsart}
\usepackage[margin=1in]{geometry}
\usepackage{amsmath,amssymb,setspace}
\usepackage{indentfirst}
\usepackage{comment}
\usepackage{booktabs}
\usepackage{wrapfig}
\usepackage{longtable}
\usepackage{lipsum}
\newcommand{\ints}{\mathbb{Z}}
\newcommand{\real}{\mathbb{R}}

\newcommand{\cplx}{\mathbb{C}}

\newcommand{\hcal}{\mathcal{H}}

\newcommand{\supp}{\text{Supp}}

\newcommand{\fiber}{\text{Fiber}}
\newcommand{\mon}{\text{Mon}}
\newcommand{\sgn}{\text{sgn}}

\renewcommand{\phi}{\varphi}

\renewcommand{\bar}{\overline}
\newcommand{\im}{\text{im }}
\renewcommand{\hom}{\text{Hom}}
\newcommand{\aut}{\text{Aut}}

\newcommand{\End}{\text{End}}

\newcommand{\ext}{\text{Ext}}

\newcommand{\ann}{\text{Ann}}

\newcommand{\id}{\text{Id}}

\newcommand{\stab}{\text{Stab}}

\newcommand{\res}{\text{Res}}
\newcommand{\ind}{\text{Ind}}

\newcommand{\eu}{\text{eu}}
\newcommand{\irr}{\text{Irr}}
\newcommand{\hres}{{ }^H\text{Res}}
\newcommand{\hind}{{ }^H\text{Ind}}
\newcommand{\htw}{{ }^H\text{tw}}

\newcommand{\cz}{{ }^Cz}
\newcommand{\cT}{{ }^CT}
\newcommand{\tw}{\text{tw}}
\newcommand{\loc}{\text{Loc}}
\newcommand{\triv}{\text{triv}}

\newcommand{\oscr}{\mathcal{O}}

\newcommand{\bfr}{\mathfrak{b}}

\newcommand{\gfr}{\mathfrak{g}}
\newcommand{\hfr}{\mathfrak{h}}

\newcommand{\nfr}{\mathfrak{n}}

\DeclareFontFamily{OT1}{rsfs}{}
\DeclareFontShape{OT1}{rsfs}{n}{it}{<-> rsfs10}{}
\DeclareMathAlphabet{\mathscr}{OT1}{rsfs}{n}{it}

\usepackage[small,nohug,heads=vee]{diagrams}
\diagramstyle[labelstyle=\scriptstyle]

\mathchardef\mhyphen="2D 

\newtheorem{theorem}{Theorem}
\newtheorem{corollary}[theorem]{Corollary}
\newtheorem{lemma}[theorem]{Lemma}
\newtheorem{proposition}[theorem]{Proposition}
\newtheorem{definition}[theorem]{Definition}
\newtheorem{remark}[theorem]{Remark}

\numberwithin{theorem}{subsection}

\begin{document}

\title [On Refined Filtration By Supports for Rational Cherednik Categories $\oscr$]{On Refined Filtration By Supports for Rational Cherednik Categories $\oscr$} \author{Ivan Losev and Seth Shelley-Abrahamson}\date{\today}\maketitle

\begin{abstract} For a complex reflection group $W$ with reflection representation $\hfr$, we define and study a natural filtration by Serre subcategories of the category $\oscr_c(W, \hfr)$ of representations of the rational Cherednik algebra $H_c(W, \hfr)$.  This filtration refines the filtration by supports and is analogous to the Harish-Chandra series appearing in the representation theory of finite groups of Lie type.  Using the monodromy of the Bezrukavnikov-Etingof parabolic restriction functors, we show that the subquotients of this filtration are equivalent to categories of finite-dimensional representations over generalized Hecke algebras.  When $W$ is a finite Coxeter group, we give a method for producing explicit presentations of these generalized Hecke algebras in terms of finite-type Iwahori-Hecke algebras.  This yields a method for counting the number of irreducible objects in $\oscr_c(W, \hfr)$ of given support.  We apply these techniques to count the number of irreducible representations in $\oscr_c(W, \hfr)$ of given support for all exceptional Coxeter groups $W$ and all parameters $c$, including the unequal parameter case.  This completes the classification of the finite-dimensional irreducible representations of $\oscr_c(W, \hfr)$ for exceptional Coxeter groups $W$ in many new cases.

\end{abstract}

\tableofcontents

\section{Introduction}  Let $W$ be a finite complex reflection group with reflection representation $\hfr$ and let $c : S \rightarrow \cplx$ be a $W$-invariant function on the set of reflections $S \subset W$.  To this data one can associate the rational Cherednik algebra $H_c(W, \hfr)$, an infinite-dimensional noncommutative $\cplx$-algebra introduced by Etingof and Ginzburg \cite{EG}.  The family of algebras $H_c(W, \hfr)$ parameterized by such class functions $c$ forms a flat deformation of the semidirect product algebra $\cplx W \ltimes D(\hfr)$ of the reflection group $W$ with the algebra $D(\hfr)$ of polynomial differential operators on $\hfr$; as a vector space $H_c(W, \hfr)$ is isomorphic to $S\hfr^* \otimes \cplx W \otimes S\hfr$ in a natural way, and the multiplication in $H_c(W, \hfr)$ depends on $c$ in a polynomial manner.  Ginzburg, Guay, Opdam, and Rouquier \cite{GGOR} defined a certain category $\oscr_c(W, \hfr)$ of representations of $H_c(W, \hfr)$ that shares many properties with the classical BGG category $\oscr$ associated to a complex semisimple Lie algebra $\gfr$.  The category $\oscr_c(W, \hfr)$ is defined as the full subcategory of $H_c(W, \hfr)$-modules that are finitely generated over $S\hfr^*$ and for which the action of any element $y \in \hfr$ is locally finite.  In particular, modules in $\oscr_c(W, \hfr)$ can be viewed as coherent sheaves on $\hfr$ with additional structure, and this geometric perspective is crucial in the study of $\oscr_c(W, \hfr)$.

A key technical tool for studying $H_c(W, \hfr)$ and the category $\oscr_c(W, \hfr)$ is the Knizhnik-Zamolodchikov (KZ) functor also introduced in \cite{GGOR}.   The KZ functor $KZ: \oscr_c(W, \hfr) \rightarrow \mathsf{H}_q(W)\mhyphen\text{mod}_{f.d.}$ is an exact functor, defined via monodromy, from $\oscr_c(W, \hfr)$ to the category $\mathsf{H}_q(W)\mhyphen\text{mod}_{f.d.}$ of finite-dimensional modules over the Hecke algebra $\mathsf{H}_q(W)$ associated to the reflection group $W$.  The parameter $q$ of the Hecke algebra $\mathsf{H}_q(W)$ depends on the parameter $c$ of the rational Cherednik algebra $H_c(W, \hfr)$ in an exponential manner.  $KZ$ induces an equivalence of categories $\oscr_c(W, \hfr)/\oscr_c(W, \hfr)^{tor} \cong \mathsf{H}_q(W)\mhyphen\text{mod}_{f.d}$ \cite{GGOR, Losevfinitedimensional}, where $\oscr_c(W, \hfr)^{tor}$ denotes the Serre subcategory of modules in $\oscr_c(W, \hfr)$ supported on the union of the reflection hyperplanes $\cup_{s \in S}\ker(s)$.  In this way, $KZ$ establishes a bijection between the irreducible representations in $\oscr_c(W, \hfr)$ of full support in $\hfr$ and the finite-dimensional irreducible representations of $\mathsf{H}_q(W)$.  Following previous work of Etingof-Rains \cite{Etingof-Rains}, Marin-Pfeiffer \cite{Marin-Pfeiffer}, Losev \cite{Losevfinitedimensional}, Chavli \cite{Chavli}, and others towards proving the Brou\'{e}-Malle-Rouquier conjecture \cite{BMR}, Etingof \cite{Et-BMR} recently showed that the Hecke algebra $\mathsf{H}_q(W)$ is always finite-dimensional with dimension $\#W$, even in the case of complex reflection groups.

In this paper, in the case that $W$ is a Coxeter group with complexified reflection representation $\hfr$, we extend this correspondence between irreducible representations $L$ in $\oscr_c(W, \hfr)$ and irreducible representations of finite-type Hecke algebras to include all cases in which the support of $L$ is not equal to $\{0\} \subset \hfr$, i.e. all cases in which $L$ is not finite-dimensional over $\cplx$.  Our approach is inspired by the Harish-Chandra series appearing in the representation theory of finite groups of Lie type.  In place of the parabolic induction and restriction functors defined for finite groups of Lie type, in the setting of rational Cherednik algebras one has analogous parabolic induction and restriction functors introduced by Bezrukavnikov and Etingof \cite{BE}.  In particular, suppose $(W', S') \subset (W, S)$ is a parabolic Coxeter subsystem of $(W, S)$ with complexified reflection representation $\hfr_{W'}$.  Restricting the parameter $c$ to $S'$ we may form the rational Cherednik algebra $H_c(W', \hfr_{W'})$ and the associated category of representations $\oscr_c(W', \hfr_{W'})$.  The algebra $H_c(W', \hfr_{W'})$ does not naturally embed as a subalgebra of the algebra $H_c(W, \hfr)$ as is the case for Hecke algebras, so there is no naive definition of parabolic induction and restriction functors as there is for Hecke algebras.  Rather, the Bezrukavinikov-Etingof parabolic induction functor $\ind_{W'}^W : \oscr_c(W', \hfr_{W'}) \rightarrow \oscr_c(W, \hfr)$ and restriction functor $\res_{W'}^W : \oscr_c(W, \hfr) \rightarrow \oscr_c(W', \hfr_{W'})$ are more technical and are defined using the geometric interpretation of rational Cherednik algebras and a certain isomorphism involving completions of $H_c(W, \hfr)$ and $H_c(W', \hfr_{W'})$.  The functors $\res_{W'}^W$ and $\ind_{W'}^W$ are exact \cite{BE} and biadjoint \cite{Losevbiadjoint, Shan}.

As for representations of finite groups of Lie type, one defines the cuspidal representations of a rational Cherednik algebra $H_c(W, \hfr)$ to be those irreducible representations $L$ in $\oscr_c(W, \hfr)$ such that $\res_{W'}^W L = 0$ for all proper parabolic subgroups $W' \subset W$.  For rational Cherednik algebras, the cuspidal representations are precisely the finite-dimensional irreducible representations.  The class of cuspidal representations is the smallest class of representations in the categories $\oscr_c(W, \hfr)$ such that any irreducible representation in a category $\oscr_c(W, \hfr)$ appears as a subobject (or, as a quotient) of a representation induced from a representation in that class.  From this perspective, the study of irreducible representations in $\oscr_c(W, \hfr)$ is largely reduced to the study of cuspidal representations $L$ in categories $\oscr_c(W', \hfr_{W'})$ for $W' \subset W$ a parabolic subgroup and of the structure and decomposition of the induced representations $\ind_{W'}^W L$.

The most basic finite-dimensional irreducible representation of a rational Cherednik algebra is the trivial representation $\cplx$ of the trivial rational Cherednik algebra $H_c(1, \{0\}) = \cplx$.  The induced representation $\ind_1^W \cplx$, denoted $P_{KZ}$, is a remarkable projective object in $\oscr_c(W, \hfr)$.  In particular, there is an algebra isomorphism $\End_{H_c(W, \hfr)}(P_{KZ})^{opp} \cong \mathsf{H}_q(W)$ with respect to which $P_{KZ}$ represents the $KZ$ functor.  The equivalence of categories $\oscr_c(W, \hfr)/\oscr_c(W, \hfr)^{tor} \cong \mathsf{H}_q(W)\mhyphen\text{mod}_{f.d.}$ induced by $KZ$ therefore establishes a correspondence between the irreducible representations in $\oscr_c(W, \hfr)$ of full support in $\hfr$ and the irreducible representations of $\End_{H_c(W, \hfr)}(\ind_1^W \cplx)$.

In this paper, we study the endomorphism algebras $\End_{H_c(W, \hfr)}(\ind_{W'}^W L)$ of induced cuspidal representations $L$ in the general case that $L$ is not necessarily the trivial representation $\cplx$ in $\oscr_c(1, \{0\})$ and provide results analogous to and generalizing those summarized above for the case $L = \cplx$.  In the parallel setting of representations of finite groups of Lie type, endomorphism algebras of induced cuspidal representations have been studied in great detail and are closely related to Hecke algebras of finite type.  For example, in the most basic case one considers the parabolic induction $\ind_B^G \cplx$ of the trivial representation $\cplx$ of the trivial group 1 to the finite general linear group $G := GL_n(\mathbb{F}_q)$, where $q$ is a prime power and $B$ is the standard Borel subgroup of upper triangular matrices.  In that case, the endomorphism algebra is precisely the Hecke algebra $\mathsf{H}_q(S_n)$ where $q$ is the order of the finite field, exactly analogous to the case of $P_{KZ} = \ind_1^W \cplx$ for rational Cherednik algebras.  In the general case, Howlett and Lehrer \cite[Theorem 4.14]{HoLe} showed that, in characteristic 0, the endomorphism algebra of a parabolically induced cuspidal representation of a finite group of Lie type can be described as a semidirect product of a finite type Hecke algebra by a finite group acting by a diagram automorphism, twisted by a certain 2-cocycle (see, for instance, \cite[Theorem 10.8.5]{Car}).  We obtain an exactly analogous result for the endomorphism algebras of induced representations $\ind_{W'}^W L$ of finite-dimensional irreducible representations $L$ of $H_c(W', \hfr_{W'})$.  Howlett and Lehrer conjectured \cite[Conjecture 6.3]{HoLe} that the 2-cocycle appearing in the description of the endomorphism algebra was trivial, as was later proved (see Lusztig \cite[Theorem 8.6]{Lu} and Geck \cite{Geck}).  Appealing to the classification of irreducible finite Coxeter groups and to previous results about finite-dimensional irreducible representations of rational Cherednik algebras, we find that the 2-cocycles appearing in our setting are trivial as well.  The connections between the categories $\oscr_c(W, \hfr)$ and the modular representation theory of finite groups of Lie type is more than an analogy; Norton \cite{Norton} explains various facts and conjectures relating these two subjects, as well as ways in which these connections break down.

The methods used in the setting of finite groups of Lie type ultimately rely on the comparatively simple and explicit definition of parabolic induction, which allows for a basis of intertwining operators to be written down in closed form.  Instead, in our setting of the more complicated Bezrukavnikov-Etingof parabolic induction functors for rational Cherednik algebras, we consider for each finite-dimensional irreducible representation $L$ in $\oscr_c(W', \hfr_{W'})$ the functor $KZ_L$ represented by $\ind_{W'}^W L$, analogous to the $KZ$ functor.  Recall that $KZ$ induces an equivalence of categories $\oscr_c(W, \hfr)/\oscr_c(W, \hfr)^{tor} \cong \mathsf{H}_q(W)\mhyphen\text{mod}_{f.d.}$; the functor $KZ_L$ induces an analogous equivalence.  Specifically, let $\oscr_c(W, \hfr)_{(W', L)}$ be the Serre subcategory of $\oscr_c(W, \hfr)$ generated by those irreducible representations $M$ supported on $W\hfr^{W'}$ as coherent sheaves on $\hfr$ and with $L$ appearing as a composition factor of $\res_{W'}^W M$, and let $\oscr_c(W, \hfr)_{(W', L)}^{tor}$ be the kernel of $\res_{W'}^W$ in $\oscr_c(W, \hfr)_{(W', L)}.$  Then $KZ_L$ induces an equivalence of categories $$\oscr_c(W, \hfr)_{(W', L)}/\oscr_c(W, \hfr)_{(W', L)}^{tor} \cong \End_{H_c(W, \hfr)}(\ind_{W'}^W L)^{opp}\mhyphen\text{mod}_{f.d}.$$  As with the $KZ$ functor, we provide a geometric interpretation of the functor $KZ_L$ in terms of the monodromy of an equivariant local system on $\hfr^{W'}_{reg}$, the open locus of points inside the fixed space $\hfr^{W'}$ with stabilizer precisely equal to $W'$.  Using a transitivity result for local systems of parabolic restriction functors due to Gordon and Martino \cite{GoMa}, we reduce the problem of computing the eigenvalues of the monodromy around the missing hyperplanes, i.e. the parameters of the underlying ``Hecke algebra'' $\End_{H_c(W, \hfr)}(\ind_{W'}^W L)^{opp}$, to the case in which $W'$ is a corank-1 parabolic subgroup of $W$.  In that case, we use a different application of the Gordon-Martino transitivity result and the fact that the usual $KZ$ functor is fully faithful on projective objects to further reduce the problem of computing eigenvalues of monodromy to concrete, although often involved, computations inside the Hecke algebra $\mathsf{H}_q(W)$ itself.

As a corollary of our explicit descriptions of the endomorphism algebras $\End_{H_c(W, \hfr)}(\ind_{W'}^W L)$, we are able to count the number of irreducible representations in $\oscr_c(W, \hfr)$ with any given support variety in $\hfr$.  In particular, as an application we calculate the number of finite-dimensional irreducible representations of $H_c(W, \hfr)$ for all exceptional Coxeter groups $W$ and all parameters $c$.

This paper is organized as follows.  In Section \ref{background-section}, we recall the necessary background and definitions needed in later sections, emphasizing the local systems of Bezrukavnikov-Etingof parabolic induction and restriction functors and an important transitivity result of Gordon-Martino \cite{GoMa} for these local systems.  In Section \ref{generalities-section}, we introduce the filtration on $\oscr_c(W, \hfr)$ that we study in this paper.  We construct the functor $KZ_L$ using the monodromy of the Bezrukavnikov-Etingof parabolic restriction functors, providing a description of the algebra $\End_{H_c(W, \hfr)}(\ind_{W'}^W L)^{opp}$ as a quotient of a twisted group algebra of a fundamental group.  We use the Gordon-Martino transitivity result to show that the generator of monodromy around a deleted hyperplane satisfies a polynomial relation of degree equal to the order of the stabilizer of that hyperplane in the inertia group of $L$ (see Definition \ref{inertia-definition}).  In Section \ref{coxeter-section}, we specialize to the case in which $W$ is a Coxeter group.  In this case, we provide a presentation of the endomorphism algebra as a generalized Hecke algebra directly analogous to the presentation by Howlett and Lehrer \cite[Theorem 4.14]{HoLe}.  Using the $KZ$ functor, we provide a method for computing the parameters of these generalized Hecke algebras.  We apply our methods systematically, using the classification of finite Coxeter groups and previously known results about finite-dimensional representations, to count the number of irreducible representations in $\oscr_c(W, \hfr)$ of given support for all exceptional Coxeter groups $W$ and parameters $c$.  We describe the generalized Hecke algebras appearing in types $E$ and $H$ explicitly.

The results we obtain for finite Coxeter groups confirm, unify, and extend many previously known results in both classical and exceptional types.  In type A, we recover Wilcox's description \cite[Theorem 1.2]{Wilcox} of the subquotients of the filtration by supports of the categories $\oscr_c(S_n, \hfr)$.  In type $B$, we show that the subquotients of our refined filtration are equivalent to categories of finite-dimensional modules over tensor products of Iwahori-Hecke algebras of type $B$, and we obtain similar descriptions in type $D$.  These facts follows from results of Shan-Vasserot \cite{SV}, although their results do not generalize to the exceptional types.

In the exceptional types, however, our methods provide many new results.  Among the exceptional Coxeter types, complete knowledge of character formulas for all irreducible representations in $\oscr_c(W, \hfr)$ for all parameters $c$ is only known for the dihedral types, treated by Chmutova \cite{Chm}, and type $H_3$, treated by Balagovic-Puranik \cite{BP}.  For parameters $c = 1/d$, where $d > 2$ is an integer dividing a fundamental degree of the Coxeter group $W$, Norton \cite{Norton} computed the decomposition matrices for $\oscr_c(W, \hfr)$, thereby classifying the finite-dimensional irreducible representations and determining character formulas and supports for all irreducible representations, when $W$ is a Coxeter group of type $E_6, E_7, H_4$, or $F_4$.  After this paper was finished, we learned that Norton, in a sequel to \cite{Norton} to be announced, has produced complete decomposition matrices in type $E_8$ when the denominator $d$ of $c$ is not $2,3,4$ or 6, decomposition matrices for several blocks of $\oscr_c(E_8, \hfr)$ when the denominator $d$ equals 4 or 6, as well as decomposition matrices in some unequal parameter cases in type $F_4$.  However, to our knowledge, the classification of the finite-dimensional representations of rational Cherednik algebras remains unknown in types $E_6, E_7, H_4$ and $F_4$ at half-integer parameters, in type $F_4$ in most unequal parameter cases, and in type $E_8$ when the denominator $d$ of $c$ is 2,3,4, or 6.

Recently, Griffeth-Gusenbauer-Juteau-Lanini \cite{GGJL} produced a necessary condition for an irreducible representation $L_c(\lambda)$ in $\oscr_c(W, \hfr)$ to be finite-dimensional that applies to all complex reflection groups $W$ and parameters $c$.  Our results show that this condition is also sufficient in many of the previously unknown cases mentioned above, providing complete classifications of the irreducible finite-dimensional representations of $H_c(W, \hfr)$ in these cases.  In combination with previous results of Balagovic-Puranik \cite{BP}, Norton \cite{Norton}, in this way we complete the classification of finite-dimensional irreducible representations of $H_c(W, \hfr)$ in the case that $W$ is an exceptional Coxeter group and $c = 1/d$, where $d$ is a positive integer possibly equal to 2, except in the case in which both $W = E_8$ and also $d = 3$ simultaneously.  In the case $W = E_8$ and $c = 1/3$, we show that there are 8 non-isomorphic finite-dimensional irreducible representations of $H_{1/3}(E_8, \hfr)$, and results of \cite{GGJL} rule out all but 9 of the 112 irreducible representations of the Coxeter group $E_8$ as the potential lowest weights of these representations.  In particular, determining which one of these 9 irreducible representations is infinite-dimensional will complete the classification in type $E$ entirely.  This analysis in types $E, H$, and $F$ is carried out Sections \ref{typeE}, \ref{typeH}, and \ref{typeF4}, respectively.

In Sections \ref{typeE} and \ref{typeH}, we describe the generalized Hecke algebras arising from exceptional Coxeter groups $W$ of types $E_6, E_7, E_8, H_3$ and $H_4$ at all parameters $c$ of the form $c = 1/d$, where $d > 1$ is an integer dividing a fundamental degree of $W$.  In Section \ref{typeF4}, we count the number of irreducible representations in $\oscr_{c_1, c_2}(F_4, \hfr)$ of given support in $\hfr$ for all parameter values $c_1, c_2$, including the unequal parameter case.  Along with previous results of Chmutova \cite{Chm} for dihedral Coxeter groups, this completes the counting of irreducible representations in $\oscr_c(W, \hfr)$ of given support for all exceptional Coxeter groups $W$ and all, possibly unequal, parameters $c$.

These results for parameters $c = 1/d$ can be extended to parameters $c = r/d$ with $r > 0$ a positive integer relatively prime to $d$; the case $r < 0$ is then obtained by tensoring with the sign character.  The reduction from the $c = r/d$ case to the $c = 1/d$ case can be achieved by results of Rouquier \cite{Ro}, Opdam \cite{Opdam}, and Gordon-Griffeth \cite{GG} when $d \geq 3$.  Specifically, a result of Rouquier, as it appears in \cite[Theorem 2.7]{GG} after appropriate modifications, implies that for any integer $d \geq 3$ and integer $r > 1$ relatively prime to $d$ there is a bijection $\varphi$ on the set $\irr(W)$ of the irreducible complex representations of the Coxeter group $W$ such that for all $\lambda \in \irr(W)$ the irreducible representations $L_{1/d}(\lambda)$ and $L_{r/d}(\phi(\lambda))$ have the same support in $\hfr$.  In \cite{Opdam} these bijections are calculated, and in \cite[Section 2.16]{GG} a method for computing the bijections $\phi$ is given which applies in greater generality.  This allows the reduction of the classification of irreducible representations of given support in $\oscr_{r/d}(W, \hfr)$ to the classification of irreducible representations of the same given support in $\oscr_{1/d}(W, \hfr)$.  In particular, as the finite-dimensional representations are those with support $\{0\} \subset \hfr$, this provides the same reduction for classifying finite-dimensional irreducible representations.  In type $H_3$, this analysis has been carried out in this way by Balagovic-Puranik \cite[Section 5.4]{BP}, and this approach generalizes to the other Coxeter types.  When $d = 2$, our results show that the necessary condition for finite-dimensionality appearing in \cite{GGJL} is in fact also sufficient for the exceptional types $E, H$, and $F$, completing the classification of finite-dimensional irreducible representations for all half-integer parameters in these cases without the use of such bijections on the labels of the irreducibles.

$\mathbf{Acknowledgements.}$  We would like to thank Pavel Etingof, Rapha\"el Rouquier, and Jos\'{e} Simental for many useful conversations.  We also thank Emily Norton for providing comments on a preliminary version of this paper.  The work of the first author was partially supported by the NSF under grants DMS-1161584 and DMS-1501558.

\section{Background and Definitions} \label{background-section}

In this section we will recall rational Cherednik algebras and related constructions and results.

\subsection{Rational Cherednik Algebras}  Let $\hfr$ be a finite-dimensional complex vector space.  A \emph{complex reflection} is an invertible linear operator $s \in GL(\hfr)$ of finite order such that $\text{rank}(1 - s) = 1$.  Let $W \subset GL(\hfr)$ be a \emph{complex reflection group}, i.e. a finite subgroup of $GL(\hfr)$ generated by the subset $S \subset W$ of complex reflections lying in $W$.  The complex reflection groups were classified by Shephard and Todd in 1954 \cite{ST}; each such group decomposes as a product of irreducible complex reflection groups, and every irreducible complex reflection group either appears in the infinite family of complex reflection groups $G(m, p, n)$ indexed by integers $m, p, n \geq 1$ with $p \mid m$ or is one of the 34 exceptional irreducible complex reflection groups.  Important special cases of complex reflection groups include the finite real reflection groups, i.e. the finite Coxeter groups, which may be regarded as complex reflection groups via complexification of the reflection representation.  For example, the symmetric groups $S_n$ are the complex reflection groups $G(1, 1, n)$.

Let $c : S \rightarrow \cplx$ be a class function on the reflections $S \subset W$.  Also, for each reflection $s \in S$ let $\alpha_s \in \hfr^*$ be an eigenvector for $s$ in $\hfr^*$ with nontrivial eigenvalue $\lambda_s \neq 1$ and let $\alpha_s^\vee \in \hfr$ be an eigenvector for $s$ in $\hfr$ with eigenvalue $\lambda_s^{-1}$.  The vectors $\alpha_s$ and $\alpha_s^\vee$ are determined up to $\cplx^\times$ scaling, and we partially normalize them so that $\langle\alpha_s, \alpha_s^\vee \rangle = 2$, where $\langle \cdot, \cdot \rangle$ denotes the natural pairing of $\hfr^*$ with $\hfr$.  The \emph{rational Cherednik algebra} $H_c(W, \hfr)$ attached to $(W, \hfr)$ with parameter $c$ is given by generators and relations as follows:  $$H_c(W, \hfr) := \frac{\cplx W \ltimes T(\hfr^* \oplus \hfr)}{[x, x'] = [y, y'] = 0, \ \ [y, x] = \langle y, x\rangle - \sum_{s \in S} c_s \langle \alpha_s, y\rangle \langle x, \alpha_s^\vee \rangle s \ \ \forall x \in \hfr^*, y \in \hfr}.$$

Observe that for $c = 0$ we have $H_c(W, \hfr) = \cplx W \ltimes D(\hfr)$, where $D(\hfr)$ is the algebra of polynomial differential operators on $\hfr$.

\subsection{PBW Theorem, Category $\oscr_c$, and Support Varieties}  From the above presentation it is clear that $H_c(W, \hfr)$ admits a filtration with $\deg W = \deg \hfr = 0$ and $\deg \hfr^* = 1$, analogous to the usual filtration on $D(\hfr)$ by the order of differential operators.  Etingof and Ginzburg \cite[Theorem 1.3]{EG} showed that the associated graded algebra of $H_c(W, \hfr)$ with respect to this filtration is naturally identified with $\cplx W \ltimes S(\hfr^* \oplus \hfr)$, and in particular the natural multiplication map $S\hfr^* \otimes \cplx W \otimes S\hfr \rightarrow H_c(W, \hfr)$ is an isomorphism of $\cplx$-vector spaces.  This isomorphism should be viewed as an analogue for rational Cherednik algebras of the triangular decomposition $U(\gfr)  = U(\nfr_-) \otimes U(\hfr) \otimes U(\nfr_+)$ of a complex semisimple Lie algebra $\gfr$ with respect to a choice of Borel subalgebra $\bfr = \hfr \oplus \nfr_+$.  In particular, parallel to the setting of semisimple Lie algebras and following \cite{GGOR}, from this triangular decomposition one defines the category $\oscr_c = \oscr_c(W, \hfr)$ of representations of $H_c(W, \hfr)$ to be the full subcategory of $H_c(W, \hfr)$-modules that are finitely generated over $S\hfr^*$ and on which elements $y \in \hfr$ act locally nilpotently.

The category $\oscr_c(W, \hfr)$ is similar in structure to the classical BGG categories $\oscr$ attached to semisimple Lie algebras $\gfr$.  To any representation $\lambda$ of $\cplx W$ one defines the \emph{standard} or \emph{Verma} module $\Delta_c(\lambda) := H_c(W, \hfr) \otimes_{\cplx W \ltimes S\hfr} \lambda$ where the action of $\cplx W \ltimes S\hfr$ on $\lambda$ is obtained by extending the action of $\cplx W$ so that $\hfr$ acts by $0$.  When $\lambda$ is irreducible, the Verma module $\Delta_c(\lambda)$ has a unique irreducible quotient $L_c(\lambda)$, and the correspondence $\lambda \mapsto L_c(\lambda)$ establishes a bijection $\irr(W) \rightarrow \irr(\oscr_c(W, \hfr))$ between the irreducible representations of $W$ and the irreducible representations of $H_c(W, \hfr)$ in $\oscr_c(W, \hfr$) \cite{GGOR}.

By restriction, any module $M$ in $\oscr_c(W, \hfr)$ can be viewed as a finitely generated $S\hfr^*$-module, i.e. as a coherent sheaf on $\hfr$.  In particular, one defines the support of $M$ to be its support as a coherent sheaf on $\hfr$.  The determination of the support varieties $\supp (L_c(\lambda))$ of the irreducible representations $L_c(\lambda)$ is a fundamental question about $\oscr_c(W, \hfr)$.  Using parabolic induction and restriction functors, Bezrukavnikov and Etingof \cite{BE} showed that the subvarieties of $\hfr$ appearing as supports of irreducible representations in $\oscr_c(W, \hfr)$ are of the form $X(W') := W\hfr^{W'}$ where $W'$ is a parabolic subgroup of $W$, i.e. the stabilizer of a point $b \in \hfr$.  Note that $X(W') = X(W'')$ exactly when $W'$ and $W''$ are conjugate in $W$, and in particular the possible supports of simple modules in $\oscr_c(W, \hfr)$ are indexed by conjugacy classes of parabolic subgroups of $W$.  It is important to note that not necessarily all such varietes $X(W')$ appear as the support varieties of representations in $\oscr_c(W, \hfr)$, and in fact for generic values of the parameter $c$ all irreducible representations in $\oscr_c(W, \hfr)$ have full support, in which case only the variety $X(1) = \hfr$ appears.

\subsection{The Localization Lemma, Hecke Algebras, and the KZ Functor} \label{KZ} Let $\hfr_{reg} := \hfr \backslash \bigcup_{s \in S} \ker(\alpha_s)$, the \emph{regular locus} for the action of $W$ on $\hfr$, be the complement in $\hfr$ of the arrangement of reflection hyperplanes for the action of $W$ on $\hfr$.  Note that $\hfr_{reg}$ is precisely the principal affine open subset of $\hfr$ defined by the nonvanishing of the discriminant element $\delta := \prod_{s \in S} \alpha_s$.  As $\delta^{|W|}$ is $W$-invariant and has locally nilpotent adjoint action on $H_c(W, \hfr)$, it follows that the noncommutative localization $H_c(W, \hfr)[\delta^{-1}]$ coincides with the localization of $H_c(W, \hfr)$ as a $\cplx[\hfr]$-module, and this localized algebra is isomorphic to the algebra $\cplx W\ltimes D(\hfr_{reg})$ in a natural way (see \cite[Theorem 5.6]{GGOR}).  Thus, a module $M$ in $\oscr_c(W, \hfr)$ determines a module $M[\delta^{-1}]$ over $\cplx W \ltimes D(\hfr_{reg})$ by localization, which may be regarded as a $W$-equivariant $D$-module on $\hfr_{reg}$.  The $W$-equivariant $D(\hfr_{reg})$-modules occurring in this way are $\oscr(\hfr_{reg})$-coherent, by definition of the category $\oscr_c(W, \hfr)$, and have regular singularities \cite{GGOR}.  By the Riemann-Hilbert correspondence \cite{De}, descending to $\hfr_{reg}/W$ and taking the monodromy representation at a base point $b \in \hfr_{reg}$ defines an equivalence of categories $$\cplx W \ltimes D(\hfr_{reg})\mhyphen\text{mod}_{\oscr(\hfr_{reg})\mhyphen\text{coh, r.s.}} \cong \pi_1(\hfr_{reg}/W, b)\mhyphen\text{mod}_{f.d.}$$ where the subscript $\oscr(\hfr_{reg})\mhyphen\text{coh, r.s.}$ indicates that the $D(\hfr_{reg})$-modules are $\oscr(\hfr_{reg})$-coherent with regular singularities and where the subscript $f.d.$ indicates that $\pi_1(\hfr_{reg}/W, b)$-modules are finite-dimensional.  This procedure of localization to $\hfr_{reg}$ followed by descent to $\hfr_{reg}$ and the Riemann-Hilbert correspondence thus defines an exact functor $$\oscr_c(W, \hfr) \rightarrow \pi_1(\hfr_{reg}/W, b)\mhyphen\text{mod}_{f.d.}.$$

The fundamental group $\pi_1(\hfr_{reg}/W, b)$ is the \emph{generalized braid group} attached to $W$ and is denoted $B_W$.  In \cite{GGOR}, it was shown that the above functor in fact factors through the category $\mathsf{H}_q(W)\mhyphen\text{mod}_{f.d.}$ of finite-dimensional representations over a certain quotient $\mathsf{H}_q(W)$, the \emph{Hecke algebra}, of the group algebra $\cplx B_W$.   The resulting exact functor $$KZ : \oscr_c(W, \hfr) \rightarrow \mathsf{H}_q(W)\mhyphen\text{mod}_{f.d.}$$ is the $KZ$ functor.  The quotient $\mathsf{H}_q(W)$ is defined as follows \cite{BMR}.  Let $\mathcal{H} = \{\ker(s) : s \in S\}$ be the set of reflection hyperplanes for the action of $W$ on $\hfr$.  For each reflection hyperplane $H \in \mathcal{H}$, let $T_H$ be a representative of the conjugacy class in $B_W$ determined by a small loop, oriented counterclockwise, in $\hfr_{reg}/W$ about the hyperplane $H$.  The group $B_W$ is generated by the union of these conjugacy classes (\cite[Theorem 2.17]{BMR}).  For each $H \in \mathcal{H}$, let $l_H$ be the order of the cyclic subgroup of $W$ stabilizing $H$ pointwise, and let $q_{1, H}, ... q_{l_H - 1, H} \in \cplx^\times$ be nonzero complex numbers that are $W$-invariant, i.e. $q_{j, H} = q_{j, H'}$ whenever $H = wH'$ for some $w \in W$.  Let $q$ denote the collection of these $q_{j, H}$, and define the Hecke algebra $\mathsf{H}_q(W)$ as the quotient of $\cplx B_W$ by the relations $$(T_H - 1)\prod_{j = 1}^{l_H - 1}(T_H - e^{2 \pi i j/l_H}q_{j, H}) = 0.$$
The choice of the parameter $q$ such that the functor $KZ$ is defined as above is given explicitly in \cite{GGOR} as a function of the parameter $c$ for the algebra $H_c(W, \hfr)$.  In the special case that $W$ is a real reflection group, which is the only case relevant to this paper, the dependence of $q$ on $c$ is especially simple.  In that case we have $l_H = 2$ for all $H \in \mathcal{H}$, and $q_{1, H} = e^{-2 \pi i c_s}$ where $s \in S$ is the reflection such that $H = \ker(s)$ (note that the sign convention for $c$ in this paper differs from that appearing in \cite{GGOR}).  In particular, in the Coxeter case the Hecke algebra $\mathsf{H}_q(W)$ appearing in the $KZ$ functor is precisely the Iwahori-Hecke algebra attached to the Coxeter group $W$ with generators $\{T_s : s \in S\}$ satisfying the relations in the braid group $B_W$ and the quadratic relations $$(T_s - 1)(T_s + e^{-2 \pi i c_s}) = 0.$$
 
\subsection{Bezrukavnikov-Etingof Parabolic Induction and Restriction Functors}  Parabolic induction and restriction functors are central to the study of representations of Iwahori-Hecke algebras.  The definition of these functors relies on the natural embedding $\mathsf{H}_q(W') \subset \mathsf{H}_q(W)$ of the Hecke algebra $\mathsf{H}_q(W')$ attached to a parabolic subgroup $W'$ of a Coxeter group $W$ in the Hecke algebra $\mathsf{H}_q(W)$ attached to $W$.  Parabolic restriction is then simply the naive restriction, and parabolic induction is the tensor product $\mathsf{H}_q(W) \otimes_{\mathsf{H}_q(W')} \bullet.$  As for Hecke algebras, parabolic induction and restriction functors are central to the study of representations of rational Cherednik algebras $H_c(W, \hfr)$.  Let $\hfr_{W'}$ denote the unique $W'$-stable complement to $\hfr^{W'}$ in $\hfr$.  Then $W'$ acts on $\hfr_{W'}$, and the action is generated by reflections.  Let $H_c(W', \hfr_{W'})$ denote the associated rational Cherednik algebra, where by abuse of notation $c$ here denotes the restriction of $c$ to the reflections $S' = S \cap W'$ in $W$ lying in $W'$.  Then, unlike for Hecke algebras, the algebra $H_c(W', \hfr_{W'})$ does not embed as a subalgebra of $H_c(W, \hfr)$ in a natural way, and the restriction and induction functors must be defined differently.

Bezrukavnikov and Etingof \cite{BE} circumvented this difficulty by using the geometric interpretation of rational Cherednik algebras.  The price to pay for this approach is that rather than defining a single parabolic restriction functor $\res_{W'}^W : \oscr_c(W, \hfr) \rightarrow \oscr_c(W', \hfr_{W'})$ and a single parabolic induction functor $\ind_{W'}^W : \oscr_c(W', \hfr_{W'}) \rightarrow \oscr_c(W, \hfr)$, one instead defines a restriction functor $\res_b$ and induction functor $\ind_b$ defined for every point $b$ in the locally closed subvariety $\hfr^{W'}_{reg}$ of $\hfr$ consisting of those points $b \in \hfr$ with stabilizer $\stab_W(b) = W'$.  Bezrukavnikov and Etingof explain the dependence of $\res_b$ on the choice of $b \in \hfr^{W'}_{reg}$ in the following elegant manner.  They construct an exact functor $$\underline{\res}_{W'}^W : \oscr_c(W, \hfr) \rightarrow (\oscr_c(W', \hfr_{W'}) \boxtimes \loc(\hfr^{W'}_{reg}))^{N_W(W')},$$ where $\loc(\hfr^{W'}_{reg})$ denotes the category of local systems on $\hfr^{W'}_{reg}$ and the $N_W(W')$ in the superscript indicates equivariance for the normalizer $N_W(W')$ of $W'$ in $W$, and they show that the fiber $(\underline{\res}_{W'}^W)_b$ of the resulting local system at $b \in \hfr^{W'}_{reg}$ gives precisely the parabolic restriction functor $\res_b$.  In particular, for any points $b, b' \in \hfr^{W'}_{reg}$, parallel transport along any path $\gamma$ in $\hfr^{W'}_{reg}$ connecting $b$ and $b'$ defines an isomorphism between the functors $\res_b$ and $\res_{b'}$.  As explained in \cite{BE}, in the case $W' = 1$, the functor $\underline{\res}_{W'}^W$ can be identified with the $KZ$ functor recalled above in Section \ref{KZ}, and therefore the functors $\underline{\res}_{W'}^W$ can be regarded as relative versions of the $KZ$ functor.  We refer to the functors $\underline{\res}_{W'}^W$ as the \emph{partial KZ functors}, and these functors are central to the approach we take here.

We recall the construction of the functors $\res_b, \ind_b,$ and $\underline{\res}_{W'}^W$ below, following \cite[Section 3]{BE}.

\subsubsection{Construction of $\res_b$ and $\ind_b$}  Let $b \in \hfr^{W'}_{reg}$, i.e. a point in $\hfr$ with stabilizer $W'$ in $W$.  The completion $\widehat{H}_c(W, \hfr)_b$ of the rational Cherednik algebra $H_c(W, \hfr)$ at the orbit $Wb \subset \hfr$ is defined to be the associative algebra $$\widehat{H}_c(W, \hfr)_b := \cplx[\hfr]^{\wedge_{Wb}} \otimes_{\cplx[\hfr]} H_c(W, \hfr),$$ where $\cplx[\hfr]^{\wedge_{Wb}}$ denotes the completion of $\cplx[\hfr]$ at the orbit $Wb$.  Restricting the parameter $c : S \rightarrow \cplx$ to the set of reflections $S' \subset W'$, one may form the rational Cherednik algebra $H_c(W', \hfr)$ and its completion $\widehat{H}_c(W', \hfr)_0$ at $0 \in \hfr$.  As explained in \cite[Section 3.3]{BE}, one may think of the algebra $\widehat{H}_c(W, \hfr)_b$ in a more geometric manner as the subalgebra of $\End_\cplx(\cplx[\hfr]^{\wedge_{Wb}})$ generated by $\cplx[\hfr]^{\wedge_{Wb}}$, the group $W$, and the Dunkl operators $D_y$ associated to points $y \in \hfr$.  This geometric interpretation leads naturally to an isomorphism (\cite[Theorem 3.2]{BE}) $$\theta : \widehat{H}_c(W, \hfr)_b \rightarrow Z(W, W', \widehat{H}_c(W', \hfr)_0)$$ where the algebra $Z(W, W', \widehat{H}_c(W', \hfr)_0)$, the \emph{centralizer algebra}, is the endomorphism algebra of the right $\widehat{H}_c(W, \hfr)_0$-module $Fun_{W'}(W, \widehat{H}_c(W', \hfr)_0)$ of functions $f : W \rightarrow \widehat{H}_c(W', \hfr)_0$ satisfying $f(w'w) = w'f(w)$ for all $w \in W, w' \in W'$.  We may take the completion $\widehat{H}_c(W', \hfr)_0$ of $H_c(W', \hfr)$ at $0$ rather than at $b$ in the centralizer because the assignments $w' \mapsto w'$ for $w' \in W'$, $y \mapsto y$ for $y \in \hfr$, and $x \mapsto x - b$ extends to an isomorphism $\widehat{H}_c(W', \hfr)_0 \cong \widehat{H}_c(W', \hfr)_{b}$.  The isomorphism $\theta$ determines, by transfer of structure, an equivalence of categories $$\theta_* : \widehat{H}_c(W, \hfr)_b\mhyphen\text{mod} \rightarrow Z(W, W', \widehat{H}_c(W', \hfr)_0)\mhyphen\text{mod}.$$

The centralizer algebra $Z(W, W', \widehat{H}_c(W', \hfr)_0)$ is non-canonically isomorphic to the matrix algebra of size $|W/W'|$ over $\widehat{H}_c(W', \hfr)_0$, and in particular the isomorphism $\theta$ shows that $\widehat{H}_c(W, \hfr)_b$ is Morita equivalent to $\widehat{H}_c(W', \hfr)_0$.  A particularly natural choice of Morita equivalence is given by the idempotent $e_{W'} \in Z(W, W', \widehat{H}_c(W', \hfr)_0)$ defined by $e_{W'}(f)(w) = f(w)$ for $w \in W'$ and $e_{W'}(f)(w) = 0$ for $w \notin W'$.  Note that $e_{W'}$ is the image under $\theta_*$ of the idempotent $1_b \in \cplx[\hfr]^{\wedge_{Wb}} \subset \widehat{H}_c(W, \hfr)_b$ that is the indicator function of the point $b$ in its $W$-orbit.  As the two-sided ideal in the centralizer algebra generated by $e_{W'}$ is the entire centralizer algebra and as the associated spherical subalgebra $e_{W'}Z(W, W', \widehat{H}_c(W', \hfr)_0)e_{W'}$ is naturally identified with $\widehat{H}_c(W', \hfr)_0$, the functors $$I : \widehat{H}_c(W', \hfr)_0\mhyphen\text{mod} \rightarrow Z(W, W', \widehat{H}_c(W', \hfr)_0)\mhyphen\text{mod}$$ $$M \mapsto I(M) := Z(W, W', \widehat{H}_c(W', \hfr)_0)e_{W'} \otimes_{e_{W'}Z(W, W', \widehat{H}_c(W', \hfr)_0)e_{W'}} M$$ and $$I^{-1} : Z(W, W', \widehat{H}_c(W', \hfr)_0)\mhyphen\text{mod} \rightarrow \widehat{H}_c(W', \hfr)_0\mhyphen\text{mod}$$ $$N \mapsto I^{-1}(N) := e_{W'}N$$ are mutually quasi-inverse equivalences.

Let $\widehat{\oscr}_c(W, \hfr)_b$ denote the full subcategory of $\widehat{H}_c(W, \hfr)_b$-modules which are finitely generated over $\cplx[\hfr]^{\wedge_{Wb}}$.  Let $$\widehat{\ \ }_b : \oscr_c(W, \hfr) \rightarrow \widehat{\oscr}_c(W, \hfr)_b$$ be the functor of completion at the orbit $Wb$ and let $$E^b : \widehat{\oscr_c}(W, \hfr)_b \rightarrow \oscr_c(W, \hfr)$$ be the functor that sends a module $M$ to its subspace $E^b(M)$ of $\hfr$-locally nilpotent vectors.  By \cite[Proposition 3.6, Proposition 3.8]{BE}, these functors are exact and $E^b$ is the right adjoint of $\widehat{\ \ }_b$.  Similarly, we have the category $\widehat{\oscr}_c(W', \hfr)_0$, the completion functor $\widehat{\ \ }_0 : \oscr_c(W', \hfr) \rightarrow \widehat{\oscr}_c(W', \hfr)_0$ and the functor $E^0 : \widehat{\oscr}_c(W', \hfr)_0 \rightarrow \oscr_c(W', \hfr)$ taking $\hfr$-locally nilpotent vectors, which are in fact category equivalences \cite[Theorem 2.3]{BE}.

Let $\hfr_{W'}$ denote the unique $W'$-stable complement to $\hfr^{W'}$ in $\hfr$.  Clearly, an element $s \in W'$ acts on $\hfr$ as a complex reflection if and only if it acts on $\hfr_{W'}$ as a complex reflection, and $W'$ acts on $\hfr_{W'}$ faithfully.  By restriction of the parameter $c : S \rightarrow \cplx$ to the set of reflections $S'$ in $W'$, we may form the rational Cherednik algebra $H_c(W', \hfr_{W'})$ and its associated category $\oscr_c(W', \hfr_{W'})$.  Let $$\zeta : \oscr_c(W', \hfr) \rightarrow \oscr_c(W', \hfr_{W'})$$ be the functor defined by $$\zeta(M) = \{m \in M : ym = 0 \ \forall y \in \hfr^{W'}\}.$$  As discussed in \cite[Section 2.3]{BE}, $\zeta$ is an equivalence of categories - this is essentially an instance of Kashiwara's lemma for $D(\hfr^{W'})$-modules, in view of the natural tensor product decomposition $H_c(W', \hfr) \cong H_c(W', \hfr_{W'}) \otimes D(\hfr^{W'})$.

\begin{definition} The parabolic restriction functor $\res_b : \oscr_c(W, \hfr) \rightarrow \oscr_c(W', \hfr_{W'})$ is the composition $$\res_b := \zeta \circ E^0 \circ I^{-1} \circ \theta_* \circ \widehat{\ \ }_b$$ and the parabolic induction functor $\ind_b : \oscr_c(W', \hfr_{W'}) \rightarrow \oscr_c(W, \hfr)$ is the composition $$E^b \circ \theta_*^{-1} \circ I \circ \widehat{\ \ }_0 \circ \zeta^{-1}.$$\end{definition}

\begin{proposition}  The functors $\ind_b$ and $\res_b$ are exact and biadjoint and do not depend on the choice of $b \in \hfr^{W'}_{reg}$ up to isomorphism.  \end{proposition}

\begin{proof} Exactness is \cite[Proposition 3.9(i)]{BE}, and that $\ind_b$ is right adjoint to $\res_b$ is \cite[Theorem 3.10]{BE}.  That $\ind_b$ is also left adjoint to $\res_b$ was established by Shan \cite[Section 2.4]{Shan} under some assumptions and later in full generality by Losev \cite{Losevbiadjoint}.  The independence up to isomorphism on the choice of $b$ was established in \cite[Section 3.7]{BE} using the partial $KZ$ functor $\underline{\res}_{W'}^W$ to be recalled in the following section.\end{proof}

\subsection{Partial KZ Functors} \label{partial-KZ} In this section we give in some detail a construction of the \emph{partial KZ} functor $\underline{\res}_{W'}^W$ introduced in \cite[Section 3.7]{BE}.  Let $\widehat{H}_c(W, \hfr)_{W\hfr^{W'}_{reg}}$ denote the completion of $H_c(W, \hfr)$ at the $W$-stable locally closed subvariety $W\hfr^{W'}_{reg} \subset \hfr$.  As for the case of completion at the $W$-orbit of a point, the completion $\widehat{H}_c(W, \hfr)_{W\hfr^{W'}_{reg}}$ can be realized as the subalgebra of $\End_\cplx(\cplx[\hfr]^{\wedge_{W\hfr^{W'}_{reg}}})$ generated by the functions $\cplx[\hfr]^{\wedge_{W\hfr^{W'}_{reg}}}$ on the formal neighborhood of $W\hfr^{W'}_{reg}$ in $\hfr$, the group $W$, and the Dunkl operators $D_y$ for $y \in \hfr$ associated to the action of $W$ on $\hfr$.  (Recall that the algebra of functions $\cplx[\hfr]^{\wedge_{W\hfr^{W'}_{reg}}}$ is obtained by first localizing to a principal open subset in which $W\hfr^{W'}_{reg}$ is a closed subvariety, followed by completion at the ideal defining $W\hfr^{W'}_{reg}$ in this principal open subset.)  Note that the variety $W\hfr^{W'}_{reg}$ is the disjoint union of the $W$-translates of $\hfr^{W'}_{reg}$, and the set-wise stabilizer of $\hfr^{W'}_{reg}$ in $W$ is precisely $N_W(W')$.  Let $1_{\hfr^{W'}_{reg}} \in \cplx[\hfr]^{\wedge_{W\hfr^{W'}_{reg}}}$ denote the indicator function of $\hfr^{W'}_{reg}$ as a function on the formal neighborhood in $\hfr$ of its $W$-orbit.  Similar to the isomorphism $$1_b\widehat{H}_c(W, \hfr)_b1_b \cong \widehat{H}_c(W', \hfr)_b$$ from \cite{BE} discussed in the previous section, we have a natural isomorphism of algebras $$1_{\hfr^{W'}_{reg}}\widehat{H}_c(W, \hfr)_{W\hfr^{W'}_{reg}}1_{\hfr^{W'}_{reg}} \cong \widehat{H}_c(N_W(W'), \hfr)_{\hfr^{W'}_{reg}},$$ where the algebra $\widehat{H}_c(N_W(W'), \hfr)_{\hfr^{W'}_{reg}}$ denotes the completion of the algebra $H_c(N_W(W'), \hfr) = \cplx N_W(W') \ltimes_{W'} H_c(W', \hfr)$ at the $N_W(W')$-stable locally closed subvariety $\hfr^{W'}_{reg} \subset \hfr$ (for comparison with the case of \'{e}tale pullbacks rather than completions, see \cite[Theorem 3.2]{GGJL}).  The formal neighborhood of $\hfr^{W'}_{reg}$ inside $\hfr$ is canonically identified with its formal neighborhood inside the principal open subset $\hfr^{W'}_{reg} \times \hfr_{W'} \subset \hfr$, which is precisely $(\hfr^{W'}_{reg} \times \hfr_{W'})^{\wedge_{\hfr^{W'}_{reg}}} = \hfr^{W'}_{reg} \widehat{\times} \hfr_{W'}^{\wedge_0}$ and has ring of formal functions $\cplx[\hfr^{W'}] \widehat{\otimes} \cplx[[\hfr_{W'}]].$  From this it follows immediately that there is a natural isomorphism $$\widehat{H}_c(N_W(W'), \hfr)_{\hfr^{W'}_{reg}} \cong N_W(W') \ltimes_{W'} (\widehat{H}_c(W', \hfr_{W'})_0 \widehat{\otimes} D(\hfr^{W'}_{reg})).$$  Let $$\psi : 1_{\hfr^{W'}_{reg}}\widehat{H}_c(W, \hfr)_{W\hfr^{W'}_{reg}}1_{\hfr^{W'}_{reg}} \rightarrow N_W(W') \ltimes_{W'} (\widehat{H}_c(W', \hfr_{W'})_0 \widehat{\otimes} D(\hfr^{W'}_{reg}))$$ be the isomorphism obtained as the composition of the two natural isomorphisms discussed above, and let $\psi_*$ denote the induced equivalence of module categories.

Let $\eu_{W'} \in H_c(W', \hfr_{W'})$ denote the \emph{Euler element} $$\eu_{W'} := \sum_{i = 1}^n x_iy_i + \frac{n}{2} - \sum_{s \in S'} \frac{2c_s}{1 - \lambda_s}s$$ where $x_1, ..., x_n \in \hfr_{W'}^*$ is any basis of $\hfr_{W'}^*$ and $y_1, ..., y_n$ is the dual basis of $\hfr_{W'}$.  Recall that $\eu_{W'}$ satisfies the commutation relations $[\eu_{W'}, x] = x$ for all $x \in \hfr_{W'}^*$, $[\eu_{W'}, y] = -y$ for all $y \in \hfr_{W'}$, and $[\eu_{W'}, w] = 0$ for all $w \in W$.  Note that in our setting, $\eu_{W'}$ is also fixed by the action of $N_W(W')$ on $H_c(W', \hfr_{W'})$ because the parameter $c : S' \rightarrow \cplx$ is obtained by restriction of the $W$-invariant parameter $c : S \rightarrow \cplx$.

Let $\oscr(N_W(W') \ltimes_{W'} (H_c(W', \hfr_{W'}) \otimes D(\hfr^{W'}_{reg})))$ denote the category of modules over the algebra $$N_W(W') \ltimes_{W'} (H_c(W', \hfr_{W'}) \otimes D(\hfr^{W'}_{reg}))$$ that are finitely generated over $\cplx[\hfr^{W'}_{reg}] \otimes \cplx[\hfr_{W'}]$ and on which $\hfr_{W'}$ acts locally nilpotently, and similarly let $\widehat{\oscr}(N_W(W') \ltimes_{W'} (\widehat{H}_c(W', \hfr_{W'})_0 \widehat{\otimes} D(\hfr^{W'}_{reg})))$ denote the category of $N_W(W') \ltimes_{W'} (\widehat{H}_c(W', \hfr_{W'})_0 \widehat{\otimes} D(\hfr^{W'}_{reg}))$-modules that are finitely generated over $\cplx[\hfr^{W'}_{reg}] \widehat{\otimes} \cplx[[\hfr_{W'}]].$  By the proof of \cite[Proposition 2.4]{BE}, $\eu_{W'}$ acts locally finitely on any $H_c(W', \hfr_{W'})$-module that has locally nilpotent action of $\hfr_{W'}$, and in particular $\eu_{W'}$ acts locally finitely on any $M \in \oscr(N_W(W') \ltimes_{W'} (H_c(W', \hfr_{W'}) \otimes D(\hfr^{W'}_{reg})))$.  It follows in particular that such $M$ are naturally graded as $$M = \bigoplus_{a \in \cplx} M_a$$ where $M_a$ is the generalized $a$-eigenspace of $\eu_{W'}$ in $M$.  As the actions of $\eu_{W'}$ and $N_W(W') \ltimes D(\hfr^{W'}_{reg})$ on $M$ commute, it follows that the above decomposition of $M$ is a decomposition as $N_W(W') \times D(\hfr^{W'}_{reg})$-modules.  As $\hfr_{W'}M_a \subset M_{a + 1}$ for all $a \in \cplx$ and as $M$ is finitely generated over $\cplx[\hfr_{W'}] \otimes \cplx[\hfr^{W'}_{reg}]$, it follows that each $M_a$ is finitely generated over $\cplx[\hfr^{W'}_{reg}]$ and that the set of $a \in \cplx$ such that $M_a \neq 0$ is a subset of a finite union of sets of the form $z + \ints^{\geq 0}$.  In particular, the generalized eigenspaces $M_a$ are $N_W(W')$-equivariant $\oscr(\hfr^{W'}_{reg})$-coherent $D(\hfr^{W'}_{reg})$-modules.  An argument entirely analogous to the proof of Theorem 2.3 in \cite{BE} then shows that the functor $$E : \widehat{\oscr}(N_W(W') \ltimes_{W'} (\widehat{H}_c(W', \hfr_{W'})_0 \widehat{\otimes} D(\hfr^{W'}_{reg}))) \rightarrow \oscr(N_W(W') \ltimes_{W'} (H_c(W', \hfr_{W'}) \otimes D(\hfr^{W'}_{reg})))$$ sending a module $M$ to its subspace of $\eu_{W'}$-locally finite vectors is quasi-inverse to the functor $$\widehat{\ \ }_{\hfr^{W'}_{reg}} : \oscr(N_W(W') \ltimes_{W'} (H_c(W', \hfr_{W'}) \otimes D(\hfr^{W'}_{reg}))) \rightarrow \widehat{\oscr}(N_W(W') \ltimes_{W'} (\widehat{H}_c(W', \hfr_{W'})_0 \widehat{\otimes} D(\hfr^{W'}_{reg})))$$ of completion at $\hfr^{W'}_{reg}$.  

Let $\widehat{\oscr}(\widehat{H}_c(W, \hfr)_{W\hfr^{W'}_{reg}})$ denote the category of $\widehat{H}_c(W, \hfr)_{W\hfr^{W'}_{reg}}$-modules that are finitely generated over the ring of functions $\cplx[\hfr]^{\wedge_{W\hfr^{W'}_{reg}}}$ on the formal neighborhood of $W\hfr^{W'}_{reg}$ in $\hfr$.  Completion at $W\hfr^{W'}$ defines an exact functor $$\widehat{\ \ }_{W\hfr^{W'}_{reg}} : \oscr_c(W, \hfr) \rightarrow \widehat{\oscr}(\widehat{H}_c(W, \hfr)_{W\hfr^{W'}_{reg}}).$$  Note also that the discussion above and the observation that $1_{\hfr^{W'}_{reg}}$ generates the unit ideal in $\widehat{H}_c(W, \hfr)_{W\hfr^{W'}_{reg}}$ shows that the composition $$\psi_* \circ 1_{\hfr^{W'}_{reg}} : \widehat{\oscr}(\widehat{H}_c(W, \hfr)_{W\hfr^{W'}_{reg}}) \rightarrow \widehat{\oscr}(N_W(W') \ltimes_{W'} (\widehat{H}_c(W', \hfr_{W'})_0 \widehat{\otimes} D(\hfr^{W'}_{reg})))$$ is an equivalence of categories.  Next, consider the composite functor $$E \circ \psi_* \circ 1_{\hfr^{W'}_{reg}} \circ \widehat{\ \ }_{W\hfr^{W'}_{reg}} : \oscr_c(W, \hfr) \rightarrow \oscr(N_W(W') \ltimes_{W'} (H_c(W', \hfr_{W'}) \otimes D(\hfr^{W'}_{reg})))$$ which as we've seen is a completion functor followed by equivalences of categories.

\begin{lemma} The image of the composite functor $E \circ \psi_* \circ 1_{\hfr^{W'}_{reg}} \circ \widehat{\ \ }_{W\hfr^{W'}_{reg}}$ lies in the full subcategory $\oscr(N_W(W') \ltimes_{W'} (H_c(W', \hfr_{W'}) \otimes D(\hfr^{W'}_{reg})))_{r.s.}$ of $\oscr(N_W(W') \ltimes_{W'} (H_c(W', \hfr_{W'}) \otimes D(\hfr^{W'}_{reg})))$ consisting of those modules $M$ whose generalized $\eu_{W'}$ eigenspaces $M_a$ have regular singularities as $\oscr(\hfr^{W'}_{reg})$-coherent $D(\hfr^{W'}_{reg})$-modules.  \end{lemma}

\begin{proof} Let $N$ be a module in $\oscr_c(W, \hfr)$ and let $M$ be its image under the composite functor $E \circ \psi_* \circ 1_{\hfr^{W'}_{reg}} \circ \widehat{\ \ }_{W\hfr^{W'}_{reg}}$.  As the full subcategory of $\oscr(\hfr^{W'}_{reg})$-coherent $D(\hfr^{W'}_{reg})$-modules consisting of those $D$-modules with regular singularities is a Serre subcategory, to show that the $D(\hfr^{W'}_{reg})$-module $M_a$ has regular singularities it suffices to show that every irreducible $D(\hfr^{W'}_{reg})$-module appearing as a composition factor in $M$ has regular singularities.  Note that $\hfr_{W'}^*$ acts on $M$ by $D(\hfr^{W'}_{reg})$-module homomorphisms of degree 1 with respect to the $\eu_{W'}$-grading.  It follows that $(\hfr_{W'}^*)^kM$ is a $D(\hfr^{W'}_{reg})$-submodule of $M$ for all integers $k \geq 0$, and, as $M$ is finitely generated over $\cplx[\hfr_{W'}] \otimes \cplx[\hfr^{W'}_{reg}]$, that any generalized $\eu_{W'}$-eigenspace $M_a$ of $M$ embeds in some quotient $M/(\hfr_{W'}^*)^kM$ for sufficiently large $k$.  Therefore, it suffices to show that the module $M/(\hfr_{W'}^*)^kM$ has regular singularities for all integers $k > 0$.  But $M/\hfr^*_{W'}M$ is precisely the $D(\hfr^{W'}_{reg})$-module obtained by pulling $N$ back to $\hfr^{W'}_{reg}$ as a coherent sheaf as in \cite[Proposition 1.2]{Wilcox}, which has regular singularities by \cite[Proposition 1.3]{Wilcox}.  That $M/(\hfr^*_{W'})^kM$ has regular singularities then follows from the exact sequence $$0 \rightarrow \frac{(\hfr^*_{W'})^{k - 1}M}{(\hfr^*_{W'})^{k}M} \rightarrow \frac{M}{(\hfr^*_{W'})^{k}M} \rightarrow \frac{M}{(\hfr^*_{W'})^{k - 1}M} \rightarrow 0$$ and induction on $k$.\end{proof}

By the Riemann-Hilbert correspondence \cite{De}, passing to local systems of flat sections on $\hfr^{W'}_{reg}$ in each generalized eigenspace then defines an equivalence of categories $$RH : \oscr(N_W(W') \ltimes_{W'} (H_c(W', \hfr_{W'}) \otimes D(\hfr^{W'}_{reg})))_{r.s.} \rightarrow (\oscr_c(W', \hfr_{W'}) \boxtimes \loc(\hfr^{W'}_{reg}))^{N_W(W')}$$ where $\oscr_c(W', \hfr_{W'}) \boxtimes \loc(\hfr^{W'}_{reg})$ is the category of local systems on $\hfr^{W'}_{reg}$ of $H_c(W', \hfr_{W'})$-modules in $\oscr_c(W', \hfr_{W'})$ and where $(\oscr_c(W', \hfr_{W'}) \boxtimes \loc(\hfr^{W'}_{reg}))^{N_W(W')}$ is the associated category of $N_W(W')$-equivariant objects.

\begin{definition} Let $\underline{\res}_{W'}^W$ be the \emph{partial KZ functor} $$\underline{\res}_{W'}^W : \oscr_c(W, \hfr) \rightarrow (\oscr_c(W', \hfr_{W'}) \boxtimes \loc(\hfr^{W'}_{reg}))^{N_W(W')}$$ defined by the composition $$\underline{\res}_{W'}^W := RH \circ E \circ \psi_* \circ 1_{\hfr^{W'}_{reg}} \circ \widehat{\ \ }_{W\hfr^{W'}_{reg}}.$$\end{definition}

Let $b \in \hfr^{W'}_{reg}$ and let $\fiber_b : (\oscr_c(W', \hfr_{W'}) \boxtimes \loc(\hfr^{W'}_{reg}))^{N_W(W')} \rightarrow \oscr_c(W', \hfr_{W'})$ denote the exact functor of taking the fiber of the local system at $b$.  From \cite[Section 3.7]{BE}, the functor $\fiber_b \circ \underline{\res}_{W'}^W$ is naturally identified with the parabolic restriction functor $\res_b$.  In particular, the monodromy in $\hfr^{W'}_{reg}$ provides isomorphisms of functors $\res_b \cong \res_{b'}$ for any points $b, b' \in \hfr^{W'}_{reg}$.  Note, however, that such an isomorphism is not canonical and depends on a choice of path between $b$ and $b'$.  This monodromy action on $\res_b$ will be crucial in what follows.

When there is no loss of clarity, we will suppress the choice of $b \in \hfr^{W'}$ and write $\res_{W'}^W$ for the parabolic restriction functor $\res_b$, and similarly we will write $\ind_{W'}^W$ for the parabolic induction functor $\ind_b$.  The underlined $\underline{\res}_{W'}^W$ will always denote the associated partial $KZ$ functor with fibers $\res_{W'}^W$.

\subsection{Gordon-Martino Transitivity} \label{transitivity} Let $W'' \subset W' \subset W$ be a chain of parabolic subgroups of $W$.  In \cite[Corollary 2.5]{Shan}, Shan proved that the parabolic restriction functors are transitive in the sense that there is an isomorphism of functors $$\res_{W''}^W \cong \res_{W''}^{W'} \circ \res_{W'}^W.$$  Gordon and Martino \cite{GoMa} explained the sense in which this transitivity is compatible with the local systems appearing in the partial $KZ$ functors.  We recall their result below.

Consider the functor $$\underline{\res}_{W''}^{W'} \boxtimes \id \circ \downarrow^{N_W(W')}_{N_W(W', W'')} \circ \underline{\res}_{W'}^W : \oscr_c(W, \hfr)$$ $$\rightarrow \left(\oscr_c(W'', \hfr_{W''}) \boxtimes \loc((\hfr_{W'})^{W''}_{reg}) \boxtimes \loc(\hfr^{W'}_{reg})\right)^{N_W(W', W'')}.$$  Here $N_W(W', W'')$ is the intersection of the normalizers $N_W(W')$ and $N_W(W'')$ and $\downarrow$ denotes the restriction of equivariant structure to a subgroup.  Similarly to the notation $\hfr^{W'}_{reg}$, the space $(\hfr_{W'})^{W''}_{reg}$ is the locally closed locus of points in $\hfr_{W'}$ with stabilizer $W''$ in $W'$.  The goal is to relate this functor to the partial $KZ$ functor $\underline{\res}_{W''}^W$.  However, the latter functor produces local systems on the space $\hfr^{W''}_{reg}$, not on $(\hfr_{W'})^{W''}_{reg} \times \hfr^{W'}_{reg}$ as the functor considered above does.  In general, viewed as subvarieties of $\hfr$, these spaces do not coincide, and there is no obvious map between them.  However, $(\hfr_{W'})^{W''}_{reg} \times \hfr^{W'}_{reg}$, viewed as a complex manifold, does contain a $N_W(W', W'')$-stable deformation retract that includes naturally in $\hfr^{W''}_{reg}$, giving rise to a pullback functor $$\iota_{W', W''}^* : \loc(\hfr^{W''}_{reg})^{N_W(W', W'')} \rightarrow \loc((\hfr_{W'})^{W''}_{reg} \times \hfr^{W'}_{reg})^{N_W(W', W'')}.$$  This functor can be constructed as follows.

Choose a $W$-invariant norm $||\cdot||$ on $\hfr$.  Let $\epsilon > 0$ be a positive number.  Define the subspaces $$\hfr^{W'}_{\epsilon-reg} := \{x \in \hfr^{W'} : ||wx - x|| > \epsilon \text{ for all } w \in W\backslash W'\} \subset \hfr^{W'}_{reg}$$ and $$(\hfr_{W'})^{W'', \epsilon}_{reg} := \{x \in (\hfr_{W'})^{W''}_{reg} : ||x|| < \epsilon/2\} \subset (\hfr_{W'})^{W''}_{reg}.$$  Note that these spaces are $N_W(W', W'')$-stable, that the subspace $(\hfr_{W'})^{W'', \epsilon}_{reg} \times \hfr^{W'}_{\epsilon-reg}$ is a deformation retract of $(\hfr_{W'})^{W''}_{reg} \times \hfr^{W'}_{reg}$ via a $N_W(W', W'')$-equivariant deformation retraction, and there is a natural inclusion $(\hfr_{W'})^{W'', \epsilon}_{reg} \times \hfr^{W'}_{\epsilon-reg} \subset \hfr^{W''}_{reg}$.  Pullback along the inclusion $(\hfr_{W'})^{W'', \epsilon}_{reg} \times \hfr^{W'}_{\epsilon-reg} \subset \hfr^{W''}_{reg}$ defines a functor $$\loc(\hfr^{W''}_{reg})^{N_W(W', W'')} \rightarrow \loc((\hfr_{W'})^{W'', \epsilon}_{reg} \times \hfr^{W'}_{\epsilon-reg})^{N_W(W', W'')}$$ and the deformation retraction defines an equivalence of categories $$\loc((\hfr_{W'})^{W'', \epsilon}_{reg} \times \hfr^{W'}_{\epsilon-reg})^{N_W(W', W'')} \cong \loc((\hfr_{W'})^{W''}_{reg} \times \hfr^{W'}_{reg})^{N_W(W', W'')}.$$  We define $\iota_{W', W''}^*$ to be the composition of these functors.  Note that $\iota_{W', W''}^*$ does not depend on the choice of norm $||\cdot||$ or $\epsilon > 0$.  We comment that the definition of $\iota_{W', W''}^*$ given in \cite{GoMa} was stated in terms of fundamental groups, but the version we give here is equivalent.

We can now state the transitivity result of Gordon-Martino:

\begin{theorem} \label{trans-theorem} \emph{(Gordan-Martino, \cite[Theorem 3.11]{GoMa})} There is a natural isomorphism $$\iota_{W', W''}^* \circ \downarrow^{N_W(W'')}_{N_W(W', W'')} \circ \underline{\res}_{W''}^W \cong (\underline{\res}_{W''}^{W'} \boxtimes \id) \circ \downarrow^{N_W(W')}_{N_W(W', W'')} \circ \underline{\res}_{W'}^W$$ of functors $$\oscr_c(W, \hfr) \rightarrow \left(\oscr_c(W'', \hfr_{W''}) \boxtimes \loc((\hfr_{W'})^{W''}_{reg}) \boxtimes \loc(\hfr^{W'}_{reg})\right)^{N_W(W', W'')}.$$\end{theorem}

\subsection{Mackey Formula for Rational Cherednik Algebras for Coxeter Groups} \label{Mackey-Coxeter} Shan and Vasserot established in \cite[Lemma 2.5]{SV} that the natural analogue of the usual Mackey formula for the composition of induction and restriction functors for representations of finite groups holds for the Bezrukavnikov-Etingof parabolic induction and restriction functors at the level of Grothendieck groups.  It will be convenient for us to know that, at least in the case of rational Cherednik algebras attached to finite Coxeter groups, the Mackey formula holds at the level of the parabolic induction and restriction functors themselves.  This is established below, by lifting the Mackey formula for the associated Hecke algebras via the $KZ$ functor.

Let $(W, S)$ be a finite Coxeter system with simple reflections $S$, real reflection representation $\hfr_\real$, and complexified reflection representation $\hfr$.  Let $c : S \rightarrow \cplx$ be a class function on the simple reflections, and let $q : S \rightarrow \cplx^\times$ be the associated class function $q(s) = e^{-2 \pi i c(s)}$.  Recall that the \emph{Hecke algebra} $\mathsf{H}_q(W)$ attached to $W$ and $q$ is the associative $\cplx$-algebra generated by symbols $\{T_s : s \in S\}$ subject to \emph{braid relations} $$T_{s_1}T_{s_2}T_{s_1} \cdots = T_{s_2}T_{s_1}T_{s_2} \cdots$$ for $s_1 \neq s_2 \in S$, where there are $m_{ij}$ terms on each side where $m_{ij}$ is the order of the product $s_1s_2$, and the \emph{quadratic relations} $$(T_s - 1)(T_s + q_s) = 0$$ where we write $q_s$ for $q(s)$.  Recall that the \emph{length function} $l : W \rightarrow \ints^{\geq 0}$ assigns to $w \in W$ the minimal length $l(w)$ of an expression $s_{i_1}\cdots s_{i_{l(w)}} = w$ of $w$ as a product of simple reflections, and that such a factorization of $w$ is called a \emph{reduced expression}.  The product $T_{s_{i_1}}\cdots T_{s_{i_{l(w)}}}$ in $\mathsf{H}_q(W)$ does not depend on the choice of reduced expression $s_{i_1}\cdots s_{i_{l(w)}}$ for $W$, and the resulting elements $T_w$ form a $\cplx$-basis for $\mathsf{H}_q(W)$.  For details and proofs on such basic structural results on Hecke algebras mentioned in this section, see \cite{GP} (note that their convention is to use quadratic relations $(T + 1)(T - q) = 0$ rather than $(T - 1)(T + q) = 0$).

For a subset $J \subset S$ of the simple reflections in $W$, let $W_J \subset W$ be the parabolic subgroup generated by $J \subset S$.  Then $(W_J, J)$ is a Coxeter subsystem of $(W, S)$, and the (complexified) reflection representation of $(W_J, J)$ is identified with the action of $W_J$ on the unique $W_J$-stable complement $\hfr_{W_J}$ of $\hfr^{W_J}$ in $\hfr$.  The parameter $q : S \rightarrow \cplx^\times$ restricts to give a $W_J$-invariant function $J \rightarrow \cplx^\times$, which by abuse of notation we also denote by $q$.  We therefore may form the Hecke algebra $\mathsf{H}_q(W_J)$.  Reduced expressions of $w \in W_J$ as an element of $W_J$ are precisely the reduced expressions of $w$ as an element of $W$.  In particular the length function $l_{W_J}$ for $(W_J, J)$ coincides with the restriction to $W_J$ of the length function for $(W, J)$, and in this way the assignment $T_w \mapsto T_w$ for $w \in W$ extends uniquely to a unital embedding of the algebra $\mathsf{H}_q(W_J)$ as a subalgebra of $\mathsf{H}_q(W)$.  Via this embedding we can define the parabolic restriction functor $$\hres_{W_J}^W : \mathsf{H}_q(W)\mhyphen\text{mod} \rightarrow \mathsf{H}_q(W_J)\mhyphen\text{mod}$$ and the parabolic induction functor $$\hind_{W_J}^W := \mathsf{H}_q(W) \otimes_{\mathsf{H}_q(W_J)} \bullet : \mathsf{H}_q(W_J)\mhyphen\text{mod} \rightarrow \mathsf{H}_q(W)\mhyphen\text{mod}.$$

Let $J, K \subset S$ be two subsets of simple reflections, with associated parabolic subgroups $W_J$ and $W_K$.  Each $W_K$-$W_J$ double-coset in $W$ contains a unique element of minimal length.  Let $X_{KJ} \subset W$ denote this subset of distinguished double-coset representatives.  For $d \in X_{KJ}$, conjugation $w \mapsto dwd^{-1}$ defines a length-preserving isomorphism $W_{J \cap K^d} \rightarrow W_{^dJ \cap K}$, where $K^d := d^{-1}Kd$ and $^dJ := dJd^{-1}$.  In particular, the assignment $T_w \mapsto T_{dwd^{-1}}$ extends uniquely to an algebra isomorphism $\mathsf{H}_q(W_{J \cap K^d}) \rightarrow \mathsf{H}_q(W_{^dJ \cap K})$.  Note that this isomorphism is realized inside $\mathsf{H}_q(W)$ by conjugation by $T_d$.  Let $$\htw_d : \mathsf{H}_q(W_{J \cap K^d})\mhyphen\text{mod} \rightarrow \mathsf{H}_q(W_{^dJ \cap K})\mhyphen\text{mod}$$ denote the equivalence of categories obtained by transfer of structure via this isomorphism.

We can now state the Mackey formula for Hecke algebras.  Note that this formula holds for any numerical parameter $q : S \rightarrow \cplx^\times$.

\begin{proposition} \emph{(Mackey Formula for Hecke Algebras, \cite[Proposition 9.1.8]{GP})} Let $J, K \subset S$ and let $X_{KJ} \subset W$ be the set of distinguished (minimal length) $W_K$-$W_J$ double-coset representatives in $W$.  There is an isomorphism $$\hres^W_{W_K} \circ \hind_{W_J}^W \cong \bigoplus_{d \in X_{KJ}} \hind_{W_{{ }^dJ \cap K}}^W \circ \htw_d \circ \hres_{J \cap K^d}^{W_J}$$ of functors $\mathsf{H}_q(W_J)\mhyphen\text{mod} \rightarrow \mathsf{H}_q(W_K)\mhyphen\text{mod}.$  \end{proposition}

By restriction of the parameter $c$ to subsets $J \subset S$, one can form the associated rational Cherednik algebra $H_c(W_J, \hfr_{W_J})$.  Let $J, K \subset S$, and let $d \in X_{KJ}$.  Note that the action of $d$ on $\hfr$ induces an isomorphism $\hfr_{W_{J \cap K^d}} \rightarrow \hfr_{W_{^dJ \cap K}}$, $y \mapsto dy$, and hence also an isomorphism $\hfr^*_{W_{J \cap K^d}} \rightarrow \hfr^*_{W_{^dJ \cap K}}$, $f \mapsto d(f) = f(d^{-1}\bullet).$  As in the comments preceding \cite[Lemma 2.5]{SV}, these isomorphisms along with the isomorphism $W_{J \cap K^d} \rightarrow W_{^dJ \cap K}$ discussed above induce an isomorphism $H_c(W_{J \cap K^d}, \hfr_{W_{J \cap K^d}}) \cong H_c(W_{^dJ \cap K}, \hfr_{W_{^dJ \cap K}})$ respecting triangular decompositions.  Transfer of structure by this isomorphism therefore induces an equivalence of categories $$\tw_d : \oscr_c(W_{J \cap K^d}, \hfr_{W_{J \cap K^d}}) \cong \oscr_c(W_{^dJ \cap K}, \hfr_{W_{^dJ \cap K}}).$$

\begin{theorem} \emph{(Mackey Formula for Rational Cherednik Algebras for Coxeter Groups)} Let $J, K \subset S$ and let $X_{KJ} \subset W$ be the set of distinguished $W_K$-$W_J$ double-coset representatives in $W$.  There is an isomorphism $$\res^W_{W_K} \circ \ind_{W_J}^W \cong \bigoplus_{d \in X_{KJ}} \ind_{W_{{ }^dJ \cap K}}^W \circ \tw_d \circ \res_{J \cap K^d}^{W_J}$$ of functors $\oscr_c(W_J, \hfr_{W_J}) \rightarrow \oscr_c(W_K, \hfr_{W_K}).$\end{theorem}

\begin{proof} For a subset $A \subset S$, let $KZ(W_A, \hfr_{W_A})$ denote the $KZ$ functor $\oscr_c(W_A, \hfr_{W_A}) \rightarrow \mathsf{H}_q(W_A)\mhyphen\text{mod}.$  It follows from \cite[Lemma 2.4]{Shan} that to show the existence of the desired isomorphism of functors, we need only check that the functors $KZ(W_K, \hfr^{W_K}) \circ \res^W_{W_K} \circ \ind_{W_J}^W$ and $KZ(W_K, \hfr_{W_K}) \circ \bigoplus_{d \in X_{KJ}} \ind_{W_{{ }^dJ \cap K}}^W \circ \tw_d \circ \res_{J \cap K^d}^{W_J}$ are isomorphic.  Shan also proved (\cite[Theorem 2.1]{Shan}) that $KZ$ commutes with parabolic restriction functors in the sense that there is an isomorphism of functors $$KZ(W_A, \hfr_{W_A}) \circ \res_{W_A}^{W_B} \cong \hres_{W_A}^{W_B} \circ KZ(W_B, \hfr_{W_B})$$ for any $A \subset B \subset S$.  It is clear that $KZ$ commutes with $\tw_d$ in the sense that there is an isomorphism of functors $KZ(W_{^dJ \cap K}, \hfr_{W^dJ \cap K}) \circ \tw_d \cong \htw_d \circ KZ(W_{J \cap K^d}, \hfr_{W_{J \cap K^d}}).$  Passing the $KZ$ functor to the right using these isomorphisms, the desired statement then follows from the Mackey formula for Hecke algebras.\end{proof}

\section{Endomorphism Algebras Via Monodromy} \label{generalities-section} Throughout this section, let $W$ be a complex reflection group with reflection representation $\hfr$, let $S \subset W$ be the set of reflections in $W$, and let $c : S \rightarrow \cplx$ be a $W$-invariant function.  Let $H_c(W, \hfr)$ be the associated rational Cherednik algebra.  In light of the isomorphism $H_c(W, \hfr) \cong H_c(W, \hfr_{W}) \otimes D(\hfr^{W})$, we will always assume that the fixed space $\hfr^W$ is the trivial subspace.  In order to understand and count irreducible representations in $\oscr_c(W, \hfr)$ with a given support in $\hfr$, it will be convenient to consider certain subcategories and subquotient categories of $\oscr_c(W, \hfr)$.

\subsection{Harish-Chandra Series for Rational Cherednik Algebras}

In the following definition, let $W' \subset W$ be a parabolic subgroup and let $L$ be a finite-dimensional irreducible representation in $\oscr_c(W', \hfr_{W'})$.  

\begin{definition} Let $W' \subset W$ be a parabolic subgroup and let $L$ be a finite-dimensional irreducible representation in $\oscr_c(W', \hfr_{W'})$.  Let $\oscr_{c, W'}(W, \hfr)$ denote the full subcategory of $\oscr_c(W, \hfr)$ consisting of modules supported on $W\hfr^{W'}$, and let $\oscr_{c, W', L}(W, \hfr)$ denote the full subcategory of $\oscr_{c, W'}(W, \hfr)$ consisting of modules $M$ such that $\res^W_{W'} M \in \oscr_c(W', \hfr_{W'})$ is semisimple with all irreducible constituents in orbit of $L$ for the action of $N_W(W')$ on $\irr(H_c(W', \hfr_{W'}))$.  Let $\overline{\oscr}_{c, W'}(W, \hfr)$ and $\overline{\oscr}_{c, W', L}(W, \hfr)$ denote the respective quotient categories by the kernel of the exact functor $\res^W_{W'}$.\end{definition}

When the ambient reflection representation $(W, \hfr)$ is clear, we will write $\oscr_{c, W'}$ for $\oscr_{c, W'}(W, \hfr)$, $\oscr_{c, W', L}$ for $\oscr_{c, W', L}(W, \hfr)$, and similarly for their respective quotients $\bar{\oscr}_{c, W'}$ and $\bar{\oscr}_{c, W', L}$.  

\begin{proposition} \label{simples} The following hold:

(1) We have $$\ext^1_{\oscr_c(W', \hfr_{W'})}(L', L'') = 0$$ for all $L', L''$ in the $N_W(W')$-orbit of $L$.

(2)  The categories $\oscr_{c, W'}$ and $\oscr_{c, W', L}$ are Serre subcategories of $\oscr_c(W, \hfr)$.

(3)  Every simple object $S \in \oscr_c(W, \hfr)$ lies in such a category $\oscr_{c, W', L}$ for some parabolic subgroup $W' \subset W$ and finite dimensional irreducible $L$, and the pair $(W', L)$ is uniquely determined by $S$ up to the natural $W$-action on such pairs.

(4)  If $(W', L)$ labels the simple module $S$ in this way, then $\supp(S) = W\hfr^{W'}$.\end{proposition}

\begin{proof} For the $\ext^1$ statement, notice that the restriction of the parameter $c$ to the set of reflection $S' \subset W'$ is not only $W'$-equivariant but also $N_W(W')$-equivariant.  As the Euler grading element $\eu_{W'}$ is fixed by the $N_W(W')$-action, it follows that the lowest $\eu_{W'}$-weight spaces of $L$ and $L'$ have the same weight, and therefore there can be no nontrivial extensions between $L$ and $L'$ by usual highest weight theory (see, e.g. \cite[Section 2]{GGOR}).  More precisely, let $L$ have lowest weight $\lambda \in \irr(W')$ and $L'$ have lowest weight $\lambda' = n.\lambda$ for some $n \in N_W(W')$, and consider a module $M \in \oscr_c(W', \hfr_{W'})$ such that $[M] = [L_c(\lambda)] + [L_c(\lambda')]$ in $K_0(\oscr_c(W', \hfr_{W'}))$.  It suffices to show that $M \cong L_c(\lambda) \oplus L_c(\lambda')$.  As the lowest weight spaces $\lambda$ of $L_c(\lambda)$ and $\lambda'$ of $L_c(\lambda')$ occur in the same graded degree with respect to the grading by generalized eigenspaces of $\eu_{W'}$, the lowest weight space of $M$ is isomorphic to $\lambda \oplus \lambda'$ as a representation of $W'$.  It follows that there is an $H_c(W', \hfr_{W'})$-module homomorphism $\Delta_c(\lambda) \oplus \Delta_c(\lambda') \rightarrow M$ that is an isomorphism in the lowest weight space, and as $[M] = [L_c(\lambda)] + [L_c(\lambda')]$ this map must annihilate the unique maximal proper submodules of each $\Delta_c(\lambda)$ and $\Delta_c(\lambda')$, as their simple constituents have strictly higher lowest weight spaces, and hence factor through an isomorphism $L_c(\lambda) \oplus L_c(\lambda') \rightarrow M$, as needed.  That $\oscr_{c, W'}$ and $\oscr_{c, W', L}$ are Serre subcategories then follows from the exactness of $\res^W_{W'}$.

Given a simple module $S \in \oscr_c(W, \hfr)$, its support is of the form $W\hfr^{W'}$ for a parabolic subgroup $W' \subset W$ uniquely determined up to conjugation, and $\res^W_{W'}S$ is finite dimensional and nonzero for a parabolic subgroup $W' \subset W$ if and only if $\supp(S) = W\hfr^{W'}$, see \cite[Proposition 3.2]{BE}.  So, there exists a finite dimensional simple module $L \in \oscr_c(W', \hfr_{W'})$ and a nonzero homomorphism $L \rightarrow \res^W_{W'}S$.  By adjunction, there is therefore a nonzero homomorphism $\ind^W_{W'}L \rightarrow S$, which is a surjection because $S$ is simple.  So it suffices to check $\ind^W_{W'}L \in \oscr_{c, W', L}$, i.e. that the irreducible constituents of $\res^W_{W'}\ind^W_{W'}L$ lie in the $N_W(W')$-orbit of $L$.  This follows from the Mackey formula for rational Cherednik algebras at the level of Grothendieck groups (\cite[Lemma 2.5]{SV}) and the fact that $L$ is annihilated by any nontrivial parabolic restriction functor.\end{proof}

\begin{remark} Proposition \ref{simples} can also be obtained from \cite[Theorem 3.4.6]{Losevcompletions}.\end{remark}

The labeling of simple modules in $\oscr_c(W, \hfr)$ by pairs $(W', L)$ as in Proposition \ref{simples} is an analogue in the setting of rational Cherednik algebras of the partitioning of irreducible representations of finite groups of Lie type into Harish-Chandra series.

\begin{proposition} \label{proj-gen} The image of $\ind^W_{W'}L$ in the quotient category $\overline{\oscr}_{c, W', L}$ is a projective generator.  Furthermore, the natural map $$\End_{\oscr_{c, W', L}}(\ind^W_{W'}L) \rightarrow \End_{\overline{\oscr}_{c, W', L}}(\ind^W_{W'}L)$$ is an isomorphism, and in particular there is an equivalence of categories $$\overline{\oscr}_{c, W', L} \cong \End_{\oscr_c(W', \hfr_{W'})}(\ind^W_{W'}L)^{opp}\mhyphen\emph{mod}_{f.d.}.$$\end{proposition}

\begin{proof} The statement comparing the endomorphism algebra of $\ind^W_{W'}L$ in $\oscr_{c, W', L}$ and in $\overline{\oscr}_{c, W', L}$ follows from the observation that $\ind^W_{W'}L$ has no nonzero submodules or quotients annihilated by $\res^W_{W'}$ and from the definition of morphisms in the quotient category.  Otherwise, there would exist a simple module $S \in \oscr_{c, W', L}$ such that $\res^W_{W'}S = 0$ but such that either $\hom_{\oscr_{c, W', L}}(S, \ind^W_{W'}L)$ or $\hom_{\oscr_{c, W', L}}(\ind^W_{W'}L, S)$ is nonzero.  As $\ind^W_{W'}$ and $\res^W_{W'}$ are biadjoint, this is equivalent to one of $\hom_{\oscr_c(W', \hfr_{W'})}(\res^W_{W'}S, L)$ or $\hom_{\oscr_c(W', \hfr_{W'})}(L, \res^W_{W'}S)$ being nonzero, which contradicts $\res^W_{W'}S = 0$.

The argument in the last paragraph of the proof of Proposition \ref{simples} shows that $\ind^W_{W'}L$ admits a surjection to any simple module in $\bar{\oscr}_{c, W', L}$, so the claim that it is a projective generator will follow as soon as we see that it is projective in $\overline{\oscr}_{c, W', L}$.  For this, note that $\res_{W'}^W$ and $\ind_{W'}^W$ induce a biadjoint pair of functors between $\bar{\oscr}_{c, W', L}$ and the Serre subcategory of $\oscr_c(W', \hfr_{W'})$ generated by those simple objects in the $N_W(W')$-orbit of $L$; indeed, that these functors are biadjoint follows from the biadjunction of $\res_{W'}^W$ and $\ind_{W'}^W$ as functors between $\oscr_c(W, \hfr)$ and $\oscr_c(W', \hfr_{W'})$ and the facts that every object in $\bar{\oscr}_{c, W', L}$ is of the form $\pi(M)$, where $\pi : \oscr_{c, W', L} \rightarrow \bar{\oscr}_{c, W', L}$ is the quotient functor and where $M \in \oscr_{c, W', L}$ is such that all simple objects in its head and socle have support equal to $W\hfr^{W'}$, that $\ind_{W'}^W M'$ is such an object for every $M'$ in the specified Serre subcategory of $\oscr_c(W', \hfr_{W'})$, and that $\pi$ is fully faithful on such objects.  This Serre subcategory is semisimple by the $\ext^1$ statement from Proposition \ref{simples}, and in particular $L$ is projective in that category.  Biadjoint functors preserve projectives, so $\ind_{W'}^W L$ is projective in $\bar{\oscr}_{c, W', L}$ as needed.

The category $\oscr_c(W, \hfr)$ is equivalent to the category of finite dimensional modules of a finite dimensional $\cplx$-algebra \cite[Theorem 5.16]{GGOR}.  The same is true for its Serre subquotient $\overline{\oscr}_{c, W', L},$ which therefore is equivalent to the category of finite dimensional modules over the opposite endomorphism algebra of any projective generator.  The last statement follows.\end{proof}

Through conjugation and the $W$-action on $\hfr$, any $w \in W$ determines an isomorphism $H_c(W', \hfr_{W'}) \rightarrow H_c(^wW', \hfr_{^wW'})$ and therefore also a bijection $$\irr(H_c(W', \hfr_{W'})) \rightarrow \irr(H_c(^wW', \hfr_{^wW'})), \ \ \ M \mapsto { }^wM$$ on the sets of isomorphism classes of irreducible representations by transfer of structure.  The group $W$ therefore acts on the set of pairs $(W', L)$,  where $W' \subset W$ is a parabolic subgroup and $L$ is a finite-dimensional irreducible representation of $H_c(W', \hfr_{W'})$, by $w.(W', L) = (^wW', { }^wL)$.  Note that $W' \subset W$ stabilizes the pair $(W', L)$.  This leads to the following definition:

\begin{definition} \label{inertia-definition} The \emph{inertia group} $I_W(W', L)$ of the pair $(W', L)$ is the subgroup $$I_W(W', L) := \{w \in W : w.(W', L) = (W', L)\}/W' \subset N_W(W')/W'.$$\end{definition}

The following statement follows from Propositions \ref{simples} and \ref{proj-gen}:

\begin{corollary} \label{counting} The number of isomorphism classes of simple modules in $\oscr_c(W, \hfr)$ labeled by the pair $(W', L)$ in the sense of Proposition \ref{simples} is the number $\#\irr(\End_{\oscr_c(W, \hfr)}(\ind_{W'}^W L)^{opp})$ of irreducible representations of the $\End_{\oscr_c(W, \hfr)}(\ind_{W'}^W L)^{opp}$.  The number of isomorphism classes of simple modules in $\oscr_c(W, \hfr)$ with support variety $W\hfr^{W'}$ is given by $$\sum_{\substack{L \in \irr(\oscr_c(W', \hfr_{W'})),\\ \dim_\cplx L < \infty}} \frac{\#\irr(\End_{\oscr_c(W, \hfr)}(\ind_{W'}^W L)^{opp})}{[N_W(W')/W' : I_W(W', L)]}.$$\end{corollary}

\begin{proof} Immediate from Propositions \ref{simples} and \ref{proj-gen} and the definition of the inertia group. \end{proof}

Note that Corollary \ref{counting} is only non-vacuous in the case $W' \neq W$, i.e. for counting isomorphism classes of simple modules with support strictly containing $W\hfr^W = \{0\}$, i.e. for counting isomorphism classes of infinite-dimensional simple modules in $\oscr_c(W, \hfr)$.  But the total number of irreducible representations in $\oscr_c(W, \hfr)$ is $\#\irr(W)$, and in this way one also obtains counts of the finite-dimensional simple modules by subtracting the number of infinite-dimensional simple modules.

The inertia group $I_{W', L}$ will be important to our study of $\End_{\oscr_c(W, \hfr)}(\ind_{W'}^W L)^{opp}.$  The following is a basic result about the endomorphism algebra:

\begin{proposition} \label{dim-prop} We have $$\dim_{\cplx} \End_{H_c}(\ind^W_{W'}L) = \#I_W(W', L)$$\end{proposition}

\begin{proof} By adjunction, we have $$\End_{\oscr_c(W, \hfr)}(\ind^W_{W'}L) \cong \hom_{\oscr_c(W', \hfr_{W'})}(L, \res^W_{W'}\ind^W_{W'}L).$$  By the Mackey formula for rational Cherednik algebras at the level of Grothendieck groups \cite[Lemma 2.5]{SV} and the fact that $L$ is finite-dimensional and therefore annihilated by any nontrivial parabolic restriction functor, it follows that the class of $\res_{W'}^W \ind_{W'}^W L$ in $K_0(\oscr_c(W', \hfr_{W'}))$ equals $\sum_{n \in N_W(W')/W'} [{ }^nL]$.  The $\ext^1$ vanishing statement from Proposition \ref{simples} therefore implies that $\res_{W'}^W\ind_{W'}^W L$ is isomorphic to the direct sum $\oplus_{n \in N_W(W')/W'} { }^nL$.  We therefore have $$\hom_{\oscr_c(W', \hfr_{W'})}(L, \res^W_{W'}\ind^W_{W'}L) = \hom_{\oscr_c(W', \hfr_{W'})}(L, \bigoplus_{n \in N_W(W')/W'} { }^nL).$$ As $\hom_{\oscr_c(W', \hfr_{W'})}(L, { }^nL)$ is isomorphic to $\cplx$ if $n \in I_{W', L}$ and is 0 otherwise, the claim follows.\end{proof}

With this dimension result in mind, our goal is to give a presentation of $\End_{H_c}(\ind^W_{W'}L)^{opp}$ as a twisted extension of a Hecke algebra, with a natural basis indexed by $I_W(W', L)$, analogous to the presentation by Howlett and Lehrer \cite{HoLe} of endomorphism algebras of induced cuspidal representations in the representation theory of finite groups of Lie type.  A first step towards this goal is to realize $\End_{H_c}(\ind^W_{W'}L)^{opp}$ as a quotient of a twist of the group algebra of the fundamental group of $\hfr^{W'}_{reg}/I_W(W', L)$ by a group 2-cocycle.  This is achieved in the following section.

\subsection{The Functor $KZ_L$}

\

\noindent $\mathbf{Notation}$: Throughout this section we will again fix a parabolic subgroup $W' \subset W$ and a finite-dimensional irreducible representation $L$ of $H_c(W', \hfr_{W'})$.  The only parabolic induction and restrictions functors appearing will be between rational Cherednik algebras associated to $W'$ and $W$, so we will omit the superscript $W$ and subscript $W'$ in the notation for the functors $\res_{W'}^W$, $\ind_{W'}^W$, and $\underline{\res}_{W'}^W$.  We will suppress the basepoint $b$ in the notation for the fundamental group of $\hfr^{W'}_{reg}$ and its quotients, as this basepoint plays no role in this section.

\begin{lemma} \label{full} The restriction $$\underline{\res} : \oscr_{c, W'} \rightarrow (\oscr_c(W', \hfr_{W'}) \boxtimes \loc(\hfr^{W'}_{reg}))^{N_W(W')}$$ of the partial $KZ$ functor to the subcategory $\oscr_{c, W'}$ of $\oscr_c(W, \hfr)$ is exact, full, and has image closed under subquotients.  The same is true for its restriction to $\oscr_{c, W', L}$.\end{lemma}

\begin{proof} Recall that by definition $\underline{\res}$ is the composition $RH \circ E \circ \psi_* \circ 1_{\hfr^{W'}_{reg}} \circ \widehat{\ \ }_{W\hfr^{W'}_{reg}}.$  All the functors involved in this composition after $\widehat{\ \ }_{W\hfr^{W'}_{reg}}$ are equivalences of categories, and in particular it suffices to check the claims of the lemma for the functor $\widehat{\ \ }_{W\hfr^{W'}_{reg}}$.  Observe that $W\hfr^{W'}_{reg}$ is a principal open subset of $W\hfr^{W'}$, defined by the nonvanishing of the polynomial $$\delta_{W'} := \sum_{w \in W} w\left(\prod_{s \notin S'} \alpha_s.\right).$$  Indeed, for any $w, w' \in W$ with $w\hfr^{W'} \neq w'\hfr^{W'}$ the nonvanishing of the polynomial $w(\prod_{s \notin S'} \alpha_s)$ defines $w\hfr^{W'}_{reg}$ in $w\hfr^{W'}$ and $w'(\prod_{s \notin S'} \alpha_s)$ vanishes identically on $w\hfr^{W'}$ ($w\hfr^{W'}$ is the common vanishing locus of the $w\alpha_s$ for $s \in S'$, and as $w\hfr^{W'} \neq w'\hfr^{W'}$ it follows that $S'$ - and hence its complement in $S$ - is not stable under conjugation by $w^{-1}w'$, so $w^{-1}w'(\prod_{s \notin S'}\alpha_s)$ vanishes on $\hfr^{W'}$ and hence $w'(\prod_{s \notin S'} \alpha_s)$ vanishes on $w\hfr^{W'}$).   Modules $M$ in $\oscr_{c, W'}$ are set-theoretically supported on $W\hfr^{W'}$, and in particular the completion functor $\widehat{\ \ }_{W\hfr^{W'}_{reg}}$ is simply localization by $\delta_{W'}$.  

Define the categories $$H_c(W, \hfr)\mhyphen\text{mod}_{\text{Supp} \subset W\hfr^{W'}}, \ \ \ \ H_c(W, \hfr)^{\wedge W\hfr^{W'}_{reg}}\mhyphen\text{mod}_{\text{Supp} \subset W\hfr^{W'}}$$ to be the full subcategories of modules over the respective algebras which are annihilated by sufficiently high powers of the ideal of $W\hfr^{W'}$.  The induction functor $$\theta_* : H_c(W, \hfr)\mhyphen\text{mod}_{\text{Supp} \subset W\hfr^{W'}} \rightarrow H_c(W, \hfr)^{\wedge W\hfr^{W'}_{reg}}\mhyphen\text{mod}_{\text{Supp} \subset W\hfr^{W'}}$$ given by extension of scalars along the map of algebras $\theta : H_c(W, \hfr) \rightarrow H_c(W, \hfr)^{\wedge W\hfr^{W'}_{reg}}$ amounts to localization by the function $\delta_{W'}$ defining $W\hfr^{W'}_{reg}$ in $W\hfr^{W'}$.  We also have the associated pullback (restriction) functor $\theta^*$, and as $\theta_*$ is localization by a single element we have $\theta_* \circ \theta^*$ is isomorphic to the identity.  It follows that $\theta_*$ induces an equivalence of categories $$\frac{H_c(W, \hfr)\mhyphen\text{mod}_{\text{Supp} \subset W\hfr^{W'}}}{\ker \theta_*} \rightarrow H_c(W, \hfr)^{\wedge W\hfr^{W'}_{reg}}\mhyphen\text{mod}_{\text{Supp} \subset W\hfr^{W'}}.$$  As $\oscr_{c, W'}$ is a Serre subcategory of $H_c(W, \hfr)\mhyphen\text{mod}_{\text{Supp} \subset W\hfr^{W'}}$ and $\ker \res = (\ker \theta_*) \cap \oscr_{c, W'},$ the canonical functor $$\overline{\oscr}_{c, W'} \rightarrow \frac{H_c(W, \hfr)\mhyphen\text{mod}_{\text{Supp} \subset W\hfr^{W'}}}{\ker \theta_*}$$ is an inclusion of a full subcategory closed under subquotients, and the result follows.  For identical reasons, the result holds for the restriction of $\underline{\res}$ to the Serre subcategories $\oscr_{c, W', L}$.\end{proof}

It will be convenient for us to take a semidirect product splitting of the normalizer $N_W(W')$ as given by the following lemma:

\begin{lemma} \label{split} \emph{(\cite{Mu})} There is a subgroup $N_{W'} \subset N_W(W')$ such that there is a semidirect product decomposition $N_W(W') = W' \rtimes N_{W'}$.\end{lemma}

Fix a complement $N_{W'}$ to $W'$ in $N_W(W')$ as in Lemma \ref{split}.  Let $I_{W', L} \subset N_{W'}$ denote the stabilizer of $L$ under the action of $N_{W'}$ on $\irr(H_c(W', \hfr_{W'}))$.  The natural map $I_{W', L} \rightarrow I_W(W', L)$ is an isomorphism.

\noindent $\mathbf{Notation}$: Throughout this section, in which $W'$ and $L$ are fixed, we will simplify the notation by denoting $I_{W', L}$ by $I$.

Given a set $\mathcal{L}$ of simple objects in $\oscr_c(W', \hfr_{W'})$, let $\langle \mathcal{L} \rangle_{\oscr_c(W', \hfr_{W'})}$ denote the Serre subcategory of $\oscr_c(W', \hfr_{W'})$ generated by $\mathcal{L}$.  Note that the functor $\underline{\res} : \oscr_{c, W', L} \rightarrow (\oscr_c(W', \hfr_{W'}) \boxtimes \loc(\hfr^{W'}_{reg}))^{N_W(W')}$ factors through the Serre subcategory $$\left(\langle N_W(W').L\rangle_{\oscr_c(W', \hfr_{W'})} \boxtimes \loc(\hfr^{W'}_{reg})\right)^{N_W(W')} \subset \left(\oscr_c(W', \hfr_{W'}) \boxtimes \loc(\hfr^{W'}_{reg})\right)^{N_W(W')}.$$  Let $H_{c, L}(W', \hfr_{W'})$ be the quotient $$H_{c, L}(W', \hfr_{W'}) := H_c(W', \hfr_{W'})/\cap_{n \in N_W(W')} \ann_{H_c(W', \hfr_{W'})}({ }^nL).$$  Note that $H_{c, L}(W', \hfr_{W'})$ is a semisimple finite-dimensional $\cplx$-algebra with $N_W(W')$-action and that the category $\langle N_W(W').L \rangle_{\oscr_c(W', \hfr_{W'})}$ is naturally equivalent to $H_{c, L}(W', \hfr_{W'})\mhyphen\text{mod}_{f.d.}$.  In particular, choosing a point $b \in \hfr^{W'}_{reg}$, taking the fiber of the local system at $b$ defines an equivalence of categories $$\fiber_b : \left(\langle N_W(W').L\rangle_{\oscr_c(W', \hfr_{W'})} \boxtimes \loc(\hfr^{W'}_{reg})\right)^{N_W(W')} \rightarrow H_{c, L}(W', \hfr_{W'}) \rtimes \pi_1(\hfr^{W'}_{reg}/N_{W'})\mhyphen\text{mod}_{f.d}$$ where the semidirect product is defined by the natural action of $\pi_1(\hfr^{W'}_{reg}/N_{W'})$ on $H_{c, L}(W', \hfr_{W'})$ through the natural projection $\pi_1(\hfr^{W'}_{reg}/N_{W'}) \rightarrow N_{W'}.$  Let $e_L$ denote the central idempotent associated to the irreducible representation $L$ of $H_{c, L}(W', \hfr_{W'})$.  Clearly $e_L$ generates the unit ideal in $H_{c, L}(W', \hfr_{W'}) \rtimes \pi_1(\hfr^{W'}_{reg}/N_{W'})$ and the associated spherical subalgebra is given by $$e_L(H_{c, L}(W', \hfr_{W'}) \rtimes \pi_1(\hfr^{W'}_{reg}/N_{W'}))e_L \cong \End_\cplx(L) \rtimes \pi_1(\hfr^{W'}_{reg}/I).$$  In particular, multiplication by $e_L$ defines an equivalence of categories $$e_L : H_{c, L}(W', \hfr_{W'}) \rtimes \pi_1(\hfr^{W'}_{reg}/N_{W'})\mhyphen\text{mod}_{f.d} \rightarrow \End_\cplx(L) \rtimes \pi_1(\hfr^{W'}_{reg}/I)\mhyphen\text{mod}_{f.d}.$$

The automorphism group of $\End_\cplx(L)$ is the projective general linear group $PGL_\cplx(L)$, and in particular the action of $I$ on $\End(L)$ defines a projective representation $\pi_L$ of $I$ on $L$.  Let $\mu \in Z^2(I, \cplx^\times)$ be a group 2-cocycle of $I$ with coefficients in $\cplx^\times$ representing the cohomology class associated to $\pi_L$.  Let $\tilde{\mu} \in Z^2(\pi_1(\hfr^{W'}_{reg}/I), \cplx^\times)$ be the pullback of $\mu$ along the natural projection $\pi_1(\hfr^{W'}_{reg}/I) \rightarrow I.$  Let $-\tilde{\mu}$ denote the opposite group cocycle defined by $(-\tilde{\mu})(g) = \tilde{\mu}(g)^{-1}$, and for any group 2-cocycle $\nu \in Z^2(\pi_1(\hfr^{W'}_{reg}/I), \cplx^\times)$ let $\cplx_\nu[\pi_1(\hfr^{W'}_{reg}/I)]$ denote the $\nu$-twisted group algebra of $\pi_1(\hfr^{W'}_{reg}/I)$, i.e. the associative unital $\cplx$-algebra with $\cplx$-basis $\{e_g : g \in \pi_1(\hfr^{W'}_{reg}/I)\}$ and multiplication $e_ge_{g'} = \nu(g, g')e_{gg'}.$  Essentially by definition of $\mu$ we see that $L$ is naturally a $\cplx_{-\tilde{\mu}}[\pi_1(\hfr^{W'}_{reg}/I)]$-module, and in particular we have a functor $$\hom_{\End_{\cplx}(L)}(L, \bullet) : \End_\cplx(L) \rtimes \pi_1(\hfr^{W'}_{reg}/I)\mhyphen\text{mod}_{f.d.} \rightarrow \cplx_{-\tilde{\mu}}[\pi_1(\hfr^{W'}_{reg}/I)]\mhyphen\text{mod}_{f.d.}.$$  This functor is an equivalence of categories, with quasi-inverse $N \mapsto L \otimes_\cplx N$.

\begin{remark} \label{lowest-weight-remark} Let $\widetilde{\pi}_L : I \rightarrow GL_\cplx(L)$ be any lift of the projective representation $\pi_L$.  For any $n \in I$, the operator $\widetilde{\pi}_L(n)$ gives an isomorphism of $H_c(W', \hfr_{W'})$-modules of $L$ with its twist ${ }^nL$ and in particular preserves and is determined by its action on the lowest $\eu_{W'}$-weight space $L^0$ of $L$.  It follows that $\pi_L$ induces a projective representation $\pi_{L^0}$ of $I$ on $L^0$ with respect to which $L^0$ is (projectively) equivariant as a representation of $W'$.  Furthermore, the projective representation $\pi_L$ lifts to a linear representation of $I$, and in particular the cocycle $\mu$ is trivial, if and only if the same holds for the projective representation $\pi_{L^0}.$  Indeed, given a representation $\widetilde{\pi}_{L^0}$ of $I$ on $L^0$ with respect to which $L^0$ is an equivariant representation of $W'$, this representation extends uniquely to a representation of $I$ on all of $L$ with respect to which $L$ is $I$-equivariant as an $H_c(W', \hfr_{W'})$-module.  It follows that triviality of the cocycle $\mu$ can be checked at the level of the representation of $W'$ in $L^0$.  For example, the cocycle $\mu$ is trivial in the large class of examples in which $W'$ splits as a direct product $W' = W_1' \times W_2'$ with respect to which $L^0$ is isomorphic to a tensor product $L^0_1 \otimes L^0_2$ and such that both $I$ centralizes $W_1'$ and also $L^0_2$ is the trivial representation.\end{remark}

From the above discussion, we have:

\begin{theorem} \label{KZL} The $\hom$ functor $\hom_{\oscr_c(W, \hfr)}(\ind(L), \bullet)$ on $\oscr_{c, W', L}$ factors as the composition of a functor $$\widetilde{KZ_L} : \oscr_{c, W', L} \rightarrow \cplx_{-\tilde{\mu}}[\pi_1(\hfr^{W'}_{reg}/I)]\mhyphen\text{mod}_{f.d.}$$ followed by the forgetful functor to the category of finite-dimensional vector spaces.  $\widetilde{KZ_L}$ is exact, full, has image closed under subquotients, and induces a fully faithful embedding $$\bar{\oscr}_{c, W', L} \hookrightarrow \cplx_{-\tilde{\mu}}[\pi_1(\hfr^{W'}_{reg}/I)]\mhyphen\text{mod}_{f.d.}$$ with image closed under subquotients.  In particular, there is a surjection of $\cplx$-algebras $$\phi_L :\cplx_{-\tilde{\mu}}[\pi_1(\hfr^{W'}_{reg}/I)] \rightarrow \End_{\oscr_c(W, \hfr)}(\ind(L))^{opp}.$$  \end{theorem}

\begin{proof} By adjunction, we have $\hom_{\oscr_c(W, \hfr)}(\ind(L), \bullet) \cong \hom_{\oscr_c(W', \hfr_{W'})}(L, \res \bullet)$, and this latter functor can be expressed as the composition $\hom_{\End_\cplx(L)}(L, \bullet) \circ e_L \circ \fiber_b \circ \underline{\res}$ of functors discussed above, from which the existence of the lift $\widetilde{KZ_L}$ follows.  That $\widetilde{KZ_L}$ is exact, full, and has image closed under subquotients follows from Lemma \ref{full} and the fact that the functors appearing after $\underline{\res}$ in the composition above are equivalences of categories.   That the induced functor on $\bar{\oscr}_{c, W', L}$ is faithful follows from the fact that it is represented by the projective generator $\ind(L).$  By Proposition \ref{proj-gen}, this induced functor is identified with a fully faithful embedding $$\End_{\oscr_c(W, \hfr)}(\ind(L))^{opp}\mhyphen\text{mod}_{f.d.} \rightarrow \cplx_{-\tilde{\mu}}[\pi_1(\hfr^{W'}_{reg}/I)]\mhyphen\text{mod}_{f.d.}$$ with image closed under subquotients and which is identity at the level of $\cplx$-vector spaces.  Such a functor is simply restriction along some algebra homomorphism $\phi_L : \cplx_{-\tilde{\mu}}[\pi_1(\hfr^{W'}_{reg}/I)] \rightarrow \End_{\oscr_c(W, \hfr)}(\ind(L))^{opp}.$  The image $\im \phi_L$ is a $\cplx_{-\tilde{\mu}}[\pi_1(\hfr^{W'}_{reg}/I)]$-submodule of $\End_{\oscr_c(W, \hfr)}(\ind(L))^{opp}$, and because the functor in question has image closed under subquotients it follows that $\im \phi_L$ is also a $\End_{\oscr_c(W, \hfr)}(\ind(L))^{opp}$-submodule.  As $\im \phi_L$ contains 1, it follows that $\phi_L$ is surjective, as needed.\end{proof}

\begin{definition} Let the \emph{generalized Hecke algebra} $\mathcal{H}(c, W', L, W)$ attached to $c, W', L$ and the group $W$ be the quotient algebra $$\mathcal{H}(c, W', L, W) := \cplx_{-\tilde{\mu}}[\pi_1(\hfr^{W'}_{reg}/I)]/\ker \phi_L.$$  Let $$KZ_L : \oscr_{c, W', L} \rightarrow \mathcal{H}(c, W', L, W)\mhyphen\text{mod}_{f.d.}$$ be the natural factorization of the functor $\widetilde{KZ_L}$ from Theorem \ref{KZL} through $\mathcal{H}(c, W', L, W)\mhyphen\text{mod}_{f.d.}$.\end{definition}

\begin{corollary} \label{quotient-iso} $KZ_L$ induces an equivalence of categories $$\bar{\oscr}_{c, W', L} \cong \mathcal{H}(c, W', L, W)\mhyphen\text{mod}_{f.d.}.$$\end{corollary}

\begin{remark} In the case $W' = 1$ and $L = \cplx$, the generalized Hecke algebra $\mathcal{H}(c, W', L, W)$ is precisely the Hecke algebra $\mathsf{H}_q(W)$ attached to the complex reflection group $W$, the functor $KZ_\cplx$ is precisely the $KZ$ functor of \cite{GGOR}, and the statements of Theorem \ref{KZL} are well-known in that case.  The equivalence in Corollary \ref{quotient-iso} in that case is the well-known equivalence $\oscr_c/\oscr_c^{tor} \cong \mathsf{H}_q(W)\mhyphen\text{mod}_{f.d.}.$\end{remark}

The following remark reduces the study of the algebras $\mathcal{H}(c, W', L, W)$ to the case in which $W$ is an irreducible complex reflection group.

\begin{remark} \label{reducible-reduction-prop} Suppose the complex reflection group $W$ and its reflection representation decompose as a product $(W, \hfr) = (W_1 \times W_2, \hfr_{W_1} \oplus \hfr_{W_2})$.  Let $S_i \subset W_i$ be the set of reflections in $W_i$ for $i = 1,2$, so that $S = S_1 \sqcup S_2$, and let $c_i$ be the restriction of the parameter $c$ to $S_i$.  The rational Cherednik algebra $H_c(W, \hfr)$ decomposes naturally as the tensor product $H_{c_1}(W_1, \hfr_{W_1}) \otimes H_{c_2}(W_2, \hfr_{W_2}).$  Let $W' \subset W$ be a parabolic subgroup, and for $i = 1,2$ let $W_i' \subset W_i$ be the parabolic subgroups such that $W' = W_1' \times W_2'$.  Let $L_i$ be a finite-dimensional irreducible representation of $H_{c_i}(W_i', \hfr_{W_i'})$ for $i = 1,2$, so that $L = L_1 \otimes L_2$ is a finite-dimensional irreducible representation of $H_c(W_1' \times W_2', \hfr_{W_1' \times W_2'}) = H_{c_1}(W_1', \hfr_{W_1'}) \otimes H_{c_2}(W_2', \hfr_{W_2'}).$  Every finite-dimensional irreducible representation of $H_c(W_1' \times W_2', \hfr_{W_1'} \times \hfr_{W_2'})$ appears in this way, and all of the constructions appearing in Theorem \ref{KZL} split naturally as well.  In particular, there is a natural isomorphism of algebras $$\mathcal{H}(c, W', L, W) \cong \mathcal{H}(c_1, W_1', L_1, W_1) \otimes \mathcal{H}(c_2, W_2', L_2, W_2)$$ compatible with the surjections $\phi_{L_1}, \phi_{L_2}$ and $\phi_{L_1 \otimes L_2}$.\end{remark}

\subsection{Eigenvalues of Monodromy and Relations From Corank 1}  We want to describe the generalized Hecke algebras $\mathcal{H}(c, W', L, W)$ as explicitly as possible so that we can in turn understand the subquotient categories $\bar{\oscr}_{c, W', L}.$  To achieve this, we study the kernel $\ker \phi_L$.

Let $W'' \subset W$ be another parabolic subgroup, with $W' \subset W'' \subset W$.  Let $I_{W', L, W''} := I_{W', L} \cap W''$ be the (lift to $N_{W''}(W)$ via the splitting $N_W(W') = W' \rtimes N_{W'}$ of the) inertia group of $L$ in $N_{W''}(W')$.  The 2-cocycle $\mu \in Z^2(I_{W', L}, \cplx^\times)$ restricts to give a 2-cocycle of $I_{W', L, W''}$, which we will also denote by $\mu$, and by pullback determines a 2-cocycle of $\pi_1((\hfr_{W''})^{W'}_{reg}/I_{W', L, W''})$, which we will also denote by $\tilde{\mu}$.    We may choose $b \in \hfr^{W'}_{reg}$ such that the projection $b_{W''}$ of $b$ to $(\hfr_{W''})^{W'}$ with respect to the vector space decomposition $\hfr^{W'} = (\hfr_{W''})^{W'} \oplus \hfr^{W''}$ lies in $(\hfr_{W''})^{W'}_{reg}.$  By Theorem \ref{KZL}, there is a surjection of $\cplx$-algebras $$\phi_{L, W''} : \cplx_{-\tilde{\mu}}[\pi_1((\hfr_{W''})^{W'}_{reg}/I_{W', L, W''}, b_{W''})] \rightarrow \End_{\oscr_c(W'', \hfr_{W''})}(\ind_{W'}^{W''}L)^{opp}.$$

\noindent $\mathbf{Notation}$: Throughout this section, the parabolic subgroup $W'$ and finite-dimensional irreducible representation $L$ of $H_c(W', \hfr_{W'})$ will remain fixed, while the parabolic subgroup $W''$ will vary.  We will denote $I_{W', L}$ by $I$, as in the previous section, and we will denote $I_{W', L, W''}$ by $I_{W''}$.  We will always take the basepoints for the fundamental groups of $\hfr^{W'}_{reg}$ and $(\hfr_{W''})^{W'}_{reg}$ and their quotients to be $b$ and $b_{W''}$ as above, and we will suppress these from the notation for readability.

By the construction of the functor $\iota^*_{W'', W'}$ of Gordon and Martino recalled in Section \ref{transitivity} and the fact that $I_{W''}$ acts trivially on $\hfr^{W''}_{reg}$, there is a natural map of fundamental groups $$\iota_{W'', W', L} : \pi_1((\hfr_{W''})^{W'}_{reg}/I_{W''}) \rightarrow \pi_1(\hfr^{W'}_{reg}/I),$$ extending to a natural map of $\cplx$-algebras $$\iota_{W'', W', L} : \cplx_{-\tilde{\mu}}[\pi_1((\hfr_{W''})^{W'}_{reg}/I_{W''})] \rightarrow \cplx_{-\tilde{\mu}}[\pi_1(\hfr^{W'}_{reg}/I)].$$  By functoriality, there is also an algebra homomorphism $$\End_{\oscr_c(W'', \hfr_{W''})}(\ind_{W'}^{W''}L)^{opp} \rightarrow \End_{\oscr_c(W, \hfr)}(\ind_{W'}^{W}L)^{opp}$$ induced by the functor $\ind_{W''}^W$ and the isomorphism $\ind_{W''}^W \circ \ind_{W'}^{W''} \cong \ind_{W'}^W.$  The following lemma is then immediate from the constructions and from Gordon and Martino's transitivity result that was recalled in Theorem \ref{trans-theorem}:

\begin{lemma} \label{compatibility} The following diagram commutes: $$\begin{diagram}  \cplx_{-\tilde{\mu}}[\pi_1((\hfr_{W''})^{W'}_{reg}/I_{W''})] & \rTo^{\iota_{W'', W', L}} & \cplx_{-\tilde{\mu}}[\pi_1(\hfr^{W'}_{reg}/I)] \\ \dTo^{\phi_{L, W''}} && \dTo^{\phi_L} \\ \End_{\oscr_c(W'', \hfr_{W''})}(\ind_{W'}^{W''}L)^{opp} & \rTo^{\ind_{W''}^W} & \End_{\oscr_c(W, \hfr)}(\ind_{W'}^{W}L)^{opp}\end{diagram}$$\end{lemma}

\begin{definition} Given a chain of parabolic subgroups $W' \subset W'' \subset W$ as above, we say that $W'$ is of \emph{corank} $r$ in $W''$ if $\dim \hfr_{W''} = \dim \hfr_{W'} + r.$\end{definition}

The parabolic subgroups $W'' \subset W$ containing $W'$ in corank 1 are closely related to the reflections appearing in the action of $N_W(W')$ on $\hfr^{W'}$:

\begin{lemma} \label{corank-1-refs} Let $W'' \subset W$ be a parabolic subgroup containing $W'$ in corank 1.  The quotient $N_{W''}(W')/W'$ is cyclic and acts on $\hfr^{W'}$ by complex reflections through the hyperplane $\hfr^{W''} \subset \hfr^{W'}$.  Furthermore, every element $n \in N_W(W')$ that acts on $\hfr^{W'}$ as a complex reflection lies in some parabolic subgroup $W'' \subset W$ containing $W'$ in corank 1.\end{lemma}

\begin{proof} The action of $N_{W''}(W')/W'$ on $\hfr^{W'}$ is faithful because $W'$ is precisely the pointwise stabilizer of $\hfr^{W'}$ in $W$, and the action respects the decomposition $\hfr^{W'} = (\hfr_{W''})^{W'} \oplus \hfr^{W''}$.  As $N_{W''}(W')/W'$ acts trivially on $\hfr^{W''}$ and $\dim (\hfr_{W''})^{W'} = 1$, the first claim follows.

For the second claim, suppose $n \in N_W(W')$ acts on $\hfr^{W'}$ as a complex reflection through the hyperplane $H \subset \hfr^{W'}$.  Let $W'' \subset W$ be the point-wise stabilizer of $H$ in $W$, a parabolic subgroup.  As $n$ does not fix $\hfr^{W'}$ but $W' \subset W''$, we have $H \subset \hfr^{W''} \subsetneq \hfr^{W'}$.  As $H$ is a hyperplane, it follows that $H = \hfr^{W''}$, so $W''$ is a parabolic subgroup containing $W'$ in corank 1 and $n \in N_{W''}(W)$ acts on $\hfr^{W'}$ as a complex reflection through the hyperplane $\hfr^{W''} \subset \hfr^{W'}$, as needed.\end{proof}

It follows from Lemma \ref{corank-1-refs} that the inertia group $I_{W''}$ considered above is cyclic and acts on $(\hfr_{W''})^{W'}_{reg}$ through a faithful character.  As $W'$ is corank 1 in $W''$, we have $(\hfr_{W''})^{W'}_{reg} \cong \cplx^\times$ as a $\cplx$-manifold, and in particular there is a canonical group isomorphism $$\pi_1((\hfr_{W''})^{W'}_{reg}/I_{W''}) \cong \ints.$$

\begin{definition} Let $T_{W', L, W''}$ denote the canonical generator of $\pi_1((\hfr_{W''})^{W'}_{reg}/I_{W''})$ arising from the isomorphism above.  We will also let $T_{W', L, W''}$ denote its image in $\pi_1(\hfr^{W'}_{reg}/I)$ under the homomorphism $\iota_{W'', W', L}$, and we refer to $T_{W', L, W''}$ as the \emph{generator of monodromy} about the hyperplane $\hfr^{W''} \subset \hfr^{W'}$.  We call the parabolic subgroup $W''$ and its associated hyperplane $\hfr^{W''} \subset \hfr^{W'}$ $L$-\emph{trivial} (resp., $L$-\emph{essential}) when the inertia group $I_{W''}$ is trivial (resp., nontrivial).  Let $\hfr^{W'}_{L-reg}$ denote the complement of the arrangement of $L$-essential hyperplanes in $\hfr^{W'}$, and let $I^{ref}$ denote the subgroup of $I$ generated by the subgroups $I_{W''}$ for all $L$-essential $W''$.\end{definition}

We comment that $I^{ref}$ is a normal subgroup of $I$ and that the action of $I^{ref}$ on $\hfr^{W'}$ is generated by complex reflections through the $L$-essential hyperplanes.  In fact, from Lemma \ref{corank-1-refs} we see that $I^{ref}$ is the maximal reflection subgroup of $I$ with respect to its representation in $\hfr^{W'}$.  We have $\hfr^{W'}_{reg} \subset \hfr^{W'}_{L-reg}$, and $\hfr^{W'}_{L-reg}$ is stable under the action of $I$ on $\hfr^{W'}$.

The second cohomology group $H^2(C, \cplx^\times)$ vanishes for all cyclic groups $C$, and in particular we may assume that the 2-cocycle $\mu$ on $I$ is trivial on all subgroups $I_{W''} \subset I$ associated to parabolic subgroups $W'' \subset W$ containing $W'$ in corank 1.  Then, the twisted group algebra $\cplx_{-\tilde{\mu}}[\pi_1((\hfr_{W''})^{W'}_{reg}/I_{W''})]$ is naturally identified with the Laurent polynomial ring $\cplx[T_{W', L, W''}^{\pm 1}].$  The kernel of the map $\phi_{L, W''}$ appearing in the commutative diagram in Lemma \ref{compatibility} is therefore generated by a monic polynomial $P_{W', L, W''} \in \cplx[T_{W', L, W''}]$ nonvanishing at 0.  It follows from Proposition \ref{dim-prop} that $P_{W', L, W''}$ is a polynomial of degree $\#I_{W''}$.

\begin{definition} \label{rel-def} Let $P_{W', L, W''} \in \cplx[T]$ be the monic polynomial of degree $\#I_{W''}$ generating the kernel $\ker \phi_{W', L, W''}$.\end{definition}

Note that $P_{W', L, W''}$ only depends on $W''$ up to $I$-conjugacy and that $P_{W', L, W''}(0) \neq 0$.

\noindent $\mathbf{Notation}$: As $W'$ and $L$ remain fixed, we will denote $P_{W', L, W''}$ by $P_{W''}$ and $T_{W', L, W''}$ by $T_{W''}$ when the meaning is clear.

We can now state the main result of this section.

\begin{theorem} \label{general-presentation-theorem} The map $\phi_L$ of Theorem \ref{KZL} factors through a surjection $$\bar{\phi}_L : \frac{\cplx_{-\tilde{\mu}}[\pi_1(\hfr^{W'}_{reg}/I)]}{\langle P_{W''}(T_{W''}) : \text{ $W'$ is corank 1 in $W''$}\rangle} \rightarrow \mathcal{H}(c, W', L, W)$$ of $\cplx$-algebras.  If the inequality $$\dim \frac{\cplx_{-\tilde{\mu}}[\pi_1(\hfr^{W'}_{L-reg}/I^{ref})]}{\langle P_{W''}(T_{W''}) : W'' \text{ is $L$-essential}\rangle} \leq \#I^{ref}$$ holds then it is an equality and $\bar{\phi}_L$ is an isomorphism.

In particular, if $W$ is a Coxeter group or if the group 2-cocycle $\mu \in Z^2(I, \cplx^\times)$ has cohomologically trivial restriction to $I^{ref}$, $\bar{\phi}_L$ is an isomorphism.\end{theorem}

\begin{remark} Note that only the $L$-essential hyperplanes feature in dimension bound above.  The 2-cocycle $\tilde{\mu} \in Z^2(\pi_1(\hfr^{W'}_{L-reg}/I^{ref}), \cplx^\times)$ is obtained by pulling back the 2-cocycle $\mu \in Z^2(I, \cplx^\times)$ along the natural projection.  Note that by their definition, the generators of monodromy $T_{W''}$ are naturally elements of $\pi_1(\hfr^{W'}_{reg}/I^{ref})$.\end{remark}

The following lemma will be useful in the proof of Theorem \ref{general-presentation-theorem}

\begin{lemma} \label{twisting-character} There is a linear character $\chi : \pi_1(\hfr^{W'}_{reg}/I) \rightarrow \cplx^\times$ satisfying $$\chi(T_{W''}) = \begin{cases} -P_{W''}(0) & \text{ if $W''$ is $L$-trivial} \\ 1 & \text{ otherwise.}\end{cases}$$ \end{lemma}

\begin{proof} As $P_{W''}(0)$ is nonzero and only depends on $W''$ up to $I$-conjugacy, it suffices to show that for any parabolic subgroup $W''$ containing $W'$ in corank 1 there exists a group homomorphism $\theta : \pi_1(\hfr^{W'}_{reg}/I) \rightarrow \ints$ such that a generator of monodromy $T_H$ about one of the hyperplanes $H$ defining $\hfr^{W'}_{reg}/I$ is sent to 1 under $\theta$ if and only if $H$ is $I$-conjugate to $\hfr^{W''}$ and 0 otherwise.  One may then compose with the group homomorphism $\ints \rightarrow \cplx^\times$, $1 \mapsto -P_{W''}(0)$, and take the product of such maps over all $I$-conjugacy classes of $L$-trivial parabolic subgroups $W'' \supset W'$.  To construct such a homomorphism, choose a linear functional $\alpha \in (\hfr^{W'})^*$ defining $\hfr^{W''}$ in $\hfr^{W'}$ and let $\delta = \prod_{n \in I} n\alpha$.   Then, $\delta$ defines a continuous function $\delta : \hfr^{W'}_{reg}/I \rightarrow \cplx^\times$ with associated map $\pi_1(\delta) : \pi_1(\hfr^{W'}_{reg}/I) \rightarrow \pi_1(\cplx^\times) = \ints.$  That $\pi_1(\delta)$ has the desired effect on the generators of monodromy then follows easily as in the proof of \cite[Proposition 2.16]{BMR}.\end{proof}

\begin{proof}[Proof of Theorem \ref{general-presentation-theorem}] That $\phi_L$ factors as $\bar{\phi}_L$ through the quotient above follows immediately from Lemma \ref{compatibility}.

Assume the dimension inequality in the theorem statement holds.  As the action of $I^{ref}$ on $\hfr^{W'}$ is generated by reflections, the quotient $\hfr^{W'}/I^{ref}$ is a smooth $\cplx$-variety by the Chevalley-Shephard-Todd theorem \cite{Chev, ST}.  In particular, by Proposition A.1 of \cite{BMR} and induction it follows that the kernel of the map $\pi_1(\hfr^{W'}_{reg}/I^{ref}) \rightarrow \pi_1(\hfr^{W'}_{L-reg}/I^{ref})$ is generated (as a normal subgroup) by the generators of monodromy $T_{W''}$ for the $L$-trivial $W''$.  In particular, we have an isomorphism $$\frac{\pi_1(\hfr^{W'}_{reg}/I^{ref})}{\langle T_{W''} : W'' \text{ is $L$-trivial}\rangle} \cong \pi_1(\hfr^{W'}_{L-reg}/I^{ref}).$$  Clearly, this isomorphism respects the cocycle $\tilde{\mu}$.  Let $\chi : \pi_1(\hfr^{W'}_{reg}/I^{ref}) \rightarrow \cplx^\times$ be the composition of the natural map $\pi_1(\hfr^{W'}_{reg}/I^{ref}) \rightarrow \pi_1(\hfr^{W'}_{reg}/I)$ with the character from Lemma \ref{twisting-character}.  The assignments $g \mapsto \chi(g)g$ for $g \in \pi_1(\hfr^{W'}_{reg}/I)$ extend to an algebra automorphism $\tau_\chi$ of $\cplx_{-\tilde{\mu}}[\pi_1(\hfr^{W'}_{reg}/I)]$ satisfying $$\tau_{\chi}(P_{W''}(T_{W''})) = \begin{cases} -P_{W''}(0)(T_{W''} - 1)& \text{ if $W''$ is $L$-trivial} \\ P_{W''}(T_{W''}) & \text{ otherwise.}\end{cases}.$$  To see this formula in the case that $W''$ is $L$-trivial, recall that in that case $P_{W''}(T_{W''})$ is a monic linear polynomial, hence of the form $P_{W''}(T_{W''}) = T_{W''} + P_{W''}(0)$; applying $\tau_\chi$ gives $$\tau_\chi(P_{W''}(T_{W''})) = \tau_\chi(T_{W''} + P_{W''}(0))$$ $$= -P_{W''}(0)T_{W''} + P_{W''}(0) = -P_{W''}(0)(T_{W''} - 1),$$ as above.  The formula in the case that $W''$ is $L$-essential is clear.  Composing these two isomorphisms yields an isomorphism $$\frac{\cplx_{-\tilde{\mu}}[\pi_1(\hfr^{W'}_{L-reg}/I^{ref})]}{\langle P_{W''}(T_{W''}) : W'' \text{ is $L$-essential}\rangle} \cong \frac{\cplx_{-\tilde{\mu}}[\pi_1(\hfr^{W'}_{reg}/I^{ref})]}{\langle P_{W''}(T_{W''}) : \text{ $W'$ is corank 1 in $W''$}\rangle}.$$  By assumption, this algebra has dimension at most $\#I^{ref}$, and so the same holds for the dimension $d \leq \#I^{ref}$ of its image in $$\frac{\cplx_{-\tilde{\mu}}[\pi_1(\hfr^{W'}_{reg}/I)]}{\langle P_{W''}(T_{W''}) : \text{ $W'$ is corank 1 in $W''$}\rangle}$$ under the natural map induced by the map $\pi_1(\hfr^{W'}_{reg}/I^{ref}) \rightarrow \pi_1(\hfr^{W'}_{reg}/I)$ of fundamental groups.  As $\pi_1(\hfr^{W'}_{reg}/I)/\pi_1(\hfr^{W'}_{reg}/I^{ref}) \cong I/I^{ref},$ it follows that $$\dim \frac{\cplx_{-\tilde{\mu}}[\pi_1(\hfr^{W'}_{reg}/I)]}{\langle P_{W''}(T_{W''}) : \text{ $W'$ is corank 1 in $W''$}\rangle} \leq d\cdot(\#I/I^{ref}) \leq \#I.$$  But, as $\bar{\phi}_L$ is a surjection and $\dim \mathcal{H}(c, W', L, W) = \#I$ by Proposition \ref{dim-prop}, it follows also that $$\#I \leq \dim \frac{\cplx_{-\tilde{\mu}}[\pi_1(\hfr^{W'}_{reg}/I)]}{\langle P_{W''}(T_{W''}) : \text{ $W'$ is corank 1 in $W''$}\rangle}.$$  It follows that the dimension inequalities above are equalities and that $\bar{\phi_L}$ is an isomorphism.

For the final statement of the theorem, first suppose the restriction of $\mu$ to $I^{ref}$ is cohomologically trivial.  Then the quotient $$\frac{\cplx[\pi_1(\hfr^{W'}_{L-reg}/I^{ref})]}{\langle P_{W''}(T_{W''}) : W'' \text{ is $L$-essential}\rangle}$$ is precisely a specialization to $\cplx$ of the Hecke algebra attached to the complex reflection group $I^{ref}$ in the sense of Brou\'{e}-Malle-Rouquier \cite{BMR}.  In that paper it was conjectured that the generic Hecke algebra is a free module of rank $\#I^{ref}$ over its ring of parameters, and this conjecture was subsequently proved in characteristic zero \cite{Et-BMR}. In particular, this algebra has dimension $\#I^{ref}$, as needed.

Now suppose $W$ is a Coxeter group.  In that case, $I^{ref}$ is a Coxeter group as well, and its action on $\hfr^{W'}$ is the complexification of a real reflection representation.  We need to show that the algebra $$\frac{\cplx_{-\tilde{\mu}}[\pi_1(\hfr^{W'}_{L-reg}/I^{ref})]}{\langle P_{W''}(T_{W''}) : W'' \text{ is $L$-essential}\rangle}$$ has dimension at most, and hence equal to, $\#I^{ref}$.  In this case, the fundamental group $\pi_1(\hfr^{W'}_{L-reg}/I^{ref})$ is the Artin braid group $B_{I^{ref}}$ attached to the Coxeter group $I^{ref}$.  If $S \subset I^{ref}$ is a set of simple reflections for this Coxeter group, the group $B_{I^{ref}}$ is generated by the generators of monodromy $\{e_s : s \in S\}$ about the hyperplanes associated to the simple reflections, with braid relations as recalled in Section \ref{Mackey-Coxeter}.  Given a finite list $\bf{s} = (s_1, ..., s_n)$ of simple reflections, let $e_{\bf{s}}$ denote the product $e_{s_1}\cdots e_{s_n}$.  By Matsumoto's Theorem (see \cite[Theorem 1.2.2]{GP}), for any $w \in I^{ref}$ the element $e_{\bf{s}}$ does not depend on the choice of reduced expression $w = s_1\dots s_n$, and we denote as usual the element $e_{\bf{s}}$ by $e_w$ in this case.  For a list $\bf{s}$ let $T_{\bf{s}}$ denote the element $e_{\bf{s}}$ viewed as an element of the quotient algebra above, and similarly for $T_w$.  To show that the algebra has dimension at most $\#I^{ref}$, it suffices to check that the $\cplx$-span of the elements $\{T_w : w \in I^{ref}\}$ is closed under multiplication by $T_s$ for any $s \in S$.  If $ss_1\cdots s_n$ is a reduced expression with $w = s_1 \cdots s_n$ a reduced expression for $w$, then $T_sT_w = \mu(s, w)^{-1}T_{sw}$.  Otherwise, we may choose a reduced expression $w = s_1\cdots s_n$ for $w$ with $s_1 = s$.  The hyperplane associated to $s$ is $L$-essential and $\#I_{W''} = 2$ for all $L$-essential $W''$ in this case, so $T_s$ satisfies a quadratic relation $T_s^2 = a + bT_s$.  We then have $T_sT_w = \mu(s, sw)T_s^2T_{sw} = \mu(s, sw)(a + be_s)T_{sw} = \mu(s, sw)(aT_{sw} + \mu(s, sw)^{-1}bT_w) = \mu(s, sw)aT_{sw} + bT_w$, which lies in the desired space, as needed. \end{proof}

\section{The Coxeter Group Case} \label{coxeter-section} In this section, suppose $W$ is a finite Coxeter group with set of simple reflections $S \subset W$ and real reflection representation $\hfr_\real$ with inner product $\langle \cdot, \cdot \rangle$.  Let $\hfr := \cplx \otimes_\real \hfr_\real$ be the complexified reflection representation, let $c : S \rightarrow \cplx$ be a $W$-invariant function (by which we mean $c(s) = c(s')$ whenever $s$ and $s'$ are conjugate in $W$), and let $H_c(W, \hfr)$ be the associated rational Cherednik algebra (note that $c$ extends uniquely to a $W$-invariant function on the set of reflections in $W$).  We will use certain simplifications that arise in the Coxeter case, such as natural splittings of the normalizers of parabolic subgroups and the $T_w$-basis of the Hecke algebra $\mathsf{H}_q(W)$, to significantly improve upon the description of the generalized Hecke algebras $\mathcal{H}(c, W', L, W'')$ studied in the previous section.  In particular, we will give a natural presentation for these algebras and will explain how to compute the quadratic relations $P_{W', L, W''}(T_{W', L, W''}) = 0$ for the $L$-essential parabolic subgroups $W'' \supset W'$ introduced in Definition \ref{rel-def}.

\noindent $\mathbf{Notation}$:  As in previous sections, we will take $W'$ and $L$ to be fixed and will suppress them from the notation, e.g. write $I_{W''}$ rather than $I_{W', L}$, when the meaning is clear.

\subsection{A Presentation of the Endomorphism Algebra}  Let $J \subset S$ be a subset of the simple reflections and let $W' := \langle J \rangle$ be the parabolic subgroup of $W$ generated by $J$.  Let $H_c(W', \hfr_{W'})$ be the associated rational Cherednik algebra as considered in previous sections, and let $L$ be a finite-dimensional irreducible representation of $H_c(W', \hfr_{W'})$.  Let $\hfr_{W', \real} \subset \hfr_\real$ denote the unique $W'$-stable complement to $\hfr_\real^{W'}$ in $\hfr_{\real}$.  We have $\hfr^{W'} = (\hfr^{W'}_\real)_\cplx$ and $\hfr_{W'} = (\hfr_{W', \real})_\cplx$, where the subscript $\cplx$ denotes the complexification viewed as a subspace of $\hfr$.  The action of $W'$ on $\hfr_{W', \real}$ is naturally identified with the real reflection representation of the Coxeter system $(W', J)$.  Let $\hfr^{W'}_{\real, reg} \subset \hfr^{W'}_{\real}$ denote the subset of points with stabilizer in $W$ equal to $W'$.  The complexification $(\hfr^{W'}_{\real, reg})_\cplx$ is a proper subspace of $\hfr^{W'}_{reg}$.

Let $\Phi \subset \hfr_\real$ denote the root system attached to $(W, S)$, let $\Phi^+ \subset \Phi$ denote the set of positive roots, and for $s \in S$ let $\alpha_s \in \Phi$ denote the positive simple root defining the reflection $s$.  For a subset $J \subset S$ of the simple roots, let $\Phi_J \subset \Phi$ denote the associated root subsystem with positive roots $\Phi_J^+ = \Phi_J \cap \Phi^+$.  Let $\mathcal{C}_J := \{x \in \hfr_{W', \real} : \langle \alpha_s, x \rangle > 0 \text{ for all } s \in J\} \subset \hfr_{W', \real}$ be the associated open fundamental Weyl chamber.  In \cite[Corollary 3]{How}, Howlett explains that the normalizer $N_W(W')$ splits as a semidirect product $N_W(W') = W' \rtimes N_{W'}$, where $N_{W'}$ is the set-wise stabilizer of $\mathcal{C}_J$ in $N_W(W')$.  We then take the inertia subgroup $I \subset N_{W'}$ to be the stabilizer in $N_{W'}$ of the representation $L$, as defined after Lemma \ref{split}.  Recall that for each parabolic subgroup $W'' \subset W$ containing $W'$ in corank 1 we have the associated subgroup $I_{W''} := I \cap W''$, and recall that $I^{ref} \unlhd I$ is the normal subgroup generated by the $I_{W''}$ for all such $W''$.  The subgroups $I_{W''}$ are all either trivial (for $L$-trivial $W''$) or of order 2 (for $L$-essential $W''$) acting on $\hfr^{W'}_\real$ through orthogonal reflection through the hyperplane $\hfr^{W''}_\real \subset \hfr^{W'}_\real$.  Therefore $I^{ref}$ is a real reflection group with faithful reflection representation $\hfr^{W'}_{\real}$ (potentially with nontrivial fixed points).  Recall that $I^{ref}$ is the maximal reflection subgroup of $I$.

The complement $\hfr^{W'}_{\real, L-reg}$ of the real hyperplanes $\hfr^{W''}_\real \subset \hfr^{W'}_\real$ is the locus of points in $\hfr^{W'}_\real$ with trivial stabilizer in $I^{ref}$.  By the standard theory of real reflection groups, $I^{ref}$ acts simply transitively on the set of connected components of $\hfr^{W'}_{\real, L-reg}$.  Choose a connected component $\mathcal{C} \subset \hfr^{W'}_{\real, L-reg}$, and let $I^{comp} \subset I$ be the set-wise stabilizer of $\mathcal{C}$ in $I$.  Clearly, we then have the semidirect product decomposition $$I = I^{ref} \rtimes I^{comp}.$$  Choosing the connected component $\mathcal{C}$ amounts to choosing a fundamental Weyl chamber for $I^{ref}$, and in particular the pair $(I^{ref}, S_{W', L})$ is a Coxeter system, where $S_{W', L}$ is the set of reflections through the walls of $\mathcal{C}$.  As $I^{comp}$ acts on $\mathcal{C}$, it follows that $I^{comp}$ permutes the walls of $C$, and in particular the action of $I^{comp}$ on $I^{ref}$ is through diagram automorphisms of the Dynkin diagram of $(I^{ref}, S_{W', L})$.

Let $q : S_{W', L} \rightarrow \cplx^\times$, $s \mapsto q_s$, be the unique $I$-invariant function such that for each $s \in S_{W', L}$ the quadratic polynomial $P_{W', L, \langle W', s\rangle}$ factors as $(T - 1)(T + q_s)$ after a rescaling of the indeterminate $T$.  Let $l : I^{ref} \rightarrow \ints^{\geq 0}$ denote the length function of the Coxeter system $(I^{ref}, S_{W', L})$.  We then have the following presentation for the generalized Hecke algebra $\mathcal{H}(c, W', L, W'')$, analogous to the presentation given in \cite[Theorem 4.14]{HoLe} in the context of finite groups of Lie type:

\begin{theorem} \label{coxeter-presentation} Let $\mu \in Z^2(I, \cplx^\times)$ be the 2-cocycle appearing in Theorem \ref{KZL}.  There is a basis $\{T_x : x \in I\}$ for $\mathcal{H}(c, W', L, W'')$ with multiplication law completely described by the following relations for all $x \in I$, $d \in I^{comp}$, $w \in I^{ref}$, and $s \in S_{W', L}$:

(1) \ \ $T_dT_x = \mu(d, x)^{-1}T_{dx}$

(2) \ \ $T_xT_d = \mu(x, d)^{-1}T_{xd}$

(3) \ \ $T_sT_w = \begin{cases} \mu(s, w)^{-1}T_{sw} & \text{ if $l(sw) > l(w)$} \\ q_s\mu(s, w)^{-1}T_{sw} + (q_s - 1)T_w & \text{ if $l(sw) < l(w)$}\end{cases}$

(4) \ \ $T_wT_s = \begin{cases} \mu(w, s)^{-1}T_{ws} & \text{ if $l(ws) > l(w)$} \\ q_s\mu(w, s)^{-1}T_{ws} + (q_s - 1)T_w & \text{ if $l(ws) < l(w).$}\end{cases}$\end{theorem}

\begin{remark} When the cocycle $\mu$ is trivial, Theorem \ref{coxeter-presentation} gives an isomorphism $$\mathcal{H}(c, W', L, W'') \cong I^{comp} \ltimes \mathsf{H}_q(I^{ref})$$ where $\mathsf{H}_q(I^{ref})$ is the Iwahori-Hecke algebra attached to the Coxeter system $(I^{ref}, S_{W', L})$ with parameter $q : s \mapsto q_s$ and where $I^{comp}$ acts on $\mathsf{H}_q(I^{ref})$ by automorphisms induced by diagram automorphisms of the Dynkin diagram of $(I^{ref}, S_{W', L})$.  We will see in a later section, by inspecting all possible cases, that the cocycle $\mu$ is indeed trivial in every case.\end{remark}

\begin{proof}  Take the base point $b \in \hfr^{W'}_{reg}$ for the fundamental group to lie in the fundamental chamber $\mathcal{C} \subset \hfr^{W'}_{\real, L-reg}$ and outside of the hyperplanes $\hfr^{W''} \subset \hfr^{W'}$ for the $L$-trivial parabolic subgroups $W'' \supset W'$.  As $I$ acts freely on $\hfr^{W'}_{reg}$, so too $I^{comp}$ acts freely on the orbit $I^{comp}.b \subset \mathcal{C} \cap \hfr^{W'}_{reg}$.  For each $d \in I^{comp}$, choose a path $\gamma_d : [0, 1] \rightarrow \mathcal{C} \cap \hfr^{W'}_{reg}$ in $\mathcal{C} \cap \hfr^{W'}_{reg}$ from $b$ to $d.b$.  Let $\bar{\gamma}_d$ denote the reverse path, so that the concatenation $\bar{\gamma}_d * \gamma_d$ is a loop in $\mathcal{C} \cap \hfr^{W'}_{reg}$ with base point $b$.  As $\mathcal{C}$ is contractible, the image of the path homotopy class of $\bar{\gamma}_d * \gamma_d$ under the homomorphism $$\pi_1(\hfr^{W'}_{reg}) \rightarrow \pi_1(\hfr^{W'}_{L-reg})$$ induced by the inclusion $\hfr^{W'}_{reg} \subset \hfr^{W'}_{L-reg}$ is trivial.  As discussed in the proof of Theorem \ref{general-presentation-theorem}, the kernel of this map is generated (as a normal subgroup) by the generators of monodromy $T_{W''}$ about the hyperplanes $\hfr^{W''}$ for the $L$-trivial parabolic subgroups $W'' \supset W'$.  It follows immediately that the image of the homotopy class $[\gamma_d]$ in the quotient $$\frac{\pi_1(\hfr^{W'}_{reg}/I)}{\langle T_{W''} : W'' \text{ is $L$-trivial}\rangle}$$ is independent of the choice of the path $\gamma_d$ and that the assignments $d \mapsto [\gamma_d]$ for $d \in I^{comp}$ extends uniquely to an algebra homomorphism $$\cplx_{-\mu}[I^{comp}] \rightarrow \frac{\cplx_{-\tilde{\mu}}[\pi_1(\hfr^{W'}_{reg}/I)]}{\langle T_{W''} : W'' \text{ is $L$-trivial}\rangle}.$$  Composing with the automorphism $\tau_\chi$ of $\cplx_{-\tilde{\mu}}[\pi_1(\hfr^{W'}_{reg}/I)]$ discussed in the proof of Theorem \ref{general-presentation-theorem}, this yields an algebra homomorphism $$\cplx_{-\mu}[I^{comp}] \rightarrow \frac{\cplx_{-\tilde{\mu}}[\pi_1(\hfr^{W'}_{reg}/I)]}{\langle P_{W''}(T_{W''}) : \text{$W'$ is corank 1 in $W''$}\rangle}.$$  We also have the embedding $$\frac{\cplx[\pi_1(\hfr^{W'}_{L-reg}/I^{ref})]}{\langle P_{W''}(T_{W''}) : W'' \text{ is $L$-essential}\rangle} \rightarrow \frac{\cplx_{-\tilde{\mu}}[\pi_1(\hfr^{W'}_{reg}/I)]}{\langle P_{W''}(T_{W''}) : \text{$W'$ is corank 1 in $W''$}\rangle}$$ from Theorem \ref{general-presentation-theorem}.  It is clear that the map of $\cplx$-vector spaces $$\cplx_{-\mu}[I^{comp}] \otimes_\cplx \frac{\cplx_{-\tilde{\mu}}[\pi_1(\hfr^{W'}_{L-reg}/I^{ref})]}{\langle P_{W''}(T_{W''}) : W'' \text{ is $L$-essential}\rangle} \rightarrow \frac{\cplx_{-\tilde{\mu}}[\pi_1(\hfr^{W'}_{reg}/I)]}{\langle P_{W''}(T_{W''}) : \text{$W'$ is corank 1 in $W''$}\rangle}$$ induced by multiplication is surjective.  It follows by Theorem \ref{general-presentation-theorem} that this map is in fact an isomorphism of $\cplx$-vector spaces.

For a simple reflection $s \in S_{W', L} \subset I^{ref}$, let $e_s$ denote the generator of monodromy $T_{W', L, \langle W', s \rangle}$.  For any $w \in I^{ref}$, let $e_w \in \pi_1(\hfr^{W'}_{L-reg}/I^{ref})$ denote the element obtained as the product $e_w := e_{s_1} \cdots e_{s_l}$ for any reduced expression $w = s_1 \cdots s_l$ of $w$ as a product of simple reflections in $S_{W', L}$, as in the last paragraph of the proof of Theorem \ref{general-presentation-theorem} (recall that the product here is taken in the fundamental group $\pi_1(\hfr^{W'}_{L-reg}/I^{ref})$, i.e. without regards to the cocycle $-\tilde{\mu}$).  For $w \in I^{ref}$, let $T_w$ denote $e_w$ viewed as an element of $\cplx_{-\tilde{\mu}}[\pi_1(\hfr^{W'}_{L-reg}/I^{ref})].$  By Theorem \ref{general-presentation-theorem} and the last paragraph in its proof, the set $\{T_w : w \in I^{ref}\}$ forms a basis for the quotient $$\frac{\cplx_{-\tilde{\mu}}[\pi_1(\hfr^{W'}_{L-reg}/I^{ref})]}{\langle P_{W''}(T_{W''}) : W'' \text{ is $L$-essential}\rangle}.$$  By rescaling the elements $T_w$ appropriately using twists by characters as $\tau_\chi$ was used in the proof of Theorem \ref{general-presentation-theorem}, we may assume that the quadratic relations the $T_s$ satisfy are of the form $$(T_s - 1)(T_s + q_s) = 0$$ for some uniquely determined $q_s \in \cplx^\times$.  The assignments $s \mapsto q_s$ determine an $I$-invariant function $q : S_{W', L} \rightarrow \cplx^\times$ with $q(s) = q_s$.  The computations at the end of the proof of Theorem \ref{general-presentation-theorem} show that for $s \in S_{W', L}$ and $w \in I^{ref}$ we have the multiplication law $$T_sT_w = \begin{cases} \mu(s, w)^{-1}T_{sw} & \text{ if $l(sw) > l(w)$} \\ q_s\mu(s, w)^{-1}T_{sw} + (q_s - 1)T_w & \text{ if $l(sw) < l(w)$}\end{cases}$$ where $l : I^{ref} \rightarrow \ints^{\geq 0}$ is the length function determined by the choice of simple reflections $S_{W', L}$ (note that $\mu(s, sw) = \mu(s, w)^{-1}$ because $s^2 = 1$).  An entirely analogous calculation shows that $$T_wT_s = \begin{cases} \mu(w, s)^{-1}T_{ws} & \text{ if $l(ws) > l(w)$} \\ q_s\mu(w, s)^{-1}T_{ws} + (q_s - 1)T_w & \text{ if $l(ws) < l(w).$}\end{cases}$$

For $d \in I^{comp}$, let $$e_d \in \frac{\pi_1(\hfr^{W'}_{reg}/I)}{\langle T_{W''} : \text{ $W''$ is $L$-trivial}\rangle}$$ be the image of $[\gamma_d]$ as constructed above.  We may also regard the $e_w$ for $w \in I^{ref}$ as elements of this quotient.  Any element $x \in I$ can be written uniquely as a product $x = dw$ for some $d \in I^{comp}$ and $w \in I^{ref}$, and we may therefore define $e_x := e_de_w$.  Note that $e_de_we_{d^{-1}} = e_{dwd^{-1}}$ so $e_{dwd^{-1}}e_d = e_{x}$, so writing $x = w'd'$ with $w' \in I^{ref}$ and $d' \in I^{comp}$ and taking the element $e_{w'}e_{d'}$ defines the same element $e_x$.  Note that $e_de_d' = e_{dd'}$ for any $d, d' \in I^{comp}$, and hence $e_de_x = e_{dx}$ and $e_xe_d = e_{xd}$ for any $d \in I^{comp}$ and $w \in I$.  For any $x \in I$, let $T_x$ denote the element $e_x$ viewed as an element of the quotient algebra $$\frac{\cplx_{-\tilde{\mu}}[\pi_1(\hfr^{W'}_{reg}/I)]}{\langle P_{W''}(T_{W''}) : \text{$W'$ is corank 1 in $W''$}\rangle}.$$  For $x \in I^{ref}$, the elements $T_x$ are the images of the elements $T_x$ considered above.  The multiplication rules $T_dT_x = \mu(d, x)^{-1}T_{dx}$ and $T_xT_d = \mu(x, d)^{-1}T_{xd}$ for $d \in I^{comp}$ and $x \in I$ follow immediately from the computations with $e_x$ and $e_d$ above and the definition of the twisted group algebra.  That $\{T_x : x \in I\}$ forms a basis for this quotient, and hence also for $\mathcal{H}(c, W', L, W'')$ by Theorem \ref{general-presentation-theorem} follows immediately from the considerations above, and the theorem follows. \end{proof}

\subsection{Computing Parameters Via KZ} \label{computing-params-section} To compute the quadratic relations that the generators $T_s \in \mathcal{H}(c, W', L, W)$ satisfy, we will reduce the problem to certain explicit computations in the Hecke algebra $\mathsf{H}_q(W)$ of the ambient group $W$ and its representations.  This is possible thanks to the following result of Shan:

\begin{lemma} \label{Shan-lemma} \cite[Lemma 2.4]{Shan} Let $KZ : \oscr_c(W, \hfr) \rightarrow \mathsf{H}_q(W)\mhyphen\text{mod}_{f.d.}$ be the $KZ$ functor, and let $K, L : \oscr_c(W, \hfr) \rightarrow \oscr_c(W, \hfr)$ be two right exact functors that map projective objects to projective objects.  Then the natural map of vector spaces $$\hom(K, L) \rightarrow \hom(KZ \circ K, KZ \circ L)$$ $$f \mapsto 1_{KZ}f$$ is an isomorphism.\end{lemma}

Assume that $(W', S') \subset (W, S)$ is a Coxeter subsystem of corank 1.  The following lemma describing the canonical complement $N_{W'}$ to $W'$ in $N_W(W')$ will be useful in what follows:

\begin{lemma} \label{coxeter-normalizer-lemma} Either $W'$ is self-normalizing in $W$ or has index 2 in its normalizer $N_W(W')$.  In the latter case, the longest elements $w_0$ and $w_0'$ of $W$ and $W'$, respectively, commute, and the canonical complement $N_{W'}$ to $W'$ in $N_W(W')$ is $$N_{W'} = \{1, w_0w_0'\}.$$\end{lemma}

\begin{proof} The first statement follows from Lemma \ref{corank-1-refs}.  Let $w \in N_{W'}$ be the nontrivial element.  Let $\Phi \subset \hfr_{\real}$ denote the root system of $W$, let $\Phi_{W'} \subset \Phi$ denote the root system of $W'$, let $\alpha_1, ..., \alpha_r$ be an ordering of the simple roots in $\Phi$ so that $\alpha_1, ..., \alpha_{r - 1}$ are the simple roots in $\Phi_{W'}$, and let $\Phi^+ \subset \Phi$ and $\Phi^+_{W'} \subset \Phi_{W'}$ denote the respective subsets of positive roots.  By definition of $N_{W'}$, $w(\Phi_{W'}^+) = \Phi_{W'}^+$.  From the proof of Lemma \ref{corank-1-refs}, we see that $w$ acts by $-1$ in $\hfr_\real^{W'}$.  As the remaining simple root $\alpha_r$ is the unique simple root with nonzero component in $\hfr_\real^{W'}$ with respect to the decomposition $\hfr_\real = \hfr_{\real, W'} \oplus \hfr_\real^{W'}$ and as every positive root $\alpha \in \Phi^+$ is uniquely of the form $\alpha = \sum_i n_i\alpha_i$ for some nonnegative integers $n_i \geq 0$, it follows that $w(\Phi^+\backslash\Phi^+_{W'}) \subset \Phi^-$.  It follows that the inversion set of $w$ is precisely $\Phi^+\backslash\Phi^+_{W'}$ and hence that $w = w_0w_0'$.  As $w, w_0$, and $w_0'$ are involutions, it follows that $w_0$ and $w_0'$ commute.\end{proof}

As in previous sections, let $c : S \rightarrow \cplx$ be a $W$-invariant function and let $L$ be an irreducible finite-dimensional representation of $H_c(W', \hfr_{W'})$.  Let $w_0$ and $w_0'$ denote the longest elements in $W$ and $W'$, respectively.  We will assume that $W'$ is not self-normalizing in $W$, and by Lemma \ref{coxeter-normalizer-lemma} it follows that $w_0$ and $w_0'$ commute, that $W'$ is index 2 in $N_{W'}(W)$, and that $N_{W'} = \{1, w_0w_0'\}.$  We will assume that the involution $w_0w_0'$ fixes the isomorphism class of $L$ so that $I = N_{W'}$ and the monodromy operator $T_{W'}$ satisfies a nontrivial quadratic relation as in Theorem \ref{general-presentation-theorem}.

The following observation in this setting will be central to our approach as it will allow for the explicit computation, via computations in the Hecke algebra $\mathsf{H}_q(W)$, of the eigenvalues of monodromy in the local systems arising from the functors $KZ_L$.  This lemma should be regarded as a generalization to Coxeter groups of arbitrary type of the calculation appearing in \cite[Lemma 4.14]{GoMa}.

\begin{lemma} \label{monodromy-generator-lemma} Let $\mathcal{C}_{W'} \subset \hfr_{W', \real}$ be the open fundamental Weyl chamber associated to $(W', S')$, let $\mathcal{C}_W \subset \hfr_{\real}$ be the open fundamental Weyl chamber associated to $(W, S)$, and choose a basepoint $b = (b', b'') \in \mathcal{C}_{W'} \times \hfr^{W'}_{reg}$ lying in $\mathcal{C}_W$.  As $\mathcal{C}_{W'}$ is contractible and stable under $w_0w_0'$, the pair $(\gamma, T_{W'})$, where $\gamma$ is any path in $\mathcal{C}_{W'}$ from $b'$ to $w_0w_0'b'$ and $T_{W'}$ is the half-loop in $\hfr^{W'}_{reg}$ lifting the canonical generator of monodromy in $\pi_1(\hfr^{W'}_{reg}/N_{W'})$, determines an element in $\pi_1(((\hfr_{W'})_{reg} \times \hfr^{W'}_{reg})/N_{W'})$ that does not depend on the choice of $\gamma$.  The image of this element under the natural map $$\iota_{W', 1} : \pi_1(((\hfr_{W'})_{reg} \times \hfr^{W'}_{reg})/N_{W'}) \rightarrow \pi_1(\hfr_{reg}/W) = B_W$$ is $T_{w_0w_0'}$. \end{lemma}

\begin{proof} That the element of $\pi_1(((\hfr_{W'})_{reg} \times \hfr^{W'}_{reg})/N_{W'})$ determined by the pair $(\gamma, T_{W'})$ does not depend on the choice of $\gamma$ follows immediately from the contractibility of the fundamental Weyl chamber $\mathcal{C}_{W'}$.  By definition of $\iota_{W', 1}$, for any sufficiently large real number $R > 0$ the image $\iota_{W', 1}(\gamma, T_{W'})$ in $\pi_1(\hfr_{reg}/W)$ is represented by the image of the path $$p : [0, 1] \rightarrow \hfr_{reg}$$ $$p(t) = (\gamma(t), Re^{\pi it}b'')$$ under the natural projection $\hfr_{reg} \rightarrow \hfr_{reg}/W$.  As $w_0w_0'b'' = -b''$, it follows that $p$ is a path from the point $p(0) = (b', Rb'') \in \mathcal{C}_W$ to the point $p(1) = w_0w_0'p(0) \in w_0w_0'\mathcal{C}_W.$  For each positive root $\alpha \in \Phi^+$, let $H_\alpha := \ker(\alpha) \subset \hfr$ be the associated reflection hyperplane.  That $p$ represents the element $T_{w_0w_0'}$ follows from the observation that $p$ traverses a positively-oriented (with respect to the complex structure) half-loop about each hyperplane $H_\alpha$ for roots $\alpha \in \Phi^+\backslash\Phi^+_{W'}$ while $p$ does not encircle any of the remaining hyperplanes $H_\alpha$ for roots $\alpha \in \Phi^+_{W'}$.  More precisely, for roots $\alpha \in \Phi^+\backslash\Phi^+_{W'}$ the composition $\alpha \circ p : [0, 1] \rightarrow \cplx^\times$ determines a path from the positive real axis $\real^+$ to the negative real axis $\real^-$ lying entirely in the upper half-space $\{z \in \cplx^\times : \text{Re}(z) \geq 0\}$.  For roots $\alpha \in \Phi^+_{W'}$, the composition $\alpha \circ p : [0, 1] \rightarrow \cplx^\times$ determines a path lying entirely on $\real^+$, as $\gamma(t) \in \mathcal{C}_{W'}$ and $\alpha(b'') = 0$.  The equality $\iota_{W', 1}(\gamma, T_{W'}) = T_{w_0w_0'}$ follows.\end{proof}

Let $\mon(T_{W})$ denote the isomorphism of functors $$\res_{W'}^W \rightarrow \tw_{w_0w_0'} \circ \res_{W'}^W$$ arising from monodromy along the generator of monodromy $T_{W}$ in the local system $\underline{\res}_{W'}^W$, where $\underline{\res}_{W'}^W$ is the partial $KZ$ functor recalled in Section \ref{partial-KZ}.  The following is an immediate corollary of Lemma \ref{monodromy-generator-lemma} and the transitivity result of Gordon-Martino recalled in Theorem \ref{trans-theorem}:

\begin{lemma} \label{monodromy-lift-lemma} Multiplication by $T_{w_0w_0'}$ defines an isomorphism $$T_{w_0w_0'} : \hres_{W'}^W \rightarrow \htw_{w_0w_0'} \circ \hres_{W'}^W$$ of functors $\mathsf{H}_q(W')\mhyphen\text{mod}_{f.d.} \rightarrow \mathsf{H}_q(W')\mhyphen\text{mod}_{f.d.}.$  Furthermore, $\mon(T_{W})$ is the lift to $\oscr_c(W, \hfr)$ of $T_{w_0w_0'}$ in the sense that, with respect to the identifications $KZ \circ \res_{W'}^W = \hres_{W'}^W \circ KZ$ and $KZ \circ \tw_{w_0w_0'} = \htw_{w_0w_0'} \circ KZ$, we have an equality $$1_{KZ}\mon(T_{W}) = T_{w_0w_0'}1_{KZ}$$ of isomorphisms of functors $$KZ \circ \res_{W'}^W \rightarrow KZ \circ \tw_{w_0w_0'} \circ \res_{W'}^W.$$\end{lemma}

\begin{remark} When the meaning is clear, we denote the $KZ$ functors defined for $\oscr_c(W, \hfr)$ and $\oscr_c(W', \hfr_{W'})$ each by $KZ$.\end{remark}

\begin{proof}[Proof of Lemma \ref{monodromy-lift-lemma}] The first statement follows from the observation that the functor $\hres_{W'}^W$ is represented by the $\mathsf{H}_q(W')\mhyphen \mathsf{H}_q(W)$-bimodule $\mathsf{H}_q(W)$ and the observation that $T_{w_0w_0'}T_s = T_{\sigma_{w_0w_0'}(s)}T_{w_0w_0'}$ for all $s \in S'$, where $\sigma_{w_0w_0'}$ is the diagram automorphism of $(W', S')$ induced by conjugation by $w_0w_0'$ in $W$.  The second statement is an immediate consequence of Lemma \ref{monodromy-generator-lemma} and Theorem \ref{trans-theorem}.\end{proof}

\begin{definition} \label{z-defs} Let $X_{W'} \subset W$ be the set of minimal length left $W'$-coset representatives, and let $\{z_d\}_{d \in X_{W'}}$ be the unique set of elements $z_d \in \mathsf{H}_q(W')$ such that $$T_{w_0w_0'}^2 = \sum_{d \in X_{W'}} z_dT_d.$$\end{definition}

We are interested in the elements $z_1, z_{w_0w_0'} \in \mathsf{H}_q(W')$ arising from the left $W'$-cosets that are also right $W'$-cosets.  We now establish important properties that these elements satisfy:

\begin{lemma} \label{z-props-lemma} The element $z_{w_0w_0'}$ commutes with $T_{w_0w_0'}$ and satisfies the relation $$xz_{w_0w_0'} = z_{w_0w_0'}\sigma(x)$$ for all $x \in \mathsf{H}_q(W')$, where $\sigma : \mathsf{H}_q(W') \rightarrow \mathsf{H}_q(W')$ is the automorphism of $\mathsf{H}_q(W')$ arising from the diagram automorphism of $(W', S')$ obtained by conjugation by $w_0w_0'$ in $W$.  In particular, multiplication by $z_{w_0w_0'}$ defines a morphism $$z_{w_0w_0'} : \id \rightarrow \htw_{w_0w_0'}$$ of functors $\mathsf{H}_q(W')\mhyphen\text{mod}_{f.d.} \rightarrow \mathsf{H}_q(W')\mhyphen\text{mod}_{f.d.}.$  Furthermore, $$z_1 = q^{l(w_0w_0')},$$ where, in the unequal parameter case, $q^{l(w_0w_0')}$ denotes the product $q^{l(w_0w_0')} := \prod_{s \in S/W} q_s^{l_s(w_0w_0')}$ over the conjugacy classes of reflections in $W$, where $l_s : W \rightarrow \ints^{\geq 0}$ is the length function of $W$ attached to the conjugacy class of $s$ and $q_s \in \cplx$ is the parameter associated to the conjugacy class of $s$.\end{lemma}

\begin{proof} As $w_0$ and $w_0'$ commute, we have $T_{w_0}T_{w_0'}^{-1} = T_{w_0w_0'} = T_{w_0'w_0} = T_{w_0'}^{-1}T_{w_0}.$  Every simple reflection in $S$ lies in both the right and left descent sets of $w_0$, and similarly every simple reflection in $S'$ lies in both the right and left descent sets of $w_0'$.  It follows that $T_{w_0}T_s = T_{\sigma_{w_0}(s)}T_{w_0}$ for every $s \in S$, where $\sigma_{w_0}$ is the diagram automorphism of $(W, S)$ obtained by conjugation by $w_0$, and similarly $T_{w_0'}T_s = T_{\sigma_{w_0'}(s)}T_{w_0'}$ for all $s \in S'$, where $\sigma_{w_0'}$ is the diagram automorphism of $(W', S')$ obtained by conjugation by $w_0'$.  It follows that $T_{w_0}^2$ is central in $\mathsf{H}_q(W)$ and $T_{w_0'}^2$ is central in $\mathsf{H}_q(W')$ (this is well known \cite{BrSa, De-center}) and that $T_{w_0w_0'}x = \sigma(x)T_{w_0w_0'}$ for all $x \in \mathsf{H}_q(W')$.    In particular, the element $T_{w_0w_0'}^2 = T_{w_0}^2T_{w_0'}^{-2} \in \mathsf{H}_q(W)$ considered above centralizes $\mathsf{H}_q(W')$.  As multiplication by elements of $\mathsf{H}_q(W')$ on the right or left respects the decomposition of $T_{w_0w_0'}^2$ by $W'$-double cosets and as $T_{w_0w_0'}$ normalizes $\mathsf{H}_q(W')$, it follows that $xz_{w_0w_0'}T_{w_0w_0'} = z_{w_0w_0'}T_{w_0w_0'}x$ for all $x \in \mathsf{H}_q(W')$.  But as $T_{w_0w_0'}x = \sigma(x)T_{w_0w_0'}$ and $T_{w_0w_0'}$ is invertible, it follows that $xz_{w_0w_0'} = z_{w_0w_0'}\sigma(x)$.  That multiplication by $z_{w_0w_0'}$ defines a morphism of functors $\id \rightarrow \htw_{w_0w_0'}$ follows immediately.

Similarly, note that conjugation by $T_{w_0'}$ clearly respects decomposition of elements by $W'$-double cosets and that conjugation by $T_{w_0}$ stabilizes $\mathsf{H}_q(W')$ and sends minimal length $W'$-double-coset representatives of to minimal length $W'$-double-coset representatives of the same length.  In particular, conjugation by $T_{w_0w_0'}$ fixes the top degree term $z_{w_0w_0'}T_{w_0w_0'}$ in the decomposition of $T_{w_0w_0'}^2$ by left $W'$-cosets.  As $T_{w_0w_0'}$ commutes with itself and $T_{w_0w_0'}z_{w_0w_0'} = \sigma(z_{w_0w_0'})T_{w_0w_0'}$ this implies that $z_{w_0w_0'}T_{w_0w_0'} = \sigma(z_{w_0w_0'})T_{w_0w_0'}$ and hence that $\sigma(z_{w_0w_0'}) = z_{w_0w_0'}$.  In particular, $z_{w_0w_0'}$ commutes with $T_{w_0w_0'}$.

Finally, to show that $z_1 = q^{l(w_0w_0')}$, by multiplying on the left by $T_{w_0'}$ it suffices to show that the component of $T_{w_0}T_{w_0w_0'}$ lying in $\mathsf{H}_q(W')$ according to the decomposition of $\mathsf{H}_q(W)$ by left $W'$-cosets is $q^{l(w_0w_0')}T_{w_0'}$.  This is clear from the interaction of the multiplication laws defining $\mathsf{H}_q(W)$ and the length function.\end{proof}

\begin{definition} \label{cz} Let $\cz_{w_0w_0'}$ be the unique morphism $$\cz_{w_0w_0'} : \id \rightarrow \tw_{w_0w_0'}$$ of functors $\oscr_c(W', \hfr_{W'}) \rightarrow \oscr_c(W', \hfr_{W'})$ lifting $z_{w_0w_0'}$ in the sense of Lemma \ref{Shan-lemma}, i.e. so that $$1_{KZ}\cz_{w_0w_0'} = z_{w_0w_0'}1_{KZ}$$ with respect to the identification $KZ \circ tw_{w_0w_0'} = \htw_{w_0w_0'} \circ KZ.$\end{definition}

\begin{definition} Let $\mu_{w_0w_0'}$ be the isomorphism of functors $$\mu_{w_0w_0'} : \id \oplus \tw_{w_0w_0'} \rightarrow \tw_{w_0w_0'} \circ (\id \oplus \tw_{w_0w_0'}) = \tw_{w_0w_0'} \oplus \id$$ defined by the matrix $$\mu_{w_0w_0'} := \begin{pmatrix} 0 & q^{l(w_0w_0')} \\ 1 & \cz_{w_0w_0'}\end{pmatrix}.$$\end{definition}

Recall that by the Mackey formula for rational Cherednik algebras attached to Coxeter groups (see Section \ref{Mackey-Coxeter}), and the fact that $L$ is finite-dimensional and hence annihilated by all restriction functors $\res_{W''}^{W'}$ for proper parabolic subgroups $W'' \subsetneq W'$, that we have $$(\res_{W'}^W \circ \ind_{W'}^W)L \cong (\id \oplus \tw_{w_0w_0'})L.$$

\begin{lemma} \label{matrix-lemma} With respect to the isomorphism $$(\res_{W'}^W \circ \ind_{W'}^W)L \cong (\id \oplus \tw_{w_0w_0'})L$$ arising from the Mackey formula, the isomorphisms $$(\mon(T_{W})1_{\ind_{W'}^W})(L), \mu(L) : (\id \oplus \tw_{w_0w_0'})L \rightarrow (\tw_{w_0w_0'} \oplus \id)L$$ are equal.\end{lemma}

\begin{proof} This is an immediate consequence of the compatibility of the Mackey decomposition for rational Cherednik algebras with the $KZ$ functor, Definition \ref{z-defs}, and Lemmas \ref{monodromy-lift-lemma} and \ref{z-props-lemma}.\end{proof}

\noindent $\mathbf{Notation}$: For a central element $z \in \mathsf{H}_q(W')$, let $z|_L \in \cplx$ denote the scalar by which $z$ acts on irreducible representations of $\mathsf{H}_q(W')$ lying in the block of $\mathsf{H}_q(W')\mhyphen\text{mod}_{f.d.}$ corresponding to the block of $\oscr_c(W', \hfr_{W'})$ containing $L$ under the $KZ$ functor.

We can now give a formula for the quadratic relations satisfied by the elements $T_s$ appearing in Theorem \ref{coxeter-presentation}:

\begin{theorem} \label{quadratic-relations-theorem} Let $T$ denote the element $T_{W} \in \mathcal{H}(c, W', L, W),$ and let $n \in \aut_\cplx(L)$ denote the involution of $L$ by which $w_0w_0' \in N_{W'}$ acts making $L$ $N_{W'}$-equivariant as a $H_c(W', \hfr_{W'})$-module.  Then $T$ satisfies the quadratic relation $$T^2 = (\cz_{w_0w_0'}(L) \circ n)|_LT + q^{l(w_0w_0')}$$ where $(\cz_{w_0w_0'}(L) \circ n)|_L \in \cplx$ denotes the scalar by which the $H_c(W', \hfr_{W'})$-module homomorphism $n \circ \cz_{w_0w_0'}(L)$ acts on the irreducible representation $L$.

Furthremore, if the diagram automorphism $\sigma = \sigma_{w_0w_0'}$ of $(W', S')$ arising from conjugation by $w_0w_0'$ is trivial, $T$ satisfies the quadratic relation $$T^2 = (z_{w_0w_0'})|_LT + q^{l(w_0w_0')}.$$  If the diagram automorphism $\sigma_{w_0}$ is trivial but the diagram automorphism $\sigma_{w_0'}$ is nontrivial, $T$ satisfies the quadratic relation $$T^2 = (z_{w_0w_0'}T_{w_0'})|_L(T_{w_0'}^2)|_L^{-1/2}T + q^{l(w_0w_0')}.$$\end{theorem}

\begin{remark} The projective representation of $N_{W'} = I \cong \ints/2\ints$ on $L$ lifts to an linear representation of $N_{W'}$ in two ways, differing by a tensor product with the nontrivial character of $N_{W'}$, so there is a choice of sign for the action of the nontrivial element $w_0w_0' \in N_{W'}$ on $L$.  The quadratic relations appearing in Theorem \ref{quadratic-relations-theorem} under assumptions on the diagram automorphisms hold for an appropriate choice of sign for the operator $n$ - choosing the other sign simply negates the linear term in the quadratic relation.  The quadratic relations of the form $(T_s - 1)(T_s + q_s) = 0$ appearing in Theorem \ref{coxeter-presentation} are obtained by rescaling the generators $T_{W}$ by twisting by characters of $\pi_1(\hfr^{W'}_{reg}/I)$ as in Lemma \ref{twisting-character}, and the relations in this normalized form are only determined up to inverting $q_s$ but do not depend on the choice of sign for the action of $w_0w_0'$ on $L$.\end{remark}

\begin{remark} The case in which the diagram automorphism $\sigma_{w_0}$ is nontrivial but the diagram automorphism $\sigma_{w_0'}$ is trivial only appears for groups of type $D$.  We will show later in Section \ref{type-D} how to reduce the problem of computing the quadratic relations in type $D$ to the type $B$ case in which this complication does not arise.\end{remark}

\begin{proof} By Proposition \ref{proj-gen}, $\ind_{W'}^W L$ is a projective generator of $\bar{\oscr}_{c, W', L}$, and hence it follows from Theorem \ref{KZL} that the action of $\mathcal{H}(c, W', L, W)$ on $\hom_{\oscr_c(W', \hfr_{W'})}(L, \res_{W'}^W \ind_{W'}^W L)$ is faithful.  It therefore suffices to check that the quadratic relation holds in this representation.

It follows from Lemma \ref{matrix-lemma} that the action of $T$ in the representation $$\hom_{\oscr_c(W', \hfr_{W'})}(L, \res_{W'}^W \ind_{W'}^W L) \subset L^* \otimes_\cplx \res_{W'}^W \ind_{W'}^W L$$ is by the operator $$n^* \otimes \mu_{w_0w_0'}(L) = n^* \otimes \begin{pmatrix} 0 & q^{l(w_0w_0')} \\ 1 & \cz_{w_0w_0'}(L)\end{pmatrix}.$$  As $n^2 = 1$, a simple calculation yields $$T^2 = \left(n^* \otimes \begin{pmatrix} \cz_{w_0w_0'}(L) & 0 \\ 0 & \cz_{w_0w_0'}(L)\end{pmatrix}\right) \circ T + \begin{pmatrix} q^{l(w_0w_0')} & 0 \\ 0 & q^{l(w_0w_0')}\end{pmatrix}.$$  As the operator appearing in front of $T$ on the righthand side acts on $\hom_{\oscr_c(W', \hfr_{W'})}(L, L \oplus \tw_{w_0w_0'}L)$ by the scalar $(\cz_{w_0w_0'}(L) \circ n)|_L$, the first claim follows.

Now, suppose the diagram automorphism $\sigma_{w_0w_0'}$ of $(W', S')$ is trivial, so that $w_0w_0'$ centralizes $W'$ and acts on $\hfr_{W'}$ trivially.  In particular, $w_0w_0'$ acts trivially on $H_c(W', \hfr_{W'})$, and hence the trivial action of $N_{W'}$ on $L$ makes $L$ $N_{W'}$-equivariant, so we may take $n = \id_L$.  The quadratic relation for $T$ in this case follows immediately.

Finally, suppose the diagram automorphism $\sigma_{w_0}$ is trivial but the diagram automorphism $\sigma_{w_0'}$ is not.  It follows that the diagram automorphisms $\sigma_{w_0'}$ and $\sigma = \sigma_{w_0w_0'}$ are equal and that $T_{w_0'}x = \sigma(x)T_{w_0'}$ for all $x \in \mathsf{H}_q(W')$.  In particular, multiplication by $T_{w_0'}$ defines a morphism $$T_{w_0'} : \id \rightarrow \htw_{w_0w_0'}$$ of functors $\mathsf{H}_q(W')\mhyphen\text{mod}_{f.d.} \rightarrow \mathsf{H}_q(W')\mhyphen\text{mod}_{f.d.}.$  Let $\cT_{w_0'}$ be the morphism $$\cT_{w_0'} : \id \rightarrow \tw_{w_0w_0'}$$ of functors $\oscr_c(W', \hfr_{W'}) \rightarrow \oscr_c(W', \hfr_{W'})$ obtained by lifting $T_{w_0'}$ by Lemma \ref{Shan-lemma}, similarly to the definition of $\cz_{w_0w_0'}$ (Definition \ref{cz}).  We may then take the operator $n \in \aut_\cplx(L)$ by which $w_0w_0'$ acts to be the involutive operator $n = (T_{w_0'}^2)|_L^{-1/2}\cT_{w_0'}(L)$.  We then have $(\cz_{w_0w_0'}(L) \circ n)|_L = (\cz_{w_0w_0'}(L) \circ (T_{w_0'}^2)|_L^{-1/2}(\cT_{w_0'}(L)))|_L = (z_{w_0w_0'}T_{w_0'})|_L(T_{w_0'}^2)|_L^{-1/2}$, and the final claim follows.\end{proof}

We will see that, in the setting of Coxeter groups, the projective representation of $I$ on $L$ always lifts to a linear representation, and in particular the cocycle $\mu \in Z^2(I, \cplx^\times)$ is always trivial.  Furthermore, the inertia group $I$ is always as large as possible, i.e. it equals $N_{W'}$.  These facts make the presentations of the algebras $\mathcal{H}(c, W', L, W)$ particularly simple:

\begin{theorem} \label{simplified-coxeter-theorem} Let $W$ be a finite Coxeter group with simple reflections $S$, let $c : S \rightarrow \cplx$ be a class function, let $W'$ be a standard parabolic subgroup generated by the simple reflections $S'$, and let $L$ be an irreducible finite-dimensional representation of the rational Cherednik algebra $H_c(W', \hfr_{W'})$.  Let $N_{W'}$ denote the canonical complement to $W'$ in its normalizer $N_W(W')$, let $S_{W'} \subset N_{W'}$ denote the set of reflections in $N_{W'}$ with respect to its representation in the fixed space $\hfr^{W'}$, and let $N_{W'}^{ref}$ denote the reflection subgroup of $N_{W'}$ generated by $S_{W'}$.  Let $N_{W'}^{comp}$ be a complement for $N_{W'}^{ref}$ in $N_{W'}$, given as the stabilizer of a choice of fundamental Weyl chamber for the action of $N_{W'}^{ref}$ on $\hfr^{W'}$.  Then there is a class function $q_{W', L} : S_{W'} \rightarrow \cplx^\times$ and an isomorphism of algebras $$\mathcal{H}(c, W', L, W) \cong N_{W'}^{comp} \ltimes \mathsf{H}_{q_{W', L}}(N_{W'}^{ref})$$ where the semidirect product is defined by the action of $N_{W'}^{comp}$ on $\mathsf{H}_{q_{W', L}}(N_{W'}^{ref})$ by diagram automorphisms arising from the conjugation action of $N_{W'}^{comp}$ on $N_{W'}^{ref}$.\end{theorem}

Theorem \ref{simplified-coxeter-theorem} is proved by case-by-case analysis of the inertia subgroups $I$ and their 2-cocycles $\mu$, to which the rest of this paper is dedicated.  The class function $q_{W', L}$ can be explicitly computed using the corank-1 methods developed in Section \ref{computing-params-section}.  We compute this class function in many cases, leading to complete lists of the irreducible finite-dimensional representations of the algebras $H_c(W, \hfr)$ in many new cases in exceptional types.

\begin{remark} In all cases that we have computed explicitly, the class function $q_{W', L}$ depends only on the parabolic subgroup $W'$, and not on the finite-dimensional irreducible representation $L$.  It would be interesting to have a conceptual explanation for this fact.\end{remark}

\noindent $\mathbf{Notation}$  For a Coxeter group $W$, corank 1 parabolic subgroup $W' \subset W$, and finite-dimensional irreducible representation $L$ of the rational Cherednik algebra $H_c(W', \hfr_{W'})$, let $q(c, W', L, W) \in \cplx^\times$ denote a scalar such that the element $T_{W', L, W} \in \mathcal{H}(c, W', L, W)$, after an appropriate rescaling, satisfies the quadratic relation $$(T - 1)(T + q(c, W', L, W)) = 0.$$  Note that $q(c, W', L, W)$ is determined only up to taking an inverse.  The calculations in the proof of Theorem \ref{quadratic-relations-theorem} show that the constant term of the monic quadratic relation satisfied by the canonical (up to sign) element $T_{W', L, W}$ is $q^{l(w_0w_0')}$, and the linear term in the monic quadratic relation satisfied by $T_{W', L, W}$ can therefore be recovered, again up to sign, from $q^{l(w_0w_0')}$ and $q(c, W', L, W)$.  Note also that the ambiguity of $q(c, W', L, W)$ up to inverse has no impact on the isomorphism class of any Iwahori-Hecke algebra for which $q(c, W', L, W)$ is a parameter, and an explicit isomorphism can be obtained by scaling the corresponding generators by $-q(c, W', L, W)^{-1}$.

\subsection{Type A} \label{type-A} Let $n \geq 1$ be a positive integer, $S_n$ be the symmetric group on $n$ letters with irreducible reflection representation $\hfr = \{(z_1, ..., z_n) \in \cplx^n : \sum_i z_i = 0\}$, $c \in \cplx$ be a complex number, and let $H_c(S_n, \hfr)$ be the associated rational Cherednik algebra.  As a first illustration of the results obtained in the previous sections, let us now recover the following result of Wilcox describing the subquotients of the filtration of category $\oscr_c(S_n, \hfr)$ by the dimension of supports:

\begin{theorem} \label{Wilcox} \emph{(Wilcox, \cite[Theorem 1.8]{Wilcox})} Suppose $c = r/e > 0$ is a positive rational number with $r$ and $e$ relatively prime positive integers.  The subquotient category of $\oscr_c(S_n, \hfr)$ obtained as the quotient of the full subcategory of modules in $\oscr_c(S_n, \hfr)$ supported on the subvariety $S_n\hfr^{S_e^k}$, where $k \geq 0$ is a nonnegative integer, modulo the Serre subcategory of modules with strictly smaller support, is equivalent to the category of finite-dimensional modules over the algebra $\cplx[S_k] \otimes \mathsf{H}_q(S_{n - ke})$, where $\mathsf{H}_q(S_{n - ke})$ is the Hecke algebra of $S_{n - ke}$ with parameter $q = e^{-2 \pi i c}.$\end{theorem}

Let $c = r/e$ be as in Theorem \ref{Wilcox}.  By \cite[Theorem 1.2]{BEG}, the only parabolic subgroups $W'$ of $S_n$ such that $H_c(W', \hfr_{W'})$ has nonzero finite-dimensional representations are the conjugates of parabolic subgroups of the form $S_e^k$ for some nonnegative integer $k \leq n/e$, and the unique irreducible finite-dimensional representation of $H_c(S_e^k, \hfr_{S_e^k})$, up to isomorphism, is $L := L(\triv)$, where $\triv$ denotes the trivial representation of $S_e^k$.  It follows that the subquotient category appearing in Theorem \ref{Wilcox} is the subquotient $\bar{\oscr}_{c, S_e^k, L}$.  By Theorem \ref{KZL}, to prove Theorem \ref{Wilcox} it suffices to give an isomorphism of algebras $$\mathcal{H}(c, S_e^k, L(\triv), S_n) \cong \cplx[S_k] \otimes \mathsf{H}_q(S_{n - ek}).$$  This follows from Theorems \ref{coxeter-presentation} and \ref{quadratic-relations-theorem}, as follows.

The fixed space $\hfr^{S_e^k}$ is $$\hfr^{S_e^k} = \{(z_1, ..., z_n) \in \hfr : \text{ for } 0 \leq l < k,\  z_{le + i} = z_{le + j} \text{ for } 1 \leq i < i \leq e\}.$$  Take coordinates $x_1, ..., x_k, y_1, ..., y_{n - ek}$ for $\hfr^{S_e^k}$, where $x_l = z_{(l - 1)e + i}$ for $1 \leq l \leq k$ and $1 \leq i \leq e$ and $y_j = z_{ek + j}$ for $1 \leq j \leq n - ek$.  These coordinates satisfy the relation $\sum_i x_i + e\sum_jy_j = 0$.  The complement $N_{S_e^k}$ to $S_e^k$ in its normalizer is isomorphic to $S_k \times S_{n - ek}$, with the action of $S_k$ on $\hfr^{S_e^k}$ given by permuting the $x_i$ coordinates and the action of $S_{n - ek}$ given by permuting the $y_j$ coordinates.  The action on $H_c(S_e^k, \hfr_{S_e^k}) \cong H_c(S_e, \hfr_{S_e})^{\otimes k}$ is by permuting the tensor factors, and in particular the inertia group $I_{S_e^k, L}$ is maximal, i.e. equals $N_{S_e^k}$.  Clearly the action of $N_{S_e^k}$ on $\hfr^{S_e^k}$ is generated by reflections, so we have $I_{S_e^k, L}^{ref} = N_{S_e^k}$ and $I_{S_e^k, L}^{comp} = 1$.  The trivial action of $N_{S_e^k}$ on the trivial representation $\triv$ of $S_e^k$ makes $\triv$ equivariant.  In particular, by Remark \ref{lowest-weight-remark}, the 2-cocycle $\mu \in Z^2(I_{S_e^k}, \cplx^\times)$ is trivial.

There are three distinct $N_{S_e^k}$-orbits of hyperplanes defining $\hfr^{S_e^k}_{reg} \subset \hfr^{S_e^k}$, given by $(1) \ x_i = x_j$ for $1 \leq i < j \leq k$, (2) $x_i = y_j$ for $1 \leq i \leq k$ and $1 \leq j \leq n - ek$, and (3) $y_i = y_j$ for $1 \leq i < j \leq n - ek$.  The stabilizers in $S_n$ of the $x_i = y_j$ hyperplanes are those parabolic subgroups containing $S_e^k$ and conjugate to $S_e^{k - 1} \times S_{e + 1}$, and $S_e^k$ is self-normalizing in these groups.  The stabilizers of the hyperplanes $x_i = x_j$ are those parabolic subgroups of $S_n$ containing $S_e^k$ and conjugate to $S_{2e} \times S_e^{k - 2}$, and the stabilizers of the hyperplanes $y_i = y_j$ are of those parabolic subgroups of $S_n$ containing $S_e^k$ and conjugate to $S_e^k \times S_2$.  Note that $S_e^k$ is not self-normalizing in either of these types of parabolic subgroups, and in particular the space $\hfr^{S_e^k}_{L-reg}$ is the complement of the hyperplanes $x_i = x_j$ and $y_i = y_j$.  It follows already from Theorem \ref{coxeter-presentation} that there is an isomorphism of algebras $\mathcal{H}(c, S_e^k, L(\triv), S_n) \cong \mathsf{H}_{q_1}(S_k) \otimes \mathsf{H}_{q_2}(S_{n - ek}),$ for some complex parameters $q_1, q_2 \in \cplx^\times$.  To show Theorem \ref{Wilcox}, it therefore suffices to show that $q_1 = 1$ and $q_2 = e^{-2 \pi i c}$.

By Remark \ref{reducible-reduction-prop}, the parameter $q_1$ can be computed by studying the inclusion $S_{e}^2 \subset S_{2e}$ and the parameter $q_2$ can be computed by studying the inclusion $1 \subset S_2$.  In the latter case, the associated central element $z_{T_1} \in \mathsf{H}_q(1) = \cplx$ is $1 - q = 1 - e^{-2 \pi i c}$, and therefore by Theorem \ref{quadratic-relations-theorem} the associated quadratic relation is $T^2 = (1 - q)T + q$, so $q_2 = q = e^{-2 \pi i c}$.

To obtain the parameter $q_1$, we need to analyze the inclusion $S_e^2 \subset S_{2e}$ and the associated element $z_{w_0w_0'} \in \mathsf{H}_q(S_e^2)$, where $w_0$ is the longest element of $S_{2e}$ and $w_0'$ is the longest element of $S_e^2$.

\begin{proposition} The decomposition of the element $T_{w_0w_0'}^2$ in the $T_w$-basis of $\mathsf{H}_q(S_{2e})$ is given by $$T_{w_0w_0'}^2 = \sum_{w \in X_e} (1 - q)^{a(w)}q^{b(w)}T_w$$ where $X_e \subset S_{2e}$ is the subset of elements $w \in S_{2e}$ such that the three conditions

(1) $w^2 = 1$

(2) $w(i) = i$ or $w(i) > e$ for $1 \leq i \leq e$

(3) $w(i) = i$ or $w(i) \leq e$ for $m < i \leq 2e$,\\
hold and where the functions $a, b : X_e \rightarrow \ints^{\geq 0}$ are defined by $$a(w) = \#\{i \in [1, e] : w(i) > e\}$$ and $$b(w) = -\#\{(i, j) : 1 \leq i < j \leq e, w(i) > w(j)\} + \sum_{i = 1}^e \begin{cases} e & w(i) = i\\ 2e - w(i) & w(i) > e.\end{cases}$$  In particular, the element $z_{w_0w_0'} \in \mathsf{H}_q(S_e^2) = \mathsf{H}_q(S_e)^{\otimes 2}$ is given by $$z_{w_0w_0'} = (1 - q)^eq^{ e \choose 2} \sum_{w \in S_e} q^{-l(w)}T_w \otimes T_{w^{-1}}.$$\end{proposition}

\begin{proof} The expression for $T_{w_0w_0'}^2$ can be obtained by a simple inductive argument using the reduced expression $$w_0w_0' = (s_e \cdots s_{2e - 1})(s_{e - 1} \cdots s_{2e - 2}) \cdots (s_1 \cdots s_e)$$ and the relations defining the Hecke algebra $\mathsf{H}_q(S_{2e})$, from which the expression for $z_{w_0w_0'}$ follows immediately.\end{proof}

The image of the Verma module $\Delta_c(\triv)$ in $\oscr_c(S_e^2, \hfr_{S_e^2})$ under the $KZ$ functor is the 1-dimensional representation $KZ(\Delta_c(\triv))$ on which all of the generators $T_i \in \mathsf{H}_q(S_e^2)$ act by the identity.  In particular, the element $z_{w_0w_0'}$ acts by the scalar $(1 - q)q^{e \choose 2}\sum_{w \in S_e} q^{-l(w)} = (1 - q)q^{e \choose 2}P_{S_e}(q^{-1})$, where $P_{S_e}$ is the Poincar\'{e} polynomial of $S_e$.  By the well known identity $$P_{S_e} = \prod_{i = 1}^e \frac{1 - q^i}{1 - q}$$ and the fact that $q = e^{-2 \pi i c}$ is a primitive $e^{th}$ root of unity, it follows that $z_{w_0w_0'}$ acts by $0$ on $KZ(\Delta_c(\triv))$, and hence it follows that $z_{w_0w_0'}$ acts by $0$ on all simple objects in the block of $\mathsf{H}_q(S_e^2)$ containing $KZ(\Delta_c(\triv))$.  In particular, $\cz_{w_0w_0'}(L(\triv)) = 0$.  Note also that $l(w_0w_0') = e^2$, so $q^{l(w_0w_0')} = 1$.  By Theorem \ref{quadratic-relations-theorem}, it follows that $T_{S_e^2, L, S_{2e}}^2 = 1$.  Therefore, the parameter $q_1$ is 1, and the isomorphism $$\mathcal{H}(c, S_e^k, L, S_n) \cong \cplx[S_k] \otimes \mathsf{H}_q(S_{n - ek})$$ follows, as needed.

\subsection{Type B} \label{type-B} In this section we will illustrate our results and check Theorem \ref{simplified-coxeter-theorem} in the setting of the type $B$ Coxeter groups.  The results we obtain in type $B$ follow from the work of Shan-Vasserot \cite{SV}.

Recall that the Coxeter group $B_n$ is the semidirect product $S_n \ltimes (\ints/2\ints)^n$, where $S_n$ acts on $(\ints/2\ints)^n$ by permutation, and that it acts in its reflection representation $\hfr = \cplx^n$ by permutations and sign changes of the coordinates.  There are two conjugacy classes of reflections, the first associated to the coordinate hyperplanes $z_i = 0$ and the second associated to the hyperplanes $z_i = \pm z_j$.  A class function $c$ on the set of reflections therefore amounts to a choice of parameter $c_1 \in \cplx$ for the $z_i = 0$ hyperplanes and a choice of parameter $c_2 \in \cplx$ for the $z_i = \pm z_j$ hyperplanes.  Let $H_c(B_n, \hfr)$ be the associated rational Cherednik algebra.  We will choose the set of simple reflections $s_0, s_1, ..., s_{n - 1}$ so that $s_0$ is the reflection through the hyperplane $z_1 = 0$, negating the first coordinate, and so that $s_i$, for $0 < i < n$, is the reflection through the hyperplane $z_i = z_{i + 1}$, transposing the $i^{th}$ and $(i + 1)^{st}$ coordinates.

Recall that the irreducible representations of the Coxeter group $B_n$ are naturally labeled by pairs of partitions, or \emph{bipartitions}, $\lambda = (\lambda_1, \lambda_2) \vdash n$ of $n$ (see, for example, \cite{GP}).  In particular, the simple objects in $\oscr_c(B_n, \hfr)$ are also labeled by bipartitions, with the bipartition $\lambda$ corresponding to the irreducible representation $L(\lambda) := L(V_\lambda)$, where $V_\lambda$ is the associated irreducible representation of $B_n$ and $L(V_\lambda)$ is the unique simple quotient of the Verma module $\Delta_c(V_\lambda)$ attached to $V_\lambda$.  The representation theory of $H_c(B_n, \hfr)$ is much richer than that of $H_c(S_n, \hfr_{S_n})$, and in particular the latter algebra may admit many nonisomorphic irreducible finite-dimensional representations.

Any parabolic subgroup of $B_n$ is conjugate to a unique parabolic subgroup of the form $B_l \times S_{n_1} \times \cdots \times S_{n_k}$ for some nonnegative integers $k, l \geq 0$ and positive integers $n_1 \geq \cdots \geq n_k > 0$ with $l + \sum_i n_i \leq n$.  By \cite[Theorem 1.2]{BEG}, the only such parabolic subgroups whose rational Cherednik algebras admit nonzero finite-dimensional representations are those of the form $B_l \times S_e^k$.  Let $W'$ be such a parabolic subgroup, and let $L$ be an irreducible finite-dimensional representation of $H_c(W', \hfr_{W'})$.  Then $L$ is of the form $L(V_\lambda \otimes \triv_e^{\otimes k})$, where $\lambda$ is a bipartition of $l$, $V_\lambda$ is the associated irreducible representation of $B_l$, and $\triv_e$ denotes the trivial representation of $S_e$.  If $k > 0$ and such a finite-dimensional irreducible representation exists, the parameter $c_2$ must be of the form $r/e$ for some integer $r$ relatively prime to $e$.

The fixed space $\hfr^{W'} \subset \cplx^n$ consists of those points $(z_1, ..., z_n)$ such that both $z_i = 0$ for $1 \leq i \leq l$ and also for $0 \leq m < k$ we have $z_{l + 1 + me + i} = z_{l + 1 + me + j}$ for $1 \leq i, j \leq e$.  For $1 \leq i \leq k$ let $x_i$ denote the coordinate of $z_{l + 1 + ie + 1}$ in $\hfr^{W'}$, and for $1 \leq j \leq n - l - ke$ let $y_j$ denote the coordinate $z_{j + l + ke}$, so that $\hfr^{W'}$ is identified with $\cplx^k \oplus \cplx^{n - l - ke}$ where the $x_i$ give the standard coordinates for $\cplx^k$ and the $y_j$ give the standard coordinates for $\cplx^{n - l - ke}$.  The natural complement $N_{W'}$ to $W'$ in its normalizer $N_{B_n}(W')$ is isomorphic to $B_k \times B_{n - l - ek}$ compatibly with the natural reflection representation of the latter group on $\cplx^k \oplus \cplx^{n - l - ke}$.  In particular, $N_{W'} = N_{W'}^{ref}$.  Each parabolic subgroup $W'' \subset B_n$ containing $W'$ in corank 1 is conjugate to a unique parabolic subgroup appearing among the five following cases; the form of the fixed hyperplane $\hfr^{W''} \subset \hfr^{W'}$ is listed after the parabolic subgroup $W''$:

(1) $B_{l + k} \times S_e^{k - 1} \ \ \ \ \ \ \ \ \ \ \ \ \ \ \ \ \ \ \ \ \ \ \ \ \ \ \ \ x_i = 0$

(2) $B_l \times S_{2e} \times S_e^{k - 2} \ \ \ \ \ \ \ \ \ \ \ \ \ \ \ \ \ \ \ \ \ \ \ x_i = \pm x_j$

(3) $B_{l + 1} \times S_e^k \ \ \ \ \ \ \ \ \ \ \ \ \ \ \ \ \ \ \ \ \ \ \ \ \ \ \ \ \ \ \ y_i = 0$

(4) $B_l \times S_e^k \times S_2 \ \ \ \ \ \ \ \ \ \ \ \ \ \ \ \ \ \ \ \ \ \ \ \ \ \ \ y_i = \pm y_j$

(5) $B_l \times S_{e + 1} \times S_e^{k - 1} \ \ \ \ \ \ \ \ \ \ \ \ \ \ \ \ \ \ \ \ \ x_i = \pm y_j.$

\noindent The only such parabolic subgroups in which $W'$ is self-normalizing are those of type (5).  Furthermore, as the longest element of any Coxeter group of type $B$ acts by -1 on its reflection representation, it follows that endowing $V_\lambda \otimes \triv_e^{\otimes k}$ with the trivial representation of $N_{W'}$ gives equivariant structure to $L(V_\lambda \otimes \triv_e^{\otimes k})$.  In particular, $I = N_{W'} = N_{W'}^{ref}$, the cocycle $\mu \in Z^2(I, \cplx^\times)$ is trivial, and $\hfr^{W'}_{L-reg}$ is the complement in $\hfr^{W'}$ of the hyperplanes of the forms (1) - (4).  In particular, by Theorem \ref{coxeter-presentation} we have an isomorphism of algebras $$\mathcal{H}(c, B_l \times S_e^k, L(V_\lambda \otimes \triv_e^{\otimes k}), B_n) \cong \mathsf{H}_{q_1, q_2}(B_k) \otimes \mathsf{H}_{q_3, q_4}(B_{n - l - ke}),$$ where $q_i \in \cplx^\times$ is the complex number such that the monodromy operator (appropriately scaled) associated to the hyperplanes of type (i), as listed above, satisfies the quadratic relation $(T - 1)(T + q_i) = 0$.  Here $q_1$ and $q_3$ are associated to the reflections through hyperplanes $x_i = 0$ and $y_i = 0$, respectively, and $q_2$ and $q_4$ are associated to the reflections through hyperplanes $x_i = \pm x_j$ and $y_i = \pm y_j$, respectively.  In particular, Theorem \ref{simplified-coxeter-theorem} holds in type $B$.

By Remark \ref{reducible-reduction-prop} the parameter $q_2$ associated to the inclusion $B_l \times S_e^k \subset B_l \times S_{2e} \times S_e^{k - 2}$ can be computed using the inclusion $S_e^2 \subset S_{2e}$ and the finite-dimensional irreducible representation $L(\triv_e^{\otimes 2})$ of $H_c(S_e^2, \hfr_{S_e^2})$.  This case was treated in Section \ref{type-A}, and we have $q_2 = 1$.  Similarly, $q_4$ can by computed using the inclusion $1 \subset S_2$, where $S_2$ is generated by a reflection in $B_n$ associated to a hyperplane $z_i = \pm z_j$ for any $i, j > l + ke$, giving $q_4 = e^{-2 \pi i c_2}.$  The computation of the parameters $q_1$ and $q_3$ reduce to computing the parameters $q(c, B_k \times S_e, L(V_\lambda \otimes \triv_e), B_{k + e})$ and $q(c, B_k, L(V_\lambda), B_{k + 1}),$ respectively, which in turn can be computed using Theorem \ref{quadratic-relations-theorem}.

We now explain how to compute the parameter $q(c, B_n, L(V_\lambda), B_{n + 1})$ from only the parameter $c$ and the bipartition $\lambda \vdash n$.  Let $p = e^{-2 \pi i c_1}$ and $q = e^{-2 \pi i c_2}$, so that $\mathsf{H}_{p, q}(B_n)$ is the Hecke algebra appearing in the $KZ$ functor for $\oscr_c(B_n, \hfr_{B_n})$, where the parameter $p$ is associated with reflections through hyperplanes $z_i = 0$ and $q$ is associated with reflections through hyperplanes $z_i = \pm z_j$.  Let $w_0$ denote the longest element of $B_{n + 1}$, let $w_0'$ denote the longest element of $B_n$, and let $z_{w_0w_0'}$ denote the associated central element in $\mathsf{H}_{p, q}(B_n)$.  For $1 \leq i \leq n$, let $t_i \in B_n$ denote the reflection $t_i := s_{i - 1} \cdots s_1 s_0 s_1 \cdots s_{i - 1}$ negating the $i^{th}$ coordinate.

Fix a bipartition $\lambda = (\lambda^{(1)}, \lambda^{(2)}) \vdash n$ of $n$, and for $i = 1,2$ let $\lambda_1^{(i)} \geq \cdots \geq \lambda^{(i)}_{l_i} > 0$ be the parts of the partition $\lambda^{(i)}$.  Recall that we may view $\lambda$ as a pair of Young diagrams in the following way.  Refer to an element $b = (x, y, i) \in \ints^{> 0} \times \ints^{> 0} \times \{1, 2\}$ as a \emph{box}.  A finite subset $Y \subset \ints^{> 0} \times \ints^{> 0} \times \{1, 2\}$ is called a \emph{Young diagram} if whenever $Y$ contains the box $(x, y, i)$ it also contains all boxes of the form $(x', y', i)$ for positive integers $x', y'$ satisfying $1 \leq x' \leq x$ and $1 \leq y' \leq y$.  Let $YD(\lambda) \subset \ints^{> 0} \times \ints^{> 0} \times \{1, 2\}$ be the Young diagram consisting of those boxes $(x, y, i)$ such that $y \leq l_i$ and $x \leq \lambda^{(i)}_y$.  Define the \emph{content}, with respect to the parameters $p$ and $q$, of a box $b = (x, y, i)$ to be $q^{x - y}p^{-1}$ if $i = 1$ and to be $-q^{x - y}$ if $i = 2$.  Denote the content of $b$ by $ct_{p, q}(b)$.

\begin{definition} \label{zpq-def} Given a bipartition $\lambda \vdash n$ of $n$ and parameters $p, q \in \cplx^\times$ for the Hecke algebra $\mathsf{H}_{p, q}(B_n)$, define the scalar $z_{p, q}(\lambda) \in \cplx$ by $$z_{p, q}(\lambda) := (1 - p)q^n + (1 - q)^2q^{n - 1}p\sum_{b \in YD(\lambda)} ct_{p, q}(b).$$\end{definition}

\begin{remark} Our definition of content differs slightly from the definition of content appearing in \cite[Section 10.1.4]{GP} because we choose a different convention for the quadratic relations satisfied by the generators $T_i$ of $\mathsf{H}_{p, q}(B_n)$, i.e. that the quadratic relations should be divisible by $(T - 1)$ rather than $(T + 1)$.  This is natural from the perspective of the $KZ$ functor.  Our algebra $\mathsf{H}_{p, q}(B_n)$ is isomorphic to the algebra $\mathsf{H}_{p^{-1}, q^{-1}}(B_n)$ appearing in \cite{GP} under the isomorphism induced by the assignments $T_s \mapsto q_s^{-1}T_s$.  The inversion of the parameters explains the discrepancy in the definition of content.\end{remark}

\begin{proposition} \label{z-elt-type-B} In the notation of Lemma \ref{z-props-lemma}, $q^{l(w_0w_0')} = pq^{2n}$, and the central element $z_{w_0w_0'} \in \mathsf{H}_{p, q}(B_n)$ has the following expansion in the $T_w$-basis: $$z_{w_0w_0'} = (1 - p)q^n + (1 - q)^2\sum_{i = 1}^n q^{n - i}T_{t_i}.$$  Moreover, $z_{w_0w_0'}$ acts on any irreducible representation of $\mathsf{H}_{p, q}(B_n)$ lying in the block of $\mathsf{H}_{p, q}(B_n)\mhyphen\text{mod}_{f.d.}$ corresponding via the $KZ$ functor to the block of $\oscr_c(B_n, \hfr_{B_n})$ containing $L(V_\lambda)$ by the scalar $z_{p, q}(\lambda)$ defined in Definition \ref{zpq-def}.\end{proposition}

\begin{proof} The expressions for $z_{w_0w_0'}$ and $q^{l(w_0w_0')}$ follow immediately from standard calculations in $\mathsf{H}_{p, q}(B_{n + 1})$ using the reduced expression $s_n \cdots s_1s_0s_1 \cdots s_n$ for $w_0w_0'$, where $s_0$ is the simple reflection through the hyperplane $z_1 = 0$ and, for $i > 0$, $s_i$ is the simple reflection through the hyperplane $z_i = z_{i + 1}$.  To show that $z_{w_0w_0'}$ acts by $z_{p, q}(\lambda)$ on the irreducibles in the block of $\mathsf{H}_{p, q}(B_n)\mhyphen\text{mod}_{f.d.}$ appearing in the theorem, it suffices to show that $z_{w_0w_0'}$ acts by $z_{p, q}(\lambda)$ on $KZ(\Delta_c(V_\lambda))$.  By a standard deformation argument, it suffices to prove this for generic parameters $p, q, c$.  For generic parameters, $KZ(\Delta_c(V_\lambda))$ is isomorphic to the irreducible representation $V_\lambda^{p, q}$ of $\mathsf{H}_{p, q}(B_n)$ described in \cite[Theorem 10.1.5]{GP} in terms of Hoefsmit's matrices, and that $z_{w_0w_0'}$ acts by the scalar $z_{p, q}(\lambda)$ on $V_\lambda^{p, q}$ then follows immediately from the explicit description of the diagonal action of the elements $T_{t_i}$ on the standard Young tableau basis of $V_\lambda^{p, q}$.\end{proof}

In particular, by Theorem \ref{quadratic-relations-theorem}, the canonical generator $T_{B_n, L(V_\lambda), B_{n + 1}}$ of the algebra $\mathcal{H}(c, B_n, L(V_\lambda), B_{n + 1})$ satisfies the same quadratic relation as the matrix $$\begin{pmatrix} 0 & pq^{2n} \\ 1 & z_{p, q}(\lambda) \end{pmatrix},$$ i.e. $$T^2 = z_{p, q}(\lambda)T + pq^{2n}.$$  Rescaling appropriately, one obtains the parameter $q(c, B_n, L(V_\lambda), B_{n + 1})$.  Note that when $q$ is a primitive $e^{th}$ root of unity, this quadratic relation depends only on the $e$-cores of the components of $\lambda$.  

\subsection{Type D} \label{type-D}  In this section we will show that Theorem \ref{simplified-coxeter-theorem} holds in type $D$ and that the study of the generalized Hecke algebras $\hcal(c, W', L, W)$ when $W$ is of type $D$ largely reduces to the case in which $W$ is of type $B$.

Recall that for $n \geq 4$ the reflection group $D_n$ of type $D$ and rank $n$ is the subgroup of $B_n$ of index 2 consisting of those elements acting on $\cplx^n$ with an even number of sign changes.  $D_n$ is an irreducible reflection group with reflection representation $\cplx^n$ in this way, generated by reflections through the hyperplanes $z_i = \pm z_j$ for $1 \leq i < j \leq n$.  If $s_0, ..., s_{n - 1}$ are the simple reflections for $B_n$ introduced in the previous section, then the reflections $s_1', s_1, s_2, ..., s_{n - 1}$, where $s_1' = s_0s_1s_0$ is the reflection through the hyperplane $z_1 = -z_2$, form a system of simple reflections for $D_n$ with respect to which it is a Coxeter group.

The irreducible complex representations of $D_n$ are easily described in terms of those of $B_n$ recalled in the previous section.  In particular, when $\lambda = (\lambda^1, \lambda^2)$ is a bipartition of $n$ for which $\lambda^1 \neq \lambda^2$, then the restriction of $V_\lambda$ to $D_n$ is irreducible, and $V_{(\lambda^1, \lambda^2)}$ and $V_{(\lambda^2, \lambda^1)}$ are isomorphic as representations of $D_n$.  When $\lambda^1 = \lambda^2$, i.e. when the bipartition $\lambda$ is \emph{symmetric}, the restriction of $V_\lambda$ to $D_n$ splits as a direct sum of two non-isomorphic irreducible representations $V_\lambda^+$ and $V_\lambda^-$.  All irreducible representations of $D_n$ appear in this way, and the only isomorphisms among these representations are those of the form $V_{(\lambda^1, \lambda^2)} \cong V_{(\lambda^2, \lambda^1)}$.

All reflections in $D_n$ are conjugate, so a parameter for the rational Cherednik algebra of type $D$ is determined by a single number $c \in \cplx$.   It follows immediately from the definition by generators and relations that the rational Cherednik algebra $H_c(D_n, \cplx^n)$ embeds naturally in the type $B$ algebra $H_{0, c}(B_n, \cplx^n)$, where in the latter algebra the parameter takes value $0$ on reflections through hyperplanes $z_i = 0$ and value $c$ on reflections through hyperplanes $z_i = \pm z_j$.  Let $q = e^{-2 \pi i c}$ be the parameter for the Hecke algebra $\mathsf{H}_q(D_n)$ whose category of finite-dimensional modules is the target of the $KZ$ functor.  Note similarly that the Hecke algebra $\mathsf{H}_q(D_n)$ embeds naturally as a subalgebra of the Hecke algebra $H_{1, q}(B_n)$ compatibly with the $T_w$ bases (see \cite[Section 10.4.1]{GP}); note that $T_{s_0}^2 = 1$.  This embedding is compatible with the $KZ$ functors in the obvious way.  It is shown in \cite{SS} that when the bipartition $\lambda$ is symmetric the irreducible representations $L_c(V_\lambda^\pm)$ are always infinite dimensional.  In particular, the finite-dimensional irreducible representations of $H_c(D_n, \cplx^n)$ always extend to irreducible representations of $H_{0, c}(B_n, \cplx^n)$, although not uniquely.

Suppose $W' \subset D_n$ is a standard parabolic subgroup such that $H_c(W', \hfr_{W'})$ admits a finite-dimensional irreducible representation $L$.  The irreducible parabolic subgroups of $D_n$ are of types $A$ and $D$.  We will now describe a procedure for producing a presentation of the algebra $\hcal(c, W', L, D_n)$ in the form appearing in Theorem \ref{simplified-coxeter-theorem}, and in particular we will see that Theorem \ref{simplified-coxeter-theorem} holds in type $D$.  Clearly, by tensoring with the sign character of $D_n$, we may assume $c > 0$.

First suppose that the decomposition of $W'$ into a product of irreducible parabolic subgroups involves no factors of type $D$.  Then by \cite[Theorem 1.2]{BEG} we can assume that $c = r/e$ for positive relatively prime integers $r \geq 1, e \geq 2$, that $W'$ is of the form $S_e^k$ for some integer $k > 0$ such that $ke \leq n$, and that $L = L(\triv_e^{\otimes k})$.  It follows from Remark \ref{lowest-weight-remark} that the inertia group $I$ equals the complement $N_{W'}$ to $W'$ in $N_{D_n}(W')$ and that the cocycle $\mu \in Z^2(N_{W'}, \cplx^\times)$ is trivial.  In particular, Theorem \ref{simplified-coxeter-theorem} holds in this case.  A detailed description of the group $N_{W'}$ and its maximal reflection subgroup $N_{W'}^{ref}$ (typically a proper subgroup of $N_{W'}$) may be found in \cite{How}.  As usual, the parameter $q_{W', L}$ associated to the Hecke algebra $\mathsf{H}_{q_{W', L}}(N_{W'}^{ref})$ can be computed using Theorem \ref{quadratic-relations-theorem}.  The parabolic subgroups $W'' \subset D_n$ containing $W'$ in corank 1 and in which $W'$ is not self-normalizing are of the form (1) $S_e^k \times S_2$, (2) $S_e^{k - 2} \times S_{2e}$, (3) (in the case $e = 2$) $S_2^{k - 3} \times D_4$, and (4) (in the case $e = 4$) $S_4^{k - 1} \times D_4$.  By Remark \ref{reducible-reduction-prop}, parameter computations in these cases reduce to the cases, respectively, (1) $1 \subset S_2$, (2) $S_e^2 \subset S_{2e}$, (3) $S_2^3 \subset D_4$, and (4) $S_4 \subset D_4$.  As discussed in Section \ref{type-A} about type $A$, the quadratic relation in case (1) is $(T - 1)(T + q) = 0$ with $q = e^{-2 \pi i c}$, and the quadratic relation in case (2) is $T^2 = 1$.  To compute the quadratic relation in case (3), we use Theorem \ref{quadratic-relations-theorem} again.  In particular, letting $w_0$ denote the longest element of $D_4$ and $w_0'$ the longest element of $S_2^3$, we have $q^{l(w_0w_0')} = (-1)^9 = -1$, and computations in the computer algebra package CHEVIE in GAP3 \cite{GAP1, GAP2} show that the central element $z_{w_0w_0'}$ acts on the trivial representation of $H_{-1}(S_2)^{\otimes 3}$ by the scalar 2.  By Theorem \ref{quadratic-relations-theorem} the quadratic relation appearing in case (3) is $T^2 = 2T - 1$, i.e. $(T - 1)^2 = 0$.  Similarly, one obtains the quadratic relation $(T - 1)^2 = 0$ in remaining case (4) as well.

Now, consider the remaining case in which the decomposition of $W'$ into a product of irreducible parabolic subgroups involves a factor of type $D$.  Then again by \cite[Theorem 1.2]{BEG} we can assume that $W'$ is conjugate to a parabolic subgroup of the form $D_l \times S_e^k$ for some integers $e \geq 2, l \geq 4, k \geq 0$ such that $l + ke \leq n$, and the finite-dimensional irreducible representation $L$ is isomorphic to $L_c(V_\lambda \otimes \triv_e^{\otimes k})$ for some bipartition $\lambda = (\lambda^1, \lambda^2)$ with $\lambda^1 \neq \lambda^2$.  The parameter $c$ must again be of the form $c = r/e$ for some positive integer $r$ relatively prime to $e$.  Furthermore, it follows from \cite[Lemma 4.2, Corollary 4.3]{Losevsupports} that $H_c(D_l, \cplx^l)$ admits no nonzero finite-dimensional representations when $e$ is odd, so we may assume that $e$ is even.  To see this, consider a parameter $c = r/e$ for the rational Cherednik algebra $H_c(D_n, \cplx^n)$, where $e > 2$ is an odd positive integer and $r$ an integer relatively prime to $e$.  The algebra $H_c(D_n, \cplx^n)$ admits nonzero finite-dimensional representations if and only if the algebra $H_{0, c}(B_n, \cplx^n)$ admits nonzero finite-dimensional representations.  In the notation from \cite[Section 4]{Losevsupports}, in this case we have $\kappa = -r/e$ and $(s_1, s_2) = (0, -\frac{e}{2r})$.  Indexes 1 and 2 are not equivalent under the equivalence relation $\sim_{(0, c)}$ as we have $s_2 - s_1 = -\frac{e}{2r} \notin \frac{1}{r}\ints = \kappa^{-1}\ints + \ints$, so by \cite[Lemma 4.2]{Losevsupports} the category $\oscr_{(0, c)}(B_n, \cplx^n)$ decomposes as a direct sum of outer tensor products of categories $\oscr$ associated to reflection groups $S_k$ with reflection representation $\cplx^k$, for various $k$, in a manner preserving supports \cite[Corollary 4.3]{Losevsupports}.  As the rational Cherednik algebras associated to the reducible reflection representations $(S_k, \cplx^k)$ have no nonzero finite-dimensional representations for any parameter values, it follows that the rational Cherednik algebra $H_{r/e}(D_n, \cplx^n)$ also has no nonzero finite-dimensional representations.  We will therefore assume that $e > 1$ is a positive even integer.

As the fixed space of $W'$ equals the fixed space of the parabolic subgroup $B_l \times S_e^k$ of $B_n \supset D_n$, it follows that $N_{D_n}(W') = D_n \cap N_{B_n}(B_l \times S_e^k)$.  As $\lambda_1 \neq \lambda_2$, the representation $V_\lambda \otimes \triv_e^{\otimes k}$ of $W'$ extends to a representation of $B_l \times S_e^k$, and we've seen in Section \ref{type-B} that such a representation extends to a representation of $N_{B_n}(B_l \times S_e^k)$.  In particular, it follows that the inertia group $I$ is maximal, i.e. equals $N_{W'}$, and that the cocycle $\mu \in Z^2(N_{W'}, \cplx^\times)$ is trivial, so Theorem \ref{simplified-coxeter-theorem} holds in this remaining case in type $D$.  As discussed in ``Case 1'' of the section ``Type D'' of \cite{How}, in this case $N_{W'}$ equals $N_{W'}^{ref}$ and is isomorphic to $B_k \times B_{n - l - ke}$ as a reflection group acting on $(\cplx^n)^{W'} \cong \cplx^k \oplus \cplx^{n - l - ke}$ in a manner completely analogous to the discussion in Section \ref{type-B}.  In particular, in this case we have $$\hcal(c, D_l \times S_e^k, L_c(V_\lambda \otimes \triv_e^{\otimes k}), D_n) \cong \mathsf{H}_{q_1, 1}(B_k) \otimes \mathsf{H}_{q_2, q}(B_{n - l - ke})$$ where $q_1$ and $q_2$ are associated to the short roots of $B_k$ and $B_{n - l - ke}$, respectively, $q = e^{-2 \pi i c}$, $q_1 = q(c, D_l \times S_e, L_c(V_\lambda \otimes \triv_e), D_{l + e})$, and $q_2 = q(c, D_l, L_c(V_\lambda), D_{l + 1}).$  The following result reduces the computation of these parameters to the type $B$:

\begin{proposition} In the setting of the previous paragraph, we have $$q(c, D_n, L_c(V_\lambda), D_{n + 1}) = q((0, c), B_n, L_{(0, c)}(V_\lambda), B_{n + 1})$$ and $$q(c, D_n \times S_e, L_c(V_\lambda \otimes \triv_e), D_{n + e}) = q((0, c), B_n \times S_e, L_{(0, c)}(V_\lambda \otimes \triv_e), B_{n + e}).$$\end{proposition}

\begin{proof} We consider $q(c, D_n, L_c(V_\lambda), D_{n + 1})$ first.  Let $l_D$ denote the length function on $D_{n + 1}$ with respect to the simple reflections $s_1, s_1', ..., s_n$ introduced above, and let $l_B$ denote the length function on $B_{n + 1}$ with respect to the simple reflections $s_0, s_1, ..., s_n$.  Let $w_{0, D}, w_{0, B}, w_{0, D}',$ and $w_{0, B}'$ denote the longest elements of the Coxeter groups $D_{n + 1}, B_{n + 1}, D_n$ and $B_n$, respectively.  Then $w_{0, D}w_{0, D}' = s_0w_{0, B}w_{0, B}' = w_{0, B}w_{0, B}'s_0.$  Regard $\mathsf{H}_q(D_n)$ as a subalgebra of $H_{1, q}(B_n)$ via the $T_w$-bases.  We then have $T_{w_{0, D}w_{0, D}'} = T_{s_0}T_{w_{0, B}w_{0, B}'} = T_{w_{0, B}w_{0, B}'}T_{s_0}$ and $T_{w_0'}^2 = 1$.  In particular, $T_{w_{0, D}w_{0, D}'}^2 = T_{w_{0, B}w_{0, B}'}^2$ and it follows from Proposition \ref{z-elt-type-B} and the definitions that $z_{w_{0, D}w_{0, D}'} = T_{s_0}z_{w_{0, B}w_{0, B}'} = z_{w_{0, B}w_{0, B}'}T_{s_0}$.  As the representation $V_\lambda$ extends to $B_n$, we can choose the operator by which $w_{0, B}w_{0, B}'$ acts on $L$ as in Theorem \ref{quadratic-relations-theorem} to be $T_{s_0}$.  By Theorem \ref{quadratic-relations-theorem}, the generator of monodromy $T \in \hcal(c, D_n, L_c(V_\lambda), D_{n + 1})$ satisfies the quadratic relation $$T^2 = (T_0z_{w_{0, D}w_{0, D}'})|_LT + q^{l_D(w_{{0, D}w_{0, D}'})} = z_{1, q}(\lambda)T + q^{2n}.$$  This is precisely the quadratic relation obtained for the generator of monodromy generating the algebra $T \in \hcal(c, D_n, L_c(V_\lambda), D_{n + 1})$, as shown in Section \ref{type-B}, and the first equality follows.

The second equality follows by a similar argument.\end{proof}

\subsection{Parameters for Generalized Hecke Algebras in Exceptional Types}

We will now describe the parameters arising for the generalized Hecke algebras in exceptional type.  In each row of the following Table \ref{parameters-table}, $W$ is an irreducible finite Coxeter group and $W' \subset W$ is a corank-1 parabolic subgroup of $W$ that is not self-normalizing in $W$.  The complex number $c$ is a parameter for the rational Cherednik algebra $H_c(W, \hfr)$ such that $H_c(W', \hfr_{W'})$ admits nontrivial finite-dimensional representations, and $\lambda$ is an irreducible representation of $W'$ such that $L(\lambda)$ is a finite-dimensional irreducible representation of $H_c(W', \hfr_{W'})$.  For a given $W'$, all $c$ of the form $1/d$ such that $H_c(W', \hfr_{W'})$ admits a finite-dimensional irreducible representation are given, and for each $W', c$ a complete list of lowest weights $\lambda$ with $\dim L(\lambda) < \infty$ is given. Finally, $q(c, W', L, W)$ is a complex number such that the monodromy operator $T$ associated to the tuple $(c, W', L(\lambda), W)$, after an appropriate rescaling if necessary, satisfies the quadratic relation $$(T - 1)(T + q(c, W', L, W)) = 0.$$  Where appropriate, $q(c, W', L, W)$ is given as a power of the ``$KZ$ parameter'' $q = e^{-2 \pi i c}$.  Table \ref{parameters-table} includes every case needed to give presentations for the generalized Hecke algebras arising in types $E, H$ and $I$; this data was obtained by using Theorem \ref{quadratic-relations-theorem} and computations with the computer algebra package CHEVIE in GAP3 \cite{GAP1, GAP2} as well as SAGE.  Type $F$ can be handled by these same methods, although the description of the relevant parameters $c$ and irreducible finite-dimensional representations for the parabolic subgroups of types $B$ and $C$ arising in this case is more complicated to display in a table.  We will give the counts of modules of given support in $\oscr_c(F_4, \hfr)$ in the unequal parameter case later in Section \ref{typeF4}.

In Table \ref{parameters-table} and below, we will list the parameter for groups $B_n$ in the form $(c_1, c_2) \in \cplx^2$, where $c_1$ specifies the value of the parameter on the short roots.  In the last row of the table, the parameter $(1/2, c_2)$ for the even dihedral group $I_2(2m)$ indicates that the parameter takes value $1/2$ on those reflections conjugate to the nontrivial element of the chosen parabolic subgroup $A_1$ and arbitrary value $c_2 \in \cplx$ on the remaining parameters.  The relevant parameter values and lists of finite-dimensional irreducible representations for the groups of type $D$ are obtained by a standard reduction to type $B$ (as in Section \ref{type-D}), where these lists are easily produced using the methods of \cite{Losevsupports}.  The labeling used for irreducible representations of the exceptional groups is compatible with that appearing in \cite{GGJL}; in particular, we denote the trivial representation by $\triv$, the reflection representation by $V$ (and its Galois conjugate in type $H$ by $\widetilde{V}$), and other representations are denoted in the form $\varphi_{x, y}$ where $x$ indicates the dimension of the representation and $y$ indicates its $b$-invariant, i.e. the lowest degree in the grading of the coinvariant algebra in which the representation appears.  The labels $\varphi_{x, y}$ are compatible with the labels appearing in the GAP3 computer algebra package.  This is a different labeling system than appears in some standard references, e.g. \cite{GP}, although it is simple to convert between this labeling system and others using the tables appearing in \cite[Appendix C]{GP}.  Irreducible representations of groups of type $D$ are labeled by (unordered) pairs of partitions, in the standard way.

\begin{remark} In all cases listed except the case of $E_7$ at parameter $1/10$, the associated monodromy operator $T$ has an eigenvalue equal to 1, and in particular no rescaling was needed to list the parameter $q(c, W', L(\lambda), W)$; the monodromy operator $T$ associated to irreducible representation $L_{1/10}(V)$ of $H_{1/10}(E_7, \hfr)$ associated to the inclusion $E_7 \subset E_8$ satisfies the quadratic relation $(T + e^{\pi i/5})^2 = 0$, which is of the form $(T - 1)^2 = 0$ after rescaling $T$.  In the cases in which $T$ has an eigenvalue equal to $1$, the parameter $q(c, W', L, W)$ is necessarily equal to $q^{l(w_0w_0')}$, where $q$ is the parameter appearing in the relevant $KZ$ functor and $q^{l(w_0w_0')}$ is as in Theorem \ref{quadratic-relations-theorem}, and this covers all other cases in the table.  We remark that there are other cases, not relevant for the exceptional groups, in which $T$ does not have an eigenvalue equal to 1 before rescaling; for example, $L_{(-1/6, 1/3)}(\emph{triv})$ is a finite-dimensional irreducible representation of $H_{(-1/6, 1/3)}(B_2)$ and the quadratic relation associated to the inclusion $B_2 \subset B_3$ is $(T + p)^2 = 0$, where $p = e^{-\pi i/3}.$\end{remark}

\begin{remark} In all cases we have computed, the parameter $q(c, W', L, W)$ depends only on $c, W'$, and $W$, and notably not on the finite-dimensional irreducible representation $L$.  This fact is reflected in Table \ref{parameters-table}, where we list all relevant lowest weights $\lambda$ for each pair $(W', c)$ in the same row.  It would be interesting to have a conceptual explanation for this fact.\end{remark}

\begin{longtable}{llllr}
\caption{Parameters for Generalized Hecke Algebras}
\label{parameters-table}\\
\endfirsthead
\caption{(continued)}\\
\endhead
$W'$ & $c$ & $\lambda$ & $W$ & $q(c, W', L(\lambda), W)$ \\
\hline \hline
$A_n \times A_n$ & $1/(n + 1)$ & $\triv$ & $A_{2n + 1}$ & $1$\\
\hline
$A_1^3$ & $1/2$ & $\triv$ & $D_4$ & $-1$\\
\hline
$A_3$ & $1/4$ & $\triv$ & $D_4$ & $-1$\\
\hline
$D_4$ & $1/6$ & $\triv$ & $D_5$ & $q^2$ \\
$D_4$ & $1/4$ & $\triv$ & $D_5$ & $1$ \\
$D_4$ & $1/2$ & $\triv, (3,1)$ & $D_5$ & $1$ \\
\hline
$A_5$ & $1/6$ & $\triv$ & $D_6$ & $-1$\\
\hline
$D_5$ & $1/8$ & $\triv$ & $D_6$ & $q^2$\\
\hline
$D_4 \times A_1$ & $1/2$ & $\triv \otimes \triv, (3, 1) \otimes \triv$ & $D_6$ & $-1$\\
\hline
$A_3^2$ & $1/4$ & $\triv$ & $D_7$ & -1\\
\hline
$D_6$ & $1/10$ & $\triv$ & $D_7$ & $q^2$\\
$D_6$ & $1/6$ & $\triv$ & $D_7$ & $1$\\
$D_6$ & $1/2$ & $\triv, (0, 3^2), (1, 5), (2, 4)$ & $D_7$ & $1$\\
\hline
$A_5$ & $1/6$ & $\triv$ & $E_6$ & $-1$\\
\hline
$A_6$ & $1/7$ & $\triv$ & $E_7$ & $1$\\
\hline
$D_6$ & $1/10$ & $\triv$ & $E_7$ & $q^3$\\
$D_6$ & $1/6$ & $\triv$ & $E_7$ & $-1$\\
$D_6$ & $1/2$ & $\triv, (0, 3^2), (1, 5), (2, 4)$ & $E_7$ & $-1$\\
\hline
$E_6$ & $1/12$ & $\triv$ & $E_7$ & $q^3$\\
$E_6$ & $1/9$ & $\triv$ & $E_7$ & $1$\\
$E_6$ & $1/6$ & $\triv, V$ & $E_7$ & $-1$\\
$E_6$ & $1/3$ & $\triv, V, \Lambda^2V$ & $E_7$ & $1$\\
\hline
$A_7$ & $1/8$ & $\triv$ & $E_8$ & $-1$\\
\hline
$D_7$ & $1/12$ & $\triv$ & $E_8$ & $-1$\\
$D_7$ & $1/4$ & $\triv, (2,5)$ & $E_8$ & $-1$\\
\hline
$E_7$ & $1/18$ & $\triv$ & $E_8$ & $q^3$\\
$E_7$ & $1/14$ & $\triv$ & $E_8$ & $q$\\
$E_7$ & $1/10$ & $V$ & $E_8$ & $-1$\\
$E_7$ & $1/6$ & $\triv, V, \varphi_{15, 7}, \varphi_{21, 6}$ & $E_8$ & $-1$\\
$E_7$ & $1/2$ & $\triv, V, \phi_{15, 7}, \phi_{21, 6}, \phi_{27, 2}, \phi_{35, 13}, \phi_{189, 5}$ & $E_8$ & $-1$\\
\hline
$A_2$ & $(c_1, 1/3)$ & $\triv$ & $B_3$ & $e^{-6 \pi i c_1}$\\
\hline
$A_1^2$ & $(1/2, 1/2)$ & $\triv$ & $B_3$ & $-1$\\
\hline
$A_2$ & $1/3$ & $\triv$ & $H_3$ & $1$\\
\hline
$A_1 \times A_1$ & $1/2$ & $\triv$ & $H_3$ & $-1$\\
\hline
$I_2(5)$ & $1/5$ & $\triv$ & $H_3$ & $1$\\
\hline
$H_3$ & $1/10$ & $\triv$ & $H_4$ & $-1$\\
$H_3$ & $1/6$ & $\triv$ & $H_4$ & $-1$\\
$H_3$ & $1/2$ & $\triv, V, \widetilde{V}$ & $H_4$ & $-1$\\
\hline
$A_3$ & $1/4$ & $\triv$ & $H_4$ & $-1$\\
\hline
$A_1$ & $(1/2, c_2)$ & $\triv$ & $I_2(2m)$ & $(-1)^{m - 1}e^{-2 mc_2\pi i}$\\
\hline
\end{longtable}

\subsection{Type $E$}\label{typeE}

\subsubsection{Generalized Hecke Algebras for $E_6$}  The following table list all of the generalized Hecke algebras arising from the rational Cherednik algebra $H_c(E_6, \hfr_{E_6})$ of type $E_6$ for parameters of the form $c = 1/d$ such that $\oscr_c(E_6, \hfr_{E_6})$ is not semisimple, i.e. for those integers $d > 1$ dividing one of the fundamental degrees $2, 5, 6, 8, 9$ and $12$ of $E_6$.  The first column indicates the parameter value $c$.  The second column, labeled $W'$, lists a unique representative of each conjugacy class of parabolic subgroups of $E_6$ for which a nonzero finite-dimensional representation appears at the parameter value specified in the title of the table; if a conjugacy class is missing, no nonzero finite-dimensional irreducible representations exist for that class.  The column labeled $\lambda$ gives a complete list of the lowest weights $\lambda$ of the finite-dimensional irreducible representations $L_c(\lambda)$ of $H_c(W', \hfr_{W'})$.  By inspection, we see that in each case the inertia group is maximal, the 2-cocycle $\mu$ is trivial, and in particular Theorem \ref{simplified-coxeter-theorem} holds for algebras of type $E_6$.  Furthermore, from Table \ref{parameters-table} we see that the parameters for the generalized Hecke algebra does not depend on the choice of lowest weight for the finite-dimensional representation; therefore in the column labeled $\hcal$ we give the generalized Hecke algebra $\hcal(c, W', \lambda, E_6)$ common to each of the $\lambda$ appearing in a given row.  In the final column labeled $\#\irr$, we give the number of irreducible representations with support labeled by $W'$, obtained as the product of the number of $\lambda$ appearing in a given row with the number of irreducible representations of the corresponding generalized Hecke algebra.  As there are 25 irreducible representations of the group $E_6$, these numbers add to 25 for each parameter value.

Throughout, $q$ denotes the ``$KZ$ parameter'' $q := e^{-2 \pi i c}$.  The exact descriptions of the generalized Hecke algebras follow easily from parameters in Table \ref{parameters-table} and Howlett's detailed descriptions of the groups $N_{W'}^{ref}$ and $N_{W'}^{comp}$ appearing in \cite{How}; any semidirect products appearing in the description of the algebras $\hcal$ are given by the diagram automorphisms indicated in \cite{How}.  By convention, we list the parameters of 2-parameter Hecke algebras by giving the parameter for the short roots first, e.g. $\mathsf{H}_{p, q}(B_3)$ indicates that parameter $p$ is associated with the 3 reflections given by short roots and that parameter $q$ is associated with the remaining 6 reflections given by long roots.

The finite-dimensional irreducible representations of $H_c(E_6, \hfr_{E_6})$, if they exist, appear in the rows labeled by $E_6$.  For all parameters except $c = 1/2$, the list of the lowest weights of the finite-dimensional irreducible representations of $H_c(E_6, \hfr_{E_6})$ is obtained from results of Norton \cite{Norton}.  For $c = 1/2$, our table shows that there are no such finite-dimensional irreducible representations.

\begin{longtable}{llllr}
\caption{Refined Filtration by Supports for $E_6$}
\label{E6-table}\\
$c$ & $W'$ & $\lambda$ & $\hcal$ & $\#\irr$\\
\hline
$1/12$ & 1 & $\triv$ & $\mathsf{H}_q(E_6)$ & 24\\
& $E_6$ & $\triv$ & $\cplx$ & 1\\
\hline
$1/9$ & 1 & $\triv$ & $\mathsf{H}_q(E_6)$ & 24\\
& $E_6$ & $\triv$ & $\cplx$ & 1\\
\hline
$1/8$ & 1 & $\triv$ & $\mathsf{H}_q(E_6)$ & 24\\
&$D_5$ & $\triv$ & $\cplx$ & 1\\
\hline
1/6 & 1 & $\triv$ & $\mathsf{H}_q(E_6)$ & 20\\
&$D_4$ & $\triv$ & $\mathsf{H}_{q^2}(A_2)$ & 2\\
&$A_5$ & $\triv$ & $\mathsf{H}_{-1}(A_1)$ & 1\\
&$E_6$ & $\triv, V$ & $\cplx$ & 2\\
\hline
1/5 & 1 & $\triv$ & $\mathsf{H}_q(E_6)$ & 23\\
& $A_4$ & $\triv$ & $\mathsf{H}_q(A_1)$ & 2\\
\hline
1/4&1 & $\triv$ & $\mathsf{H}_q(E_6)$ & 19\\
&$A_3$ & $\triv$ & $\mathsf{H}_{-1, q}(B_2)$ & 3\\
&$D_4$ & $\triv$ & $\mathsf{H}_1(A_2)$ & 3\\
\hline
1/3 & 1 & $\triv$ & $\mathsf{H}_q(E_6)$ & 13\\
&$A_2$ & $\triv$ & $\ints/2\ints \ltimes \mathsf{H}_q(A_2^2)$ & 5\\
&$A_2^2$ & $\triv$ & $\mathsf{H}_{1, q}(G_2)$ & 4\\
&$E_6$ & $\triv, V, \Lambda^2V$ & $\cplx$ & 3\\
\hline
1/2 & 1 & $\triv$ & $\mathsf{H}_q(E_6)$ & 8\\
&$A_1$ & $\triv$ & $\mathsf{H}_q(A_5)$ & 4\\
&$A_1^2$ & $\triv$ & $\mathsf{H}_{1, -1}(B_3)$ & 4\\
&$A_1^3$ & $\triv$ & $\mathsf{H}_{-1}(A_1) \otimes \mathsf{H}_1(A_2)$ & 3\\
&$D_4$ & $\triv, (3, 1)$ & $\mathsf{H}_1(A_2)$ & 6\\
\hline
\end{longtable}

\subsubsection{Generalized Hecke Algebras for $E_7$}  In this section we produce a table for $E_7$ analogous to Table \ref{E6-table}, following the same conventions.  Again, we only list those parameters $c = 1/d$ for positive integers $d > 1$ dividing a fundamental degree of $E_7$ - the degrees of $E_7$ are $2, 6, 8, 10, 12, 14,$ and $18$.  The group $E_6$ has 60 irreducible representations.  By inspection, we see that Theorem \ref{simplified-coxeter-theorem} holds in type $E_7$.

Note that there are two distinct conjugacy classes of parabolic subgroups of $E_7$ of type $A_5$, while inside $E_6$ and $E_8$ there is only 1.  We denote a representative of the conjugacy class of parabolic subgroups of type $A_5$ appearing already in $E_6$ by $A_5'$, and we denote a representative of the remaining conjugacy class by $A_5''$.

There are also two distinct conjugacy classes of parabolic subgroups of $E_7$ of type $A_1^3$, although one of these is absent in Howlett's table \cite{How}.  The two classes can be distinguished by containment in parabolic subgroups of type $D_4$; we denote the class of parabolic subgroups of type $A_1^3$ contained in a parabolic subgroup of $D_4$ by $(A_1^3)'$ and the remaining class by $(A_1^3)''$.  The class $(A_1^3)'$ is treated in Howlett's paper; the complement $N_{(A_1^3)'}$ to $(A_1^3)'$ in its normalizer in $E_7$ acts on $\hfr_{(A_1^3)'}$ as a reflection group of type $B_3 \times A_1$.  In the remaining case $(A_1^3)''$, the complement in the normalizer acts in the fixed space as a reflection group of type $F_4$.  This can be seen as follows.  Computations in GAP3 \cite{GAP1, GAP2} show that this complement has order 1152.  This group has a decomposition as a semidirect product $N^{comp} \ltimes N^{ref}$, where $N^{ref}$ is a real reflection group of rank at most 4 and $N^{comp}$ is a finite group acting on $N^{ref}$ by diagram automorphisms.  From the classification of finite reflection groups, the only possibilities giving rise to groups of order 1152 are $N^{ref} = F_4$ and $N^{comp} = 1$, or $N^{ref} = D_4$ and $N^{comp} = S_3$ where $S_3$ acts on $D_4$ by the full group of diagram automorphisms.  As $(A_1^3)''$ is contained in parabolic subgroups of different types $A_1^4$ and $A_1 \times A_3$ in which it is not self-normalizing, it follows that the hyperplanes in the reflection representation of $N^{ref}$ cannot all be conjugate.  This rules out $N^{ref} = D_4$, and we conclude that the representation of the complement in the space $\hfr_{(A_1^3)''}$ is identified with the reflection representation of $F_4$.

When the denominator of $c$ is greater than 2, the lowest weights listed for the finite-dimensional irreducible representations of $E_7$ were determined by Norton \cite{Norton}, and we list those representations in the table below.  When the denominator of $c$ equals 2, we see that there are exactly 7 isomorphism classes of finite-dimensional irreducible representations of $H_c(E_7, \hfr)$.  In \cite{GGJL}, all but 7 possible lowest weights for these irreducible representations were ruled out.  In particular, these 7 representations are in fact finite-dimensional, and this completes the classification of finite-dimensional irreducible representations of the rational Cherednik algebras of type $E_7$.

\begin{longtable}{llllr}
\caption{Refined Filtration by Supports for $E_7$}
\label{E7-table}\\
$c$ & $W'$ & $\lambda$ & $\hcal$ & $\#\irr$\\
\hline
$1/18$ & 1 & $\triv$ & $\mathsf{H}_q(E_7)$ & 59\\
&$E_7$ & $\triv$ & $\cplx$ & 1\\
\hline
$1/14$ & 1 & $\triv$ & $\mathsf{H}_q(E_7)$ & 59\\
&$E_7$ & $\triv$ & $\cplx$ & 1\\
\hline
$1/12$ & 1 & $\triv$ & $\mathsf{H}_q(E_7)$ & 58\\
&$E_6$ & $\triv$ & $\mathsf{H}_{q^3}(A_1)$ & 2\\
\hline
$1/10$ & 1 & $\triv$ & $\mathsf{H}_q(E_7)$ & 57\\
&$D_6$ & $\triv$ & $\mathsf{H}_{q^3}(A_1)$ & 2\\
&$E_7$ & $V$ & $\cplx$ & 1\\
\hline
$1/9$ & 1 & $\triv$ & $\mathsf{H}_q(E_7)$ & 58\\
&$E_6$ & $\triv$ & $\mathsf{H}_1(A_1)$ & 2\\
\hline
$1/8$ & 1 & $\triv$ & $\mathsf{H}_q(E_7)$ & 56\\
&$D_5$ & $\triv$ & $\mathsf{H}_{q^2}(A_1) \otimes \mathsf{H}_q(A_1)$ & 4\\
\hline
$1/7$ & 1 & $\triv$ & $\mathsf{H}_q(E_7)$ & 58\\
&$A_6$ & $\triv$ & $\mathsf{H}_1(A_1)$ & 2\\
\hline
$1/6$ & 1 & $\triv$ & $\mathsf{H}_q(E_7)$ & 43\\
&$D_4$ & $\triv$ & $\mathsf{H}_{q, q^2}(B_3)$ & 6\\
&$A_5$ & $\triv$ & $\mathsf{H}_{-1}(A_1^2)$ & 1\\
& $A_5'$ & $\triv$ & $\mathsf{H}_{-1, q}(G_2)$ & 3\\
&$D_6$ & $\triv$ & $\mathsf{H}_{-1}(A_1)$ & 1\\
&$E_6$ & $\triv, V$ & $\mathsf{H}_{-1}(A_1)$ & 2\\
&$E_7$ &$\triv, V, \varphi_{15, 7}, \varphi_{21, 6}$ &$\cplx$&4\\
\hline
$1/5$ & 1 & $\triv$ & $\mathsf{H}_q(E_7)$ & 54\\
&$A_4$ & $\triv$ &$\ints/2\ints \ltimes \mathsf{H}_q(A_2)$&6\\
\hline
$1/4$ & 1 & $\triv$ & $\mathsf{H}_q(E_7)$ & 40\\
&$A_3$ & $\triv$ & $\mathsf{H}_{-1, q}(B_3) \otimes \mathsf{H}_q(A_1)$ & 10\\
&$D_4$ & $\triv$ & $\mathsf{H}_{q, 1}(B_3)$ & 10\\
\hline
$1/3$ & 1 & $\triv$ & $\mathsf{H}_q(E_7)$ & 32\\
&$A_2$ & $\triv$ & $\ints/2\ints \ltimes \mathsf{H}_q(A_5)$ & 14\\
&$A_2^2$ & $\triv$ & $\mathsf{H}_{1, q}(G_2) \otimes \mathsf{H}_1(A_1)$ & 8\\
&$E_6$ & $\triv, V, \Lambda^2V$ & $\mathsf{H}_1(A_1)$ & 6\\
\hline
$1/2$ & 1 & $\triv$ & $\mathsf{H}_{-1}(E_7)$ & 12\\
&$A_1$ & $\triv$ & $\mathsf{H}_{-1}(D_6)$ & 6\\
&$A_1^2$ & $\triv$ & $\mathsf{H}_{1, -1}(B_4) \otimes \mathsf{H}_{-1}(A_1)$ & 6\\
&$(A_1^3)'$ & $\triv$ & $\mathsf{H}_{-1, 1}(B_3) \otimes \mathsf{H}_{-1}(A_1)$ & 3\\
&$(A_1^3)''$&$\triv$&$\mathsf{H}_{1, -1}(F_4)$&9\\
&$A_1^4$ & $\triv$ & $\mathsf{H}_{-1, 1}(B_3)$ & 3\\
&$D_4$ & $\triv, (3, 1)$ & $\mathsf{H}_{-1, 1}(B_3)$ & 6\\
&$D_4 \times A_1$ & $\triv, (3, 1) \otimes \triv$ & $\mathsf{H}_{-1, 1}(B_2)$ & 4\\
&$D_6$ & $\triv, (0, 3^2), (1, 5), (2, 4)$ & $\mathsf{H}_{-1}(A_1)$ & 4\\
&$E_7$ & $\triv, V, \phi_{15, 7}, \phi_{21, 6}, \phi_{27, 2}, \phi_{35, 13}, \phi_{189, 5}$ & $\cplx$ & 7\\
\hline
\end{longtable}

\subsubsection{Generalized Hecke Algebras for $E_8$}  Next we produce a table describing the generalized Hecke algebras arising from $E_8$.  Again, we only list those parameters $c = 1/d$ for positive integers $d > 1$ dividing one of the fundamental degrees $2, 8, 12, 14, 18, 20, 24,$ and $30$ of $E_8$.  There are 112 irreducible representations of the group $E_8$.  By inspection, we see that Theorem \ref{simplified-coxeter-theorem} holds in type $E_8$ as well.

When the denominator of $c$ is 2, we see that there are 12 isomorphism classes of finite-dimensional irreducible representations.  Comparing with the lists of ``potential'' lowest weights appearing in \cite{GGJL}, we are again able to give a complete list of the finite-dimensional irreducible representations in his case.  By similar comparisons with the results of \cite{GGJL}, we are able to obtain such lists in all cases except when the denominator of $c$ is 3, 4 or 18.  In those cases, we give the list of ``potential'' lowest weights from \cite{GGJL} and the number of those which are in fact finite-dimensional.  In each of these three cases, we see that exactly one of these ``potential'' finite-dimensional representations is in fact infinite-dimensional.  Furthermore, this problem is resolved fairly easily when the denominator of $c$ is not 3.  For denominator 18, Rouquier \cite{Ro} proved that the representation $L_{1/18}(V)$ is finite-dimensional, and therefore the remaining ``potential'' finite-dimensional representation $L_{1/18}(\phi_{28, 8})$, is in fact infinite-dimensional.  That $L_{1/18}(\phi_{28, 8})$ is infinite-dimensional can be seen independently by the observation that its lowest $\eu$-weight, 2/3, is not a nonpositive integer.  When the denominator of $c$ is 4, the entire decomposition matrix for the block of $\oscr_{1/4}(E_8, \hfr)$ containing the simple object $L_{1/4}(\phi_{28, 8})$ can be easily produced following methods of Norton \cite[Lemmas 3.5, 3.6]{Norton}, yielding the equality $$[L_{1/4}(\phi_{28, 8})] = [\Delta_{1/4}(\phi_{28, 8})] - [\Delta_{1/4}(\phi_{700, 16})] + [\Delta_{1/4}(\phi_{1344, 19})] - [\Delta_{1/4}(\phi_{700, 28})] + [\Delta_{1/4}(\phi_{28, 68})]$$ in the Grothendieck group of $\oscr_{1/4}(E_8, \hfr)$.  It follows that $\supp(L_{1/4}(\phi_{28, 8})) = E_8\hfr^{D_4}$ and in particular that $L_{1/4}(\phi_{28, 8})$ is infinite-dimensional, ruling out this ``potential'' finite-dimensional representation.

\begin{longtable}{llllr}
\caption{Refined Filtration by Supports for $E_8$}
\label{E8-table}\\
\endfirsthead
\caption{continued}\\
\endhead
$c$ & $W'$ & $\lambda$ & $\hcal$ & $\#\irr$\\
\hline
1/30&1&$\triv$&$\mathsf{H}_q(E_8)$&111\\
&$E_8$&$\triv$&$\cplx$&1\\
\hline
1/24&1&$\triv$&$\mathsf{H}_q(E_8)$&111\\
&$E_8$&$\triv$&$\cplx$&1\\
\hline
1/20&1&$\triv$&$\mathsf{H}_q(E_8)$&111\\
&$E_8$&$\triv$&$\cplx$&1\\
\hline
1/18&1&$\triv$&$\mathsf{H}_q(E_8)$&109\\
&$E_7$&$\triv$&$\mathsf{H}_{q^3}(A_1)$&2\\
&$E_8$&$V$&$\cplx$&1\\
\hline
1/15&1&$\triv$&$\mathsf{H}_q(E_8)$&110\\
&$E_8$&$\triv, V$&$\cplx$&2\\
\hline
1/14&1&$\triv$&$\mathsf{H}_q(E_8)$&110\\
&$E_7$&$\triv$&$\mathsf{H}_{q}(A_1)$&2\\
\hline
1/12&1&$\triv$&$\mathsf{H}_q(E_8)$&102\\
&$E_6$&$\triv$&$\mathsf{H}_{q^3, q}(G_2)$&5\\
&$D_7$&$\triv$&$\mathsf{H}_{-1}(A_1)$&1\\
&$E_8$&$\triv, \phi_{28, 8}, \phi_{35, 2}, \phi_{50, 8}$&$\cplx$&4\\
\hline
1/10&1&$\triv$&$\mathsf{H}_q(E_8)$&104\\
&$D_6$&$\triv$&$\mathsf{H}_{q^2, q^3}(B_2)$&4\\
&$E_7$&$V$&$\mathsf{H}_{-1}(A_1)$&1\\
&$E_8$&$\triv, V, \phi_{28, 8}$&$\cplx$&3\\
\hline
1/9&1&$\triv$&$\mathsf{H}_q(E_8)$&106\\
&$E_6$&$\triv$&$\mathsf{H}_{1, q}(G_2)$&6\\
\hline
1/8&1&$\triv$&$\mathsf{H}_q(E_8)$&100\\
&$D_5$&$\triv$&$\mathsf{H}_{q^2, q}(B_3)$&9\\
&$A_7$&$\triv$&$\mathsf{H}_{-1}(A_1)$&1\\
&$E_8$&$\triv, \phi_{160, 7}$&$\cplx$&2\\
\hline
1/7&1&$\triv$&$\mathsf{H}_q(E_8)$&108\\
&$A_6$&$\triv$&$\mathsf{H}_{1, q}(A_1^2)$&4\\
\hline
1/6&1&$\triv$&$\mathsf{H}_q(E_8)$&75\\
&$D_4$&$\triv$&$\mathsf{H}_{q^2, q}(F_4)$&13\\
&$A_5$&$\triv$&$\mathsf{H}_{-1, q}(G_2) \otimes \mathsf{H}_{-1}(A_1)$&3\\
&$D_6$&$\triv$&$\mathsf{H}_{1, -1}(B_2)$&2\\
&$E_6$&$\triv, V$&$\mathsf{H}_{-1, q}(G_2)$&6\\
&$E_7$&$\triv, V, \phi_{15, 2}, \phi_{21, 6}$&$\mathsf{H}_{-1}(A_1)$&4\\
&$E_8$&$\triv, V, \phi_{28, 8}, \phi_{35, 2}, \phi_{50, 8}$&&\\
&&$\phi_{56, 19}, \phi_{175, 12}, \phi_{300, 8}, \phi_{210, 4}$&$\cplx$&9\\
\hline
1/5&1&$\triv$&$\mathsf{H}_q(E_8)$&96\\
&$A_4$&$\triv$&$\ints/2\ints \ltimes \mathsf{H}_q(A_4)$&12\\
&$E_8$&$\triv, V, \phi_{28, 8}, \phi_{56, 19}$&$\cplx$&4\\
\hline
1/4&1&$\triv$&$\mathsf{H}_q(E_8)$&69\\
&$A_3$&$\triv$&$\mathsf{H}_{-1, q}(B_5)$&14\\
&$D_4$&$\triv$&$\mathsf{H}_{1, q}(F_4)$&20\\
&$A_3^2$&$\triv$&$\mathsf{H}_{1, -1}(B_2)$&2\\
&$D_7$&$\triv, (2, 5)$&$\mathsf{H}_{-1}(A_1)$&2\\
&$E_8$&$\triv, \phi_{35, 2}, \phi_{50, 8}, \phi_{210, 4}, \phi_{350, 14}$&$\cplx$&5\\
\hline
1/3&1&$\triv$&$\mathsf{H}_q(E_8)$&52\\
&$A_2$&$\triv$&$\ints/2\ints \ltimes \mathsf{H}_q(E_6)$&26\\
&$A_2^2$&$\triv$&$\ints/2\ints \ltimes \mathsf{H}_{1, q}(G_2)^{\otimes 2}$&14\\
&$E_6$&$\triv, V, \Lambda^2V$&$\mathsf{H}_{1, q}(G_2)$&12\\
&$E_8$&exactly 8 of $\triv, V, \phi_{28, 8}, \phi_{35, 2}$\\
&&$\phi_{50, 8}, \phi_{160, 7}, \phi_{175, 12}, \phi_{300, 8}, \phi_{840, 13}$&$\cplx$&8\\
\hline
1/2&1&$\triv$&$\mathsf{H}_{-1}(E_8)$&23\\
&$A_1$&$\triv$&$\mathsf{H}_{-1}(E_7)$&12\\
&$A_1^2$&$\triv$&$\mathsf{H}_{1, -1}(B_6)$&12\\
&$A_1^3$&$\triv$&$\mathsf{H}_{-1, 1}(F_4)\otimes \mathsf{H}_{-1}(A_1)$&9\\
&$A_1^4$&$\triv$&$\mathsf{H}_{-1, 1}(B_4)$&5\\
&$D_4$&$\triv, (3, 1)$&$\mathsf{H}_{1, -1}(F_4)$&18\\
&$D_4 \times A_1$ & $\triv, (3, 1) \otimes \triv$ & $\mathsf{H}_{-1, 1}(B_3)$ & 6\\
&$D_6$&$\triv, (0, 3^2), (1, 5), (2, 4)$&$\mathsf{H}_{-1, 1}(B_2)$&8\\
&$E_7$&$\triv, V, \phi_{15, 7}, \phi_{21, 6}, \phi_{27, 2}, \phi_{35, 13}, \phi_{189, 5}$&$\mathsf{H}_{-1}(A_1)$&7\\
&$E_8$&$\triv, V, \phi_{28, 8}, \phi_{35, 2}, \phi_{50, 8}, \phi_{175, 12},$\\
&&$\phi_{300, 8}, \phi_{210, 4}, \phi_{560, 5}, \phi_{840, 14}, \phi_{1050, 10}, \phi_{1400, 8}$&$\cplx$&12\\
\hline
\end{longtable}

\subsection{Type $H$}\label{typeH}

\subsubsection{Generalized Hecke Algebras for $H_3$}  Next we produce a table describing the generalized Hecke algebras arising from $H_3$.  Again, we only list those parameters $c = 1/d$ for positive integers $d > 1$ dividing one of the fundamental degrees $2, 6$ and 10 of $H_3$.  There are 10 irreducible representations of the group $H_3$.  By inspection, we see that Theorem \ref{simplified-coxeter-theorem} holds in type $H_3$ as well.

\begin{longtable}{llllr}
\caption{Refined Filtration by Supports for $H_3$}
\label{H3-table}\\
\endfirsthead
\caption{continued}\\
\endhead
$c$ & $W'$ & $\lambda$ & $\hcal$ & $\#\irr$\\
\hline
1/10 & 1 & $\triv$ & $\mathsf{H}_q(H_3)$ & 9\\
&$H_3$&$\triv$&$\cplx$&1\\
\hline
1/6 & 1 & $\triv$ & $\mathsf{H}_q(H_3)$ & 9\\
&$H_3$&$\triv$&$\cplx$&1\\
\hline
1/5 & 1 & $\triv$ & $\mathsf{H}_q(H_3)$ & 8\\
&$I_2(5)$&$\triv$&$\mathsf{H}_1(A_1)$&2\\
\hline
1/3 & 1 & $\triv$ & $\mathsf{H}_q(H_3)$ & 8\\
&$A_2$&$\triv$&$\mathsf{H}_1(A_1)$&2\\
\hline
1/2 & 1 & $\triv$ & $\mathsf{H}_q(H_3)$ & 5\\
&$A_1$&$\triv$&$\mathsf{H}_{-1}(A_1^2)$&1\\
&$A_1^2$&$\triv$&$\mathsf{H}_{-1}(A_1)$&1\\
&$H_3$&$\triv, V, \widetilde{V}$&$\cplx$&3\\
\hline
\end{longtable}

\subsubsection{Generalized Hecke Algebras for $H_4$}  Next we produce a table describing the generalized Hecke algebras arising from $H_4$.  Again, we only list those parameters $c = 1/d$ for positive integers $d > 1$ dividing one of the fundamental degrees $2, 12, 20$ and 30 of $H_4$.  There are 34 irreducible representations of the group $H_4$.  By inspection, we see that Theorem \ref{simplified-coxeter-theorem} holds in type $H_4$ as well.

When the denominator of $c$ is not equal to 2, the list of lowest weights $\lambda$ giving the finite-dimensional irreducible representations of $H_c(H_4, \hfr)$ are obtained from results of Norton \cite{Norton}.  When the denominator is 2, our count below shows that there are exactly 6 isomorphism classes of finite-dimensional irreducible representations of $H_c(H_4, \hfr)$.  In \cite[Section 5.6]{GGJL}, all except 6 of the possible lowest weights $\lambda$ for the finite-dimensional irreducible representations $L_c(\lambda)$ are ruled out, and in particular those ``potential'' lowest weights are in fact precisely the lowest weights of the finite-dimensional $L_c(\lambda)$.  This confirms a conjecture of Norton \cite{Norton} about the classification of these representations.

\begin{longtable}{llllr}
\caption{Refined Filtration by Supports for $H_4$}
\label{H4-table}\\
\endfirsthead
\caption{continued}\\
\endhead
$c$ & $W'$ & $\lambda$ & $\hcal$ & $\#\irr$\\
\hline
1/30 & 1 & $\triv$ & $\mathsf{H}_q(H_4)$ &33\\
&$H_4$&$\triv$&$\cplx$&1\\
\hline
1/20 & 1 & $\triv$ & $\mathsf{H}_q(H_4)$ &33\\
&$H_4$&$\triv$&$\cplx$&1\\
\hline
1/15 & 1 & $\triv$ & $\mathsf{H}_q(H_4)$ &32\\
&$H_4$&$\triv, \widetilde{V}$&$\cplx$&2\\
\hline
1/12 & 1 & $\triv$ & $\mathsf{H}_q(H_4)$ &33\\
&$H_4$&$\triv$&$\cplx$&1\\
\hline
1/10 & 1 & $\triv$ & $\mathsf{H}_q(H_4)$ &29\\
&$H_3$&$\triv$&$\mathsf{H}_{-1}(A_1)$&1\\
&$H_4$&$\triv, V, \widetilde{V}, \phi_{9, 6}$&$\cplx$&4\\
\hline
1/6 & 1 & $\triv$ & $\mathsf{H}_q(H_4)$ &30\\
&$H_3$&$\triv$&$\mathsf{H}_{-1}(A_1)$&1\\
&$H_4$&$\triv, V, \widetilde{V}$&$\cplx$&3\\
\hline
1/5 & 1 & $\triv$ & $\mathsf{H}_q(H_4)$ &24\\
&$I_2(5)$&$\triv$&$\mathsf{H}_{1, q}(I_2(10))$&6\\
&$H_4$&$\triv, V, \phi_{9, 2}, \phi_{16, 6}$&$\cplx$&4\\
\hline
1/4 & 1 & $\triv$ & $\mathsf{H}_q(H_4)$ &30\\
&$A_3$&$\triv$&$\mathsf{H}_{-1}(A_1)$&1\\
&$H_4$&$\triv, \phi_{9, 2}, \phi_{9, 6}$&$\cplx$&3\\
\hline
1/3 & 1 & $\triv$ & $\mathsf{H}_q(H_4)$ &26\\
&$A_2$&$\triv$&$\mathsf{H}_{1, q}(G_2)$&4\\
&$H_4$&$\triv, V, \widetilde{V}, \phi_{16, 3}$&$\cplx$&4\\
\hline
1/2 & 1 & $\triv$ & $\mathsf{H}_{-1}(H_4)$ &18\\
&$A_1$&$\triv$&$\mathsf{H}_{-1}(H_3)$&5\\
&$A_1^2$&$\triv$&$\mathsf{H}_{1, -1}(B_2)$&2\\
&$H_3$&$\triv, V, \widetilde{V}$&$\mathsf{H}_{-1}(A_1)$&3\\
&$H_4$&$\triv, V, \widetilde{V}, \phi_{9, 2}, \phi_{9, 6}, \phi_{25, 4}$&$\cplx$&6\\
\hline
\end{longtable}

\subsection{Type $F_4$ with Unequal Parameters} \label{typeF4} In this section we will both prove Theorem \ref{simplified-coxeter-theorem} for rational Cherednik algebras of type $F_4$ and count the number of irreducible representations in $\oscr_{c_1, c_2}(F_4, \hfr)$ of each possible support for all values of the parameters $c_1, c_2$, including in the case of unequal parameters.  As a corollary, comparing with the results of \cite{GGJL}, we classify the irreducible finite-dimensional representations of $H_{1/2, 1/2}(F_4, \hfr)$.  This confirms a conjecture of Norton \cite{Norton} about these representations and completes the classification of the irreducible finite-dimensional representations of rational Cherednik algebras of type $F_4$ with equal parameters, the other equal parameter cases having been treated by Norton.  We expect that comparing our counts with the results of \cite{GGJL} completes the classification of the finite-dimensional irreducible representations for many other parameter values as well, although we do not perform this comparison here.

First, we fix some notation for the Coxeter group $F_4$.  Our convention will be that the simple roots of $F_4$ are labeled $s_1, s_2, s_3, s_4$ where $s_1, s_2$ are short simple roots (with parameter $c_1 \in \cplx$) and $s_3, s_4$ are long simple roots (with parameter $c_2 \in \cplx$), and that $s_2, s_3$ are adjacent in the Dynkin diagram.  The $KZ$ parameters are given by $p := e^{-2 \pi i c_1}$ and $q := e^{-2 \pi i c_2}$.  We label standard parabolic subgroups, one from each conjugacy class, as follows:

$A_1' := \langle s_1 \rangle$

$A_1 := \langle s_4 \rangle$

$A_1' \times A_1 := \langle s_1, s_4 \rangle$

$A_2' := \langle s_1, s_2 \rangle$

$A_2 := \langle s_3, s_4 \rangle$

$B_2 := \langle s_2, s_3 \rangle$

$C_3 := \langle s_1, s_2, s_3 \rangle$

$B_3 := \langle s_2, s_3, s_4 \rangle$

$A_1' \times A_2 := \langle s_1, s_3, s_4 \rangle$

$A_2' \times A_1 := \langle s_1, s_2, s_4 \rangle$

Theorem \ref{simplified-coxeter-theorem} follows in the case $W = F_4$ by application of Remark \ref{lowest-weight-remark} to the cases in which $W' \subset W$ is one of the standard parabolic subgroups appearing above.  In particular, when $W'$ is a product of Coxeter groups of type $A$, the only irreducible representations $\lambda$ of $W'$ which can appear as the lowest weights of a finite-dimensional representation $L_c(\lambda)$ are linear characters, i.e. tensor products of either trivial or sign representations.  These representations always extend to representations of $N_{F_4}(W')$, and then the statement Theorem \ref{simplified-coxeter-theorem} follows from Remark \ref{lowest-weight-remark}.  The remaining parabolic subgroups $W'$ are $B_2, B_3$, and $C_3$.  In each of these cases, Howlett \cite{How} has shown that the normalizer $N_{F_4}(W')$ splits as a direct product $W' \times W''$, and Theorem \ref{simplified-coxeter-theorem} follows by Remark \ref{lowest-weight-remark} in this case as well.

The classification of the irreducible finite-dimensional representations of the rational Cherednik algebras $H_c(W', \hfr_{W'})$ where $W' \subset F_4$ is one of the proper nontrivial parabolic subgroups appearing above is well known.  Applying our results to each of these cases, we can count the number of irreducible representations in $\oscr_c(F_4, \hfr)$ with support labeled by any of these parabolic subgroups.  This information is presented in the following tables; the left column of each table specifies certain conditions on the parameters $c_1, c_2$ in terms of the $KZ$ parameters $p = e^{-2 \pi i c_1}$ and $q = e^{-2 \pi i c_2}$, and the right column of each table gives the number of irreducible representations in $\oscr_{c_1, c_2}(F_4, \hfr)$ with support labeled by the parabolic subgroup appearing in the title of the table.  Here $\Phi_d$ denotes the $d^{th}$ cyclotomic polynomial.

\begin{longtable}{lr}
\caption{Simple Modules for $H_{c_1, c_2}(F_4)$ labeled by $A_1'$}
\label{A1PF4supporttable}\\
\endfirsthead
\caption{continued}\\
\endhead
condition & \# simples labeled\\
\hline
$p = -1$ and $\Phi_1(q)\Phi_2(q)\Phi_3(q)\Phi_4(q) \neq 0$ & 9\\
\hline
$p = -1$ and $\Phi_1(q)\Phi_2(q) = 0$ & 3\\
\hline
$p = -1$ and $\Phi_3(q) = 0$ & 6\\
\hline
$p = -1$ and $\Phi_4(q) = 0$ & 8\\
\hline
otherwise & 0\\
\hline
\end{longtable}

\begin{longtable}{lr}
\caption{Simple Modules for $H_{c_1, c_2}(F_4)$ labeled by $A_1$}
\label{A1F4supporttable}\\
\endfirsthead
\caption{continued}\\
\endhead
condition & \# simples labeled\\
\hline
$q = -1$ and $\Phi_1(p)\Phi_2(p)\Phi_3(p)\Phi_4(p) \neq 0$ & 9\\
\hline
$q = -1$ and $\Phi_1(p)\Phi_2(p) = 0$ & 3\\
\hline
$q = -1$ and $\Phi_3(p) = 0$ & 6\\
\hline
$q = -1$ and $\Phi_4(p) = 0$ & 8\\
\hline
otherwise & 0\\
\hline
\end{longtable}

\begin{longtable}{lr}
\caption{Simple Modules labeled by $A_1' \times A_1$}
\label{A1PA1F4supporttable}\\
\endfirsthead
\caption{continued}\\
\endhead
condition & \# simples labeled\\
\hline
$p = q = -1$ & 1\\
\hline
otherwise & 0\\
\hline
\end{longtable}

\begin{longtable}{lr}
\caption{Simple Modules labeled by $A_2'$}
\label{A2PF4supporttable}\\
\endfirsthead
\caption{continued}\\
\endhead
condition & \# simples labeled\\
\hline
$\Phi_3(p) = 0$ and $\Phi_2(q)\Phi_3(q)\Phi_6(q)\Phi_{12}(q) \neq 0$ & 6\\
\hline
$\Phi_3(p) = 0$ and $\Phi_2(q)\Phi_6(q) = 0$ & 3\\
\hline
$\Phi_3(p) = 0$ and $\Phi_3(q) = 0$ & 4\\
\hline
$\Phi_3(p) = 0$ and $\Phi_{12}(q) = 0$ & 5\\
\hline
otherwise & 0\\
\hline
\end{longtable}

\begin{longtable}{lr}
\caption{Simple Modules labeled by $A_2$}
\label{A2F4supporttable}\\
\endfirsthead
\caption{continued}\\
\endhead
condition & \# simples labeled\\
\hline
$\Phi_3(q) = 0$ and $\Phi_2(p)\Phi_3(p)\Phi_6(p)\Phi_{12}(p) \neq 0$ & 6\\
\hline
$\Phi_3(q) = 0$ and $\Phi_2(p)\Phi_6(p) = 0$ & 3\\
\hline
$\Phi_3(q) = 0$ and $\Phi_3(p) = 0$ & 4\\
\hline
$\Phi_3(q) = 0$ and $\Phi_{12}(p) = 0$ & 5\\
\hline
otherwise & 0\\
\hline
\end{longtable}

\begin{longtable}{lr}
\caption{Simple Modules labeled by $B_2$}
\label{B2F4supporttable}\\
\endfirsthead
\caption{continued}\\
\endhead
condition & \# simples labeled\\
\hline
$p = -1$ and $q = -1$ & 2\\
\hline
$p^{\pm 1} = -q$ and $\Phi_3(q)\Phi_6(q) = 0$ & 2\\
\hline
$p^{\pm 1} = -q$ and $\Phi_3(q)\Phi_6(q) \neq 0$ & 4\\
\hline
otherwise & 0\\
\hline
\end{longtable}

\begin{longtable}{lr}
\caption{Simple Modules labeled by $B_3$}
\label{B3F4supportstable}\\
\endfirsthead
\caption{continued}\\
\endhead
condition & \# simples labeled\\
\hline
$p = -q^{\pm 2}$ and $\Phi_1(q)\Phi_6(q)\Phi_{12}(q) \neq 0$ & 2\\
\hline
$p = -q^{\pm 2}$ and $\Phi_6(q)\Phi_{12}(q) = 0$ & 1\\
\hline
$p = -1 = -q^{\pm 2}$ and $q = 1$ & 3\\
\hline
$p = -1$ and $\Phi_3(q) = 0$ & 1\\
\hline
otherwise & 0\\
\hline
\end{longtable}

\begin{longtable}{lr}
\caption{Simple Modules labeled by $C_3$}
\label{C3F4supportstable}\\
\endfirsthead
\caption{continued}\\
\endhead
condition & \# simples labeled\\
\hline
$q = -p^{\pm 2}$ and $\Phi_1(p)\Phi_6(p)\Phi_{12}(p) \neq 0$ & 2\\
\hline
$q = -p^{\pm 2}$ and $\Phi_6(p)\Phi_{12}(p) = 0$ & 1\\
\hline
$q = -1 = -p^{\pm 2}$ and $p = 1$ & 3\\
\hline
$q = -1$ and $\Phi_3(p) = 0$ & 1\\
\hline
otherwise & 0\\
\hline
\end{longtable}

\begin{longtable}{lr}
\caption{Simple Modules labeled by $A_1' \times A_2$}
\label{A1PA2F4supporttable}\\
\endfirsthead
\caption{continued}\\
\endhead
condition & \# simples labeled\\
\hline
$p = -1$ and $\Phi_3(q) = 0$ & 1\\
\hline
otherwise & 0\\
\hline
\end{longtable}

\begin{longtable}{lr}
\caption{Simple Modules labeled by $A_2' \times A_1$}
\label{A2PA1PF4supporttable}\\
\endfirsthead
\caption{continued}\\
\endhead
condition & \# simples labeled\\
\hline
$\Phi_3(p) = 0$ and $q = -1$ & 1\\
\hline
otherwise & 0\\
\hline
\end{longtable}

To count the irreducible finite-dimensional representations for given parameters $(c_1, c_2)$, we need first to count the irreducible representations in $\oscr_{c_1, c_2}(F_4, \hfr)$ of full support, which is achieved by counting the number of irreducible representations of the Hecke algebra $\mathsf{H}_{p, q}(F_4)$.  The following table, produced by computations in CHEVIE in GAP3 \cite{GAP1, GAP2} and in Mathematica, gives this number of irreducible representations in all cases up to obvious symmetries.  In particular, the isomorphism class as a $\cplx$-algebra of the Hecke algebra $\mathsf{H}_{p, q}(F_4)$ does not change when the order of the parameters $p, q$ is reversed or when either of the parameters is replaced with its inverse.  The following table then gives the number of irreducible representations of $\mathsf{H}_{p, q}(F_4)$ for a complete set of parameters $p, q \in \cplx^\times$ up to these symmetries.  The first column specifies a hypersurface in the $p,q$-plane at which the Hecke algebra $\mathsf{H}_{p, q}(F_4)$ is not semisimple, and the second column specifies additional conditions on $q$, and the final column gives the corresponding number of irreducible representations of $\mathsf{H}_{p, q}(F_4)$.

\begin{longtable}{llr}
\caption{Irreducible Representations of $\mathsf{H}_{p, q}(F_4)$}
\label{F4fullsupptable}\\
\endfirsthead
\caption{continued}\\
\endhead
condition & $q$ condition & \# simples\\
\hline
$p = -1$ & $q = 1$ & 9\\
& $q = -1$ & 8\\
& $q = $ roots of unity order 3 & 11\\
& $q = $ roots of unity of order 4 or 6& 14\\
& otherwise & 15\\
\hline
$p^{\pm 1} + q = 0$ & $q = \pm 1$ & 9\\
& $q = $ roots of unity of order 3 or 6& 13\\
& $q = $ roots of unity of order 8& 18\\
&otherwise&19\\
\hline
$p^{\pm 1} + q^2 = 0$ & $q = 1$ & 9\\
&$q = -1$ & 8\\
&$q = $ roots of unity of order 3 &13\\
&$q = $ roots of unity of order 4 or 6 &20\\
&$q = $ roots of unity of order 9 &22\\
&$q = $ roots of unity of order 10 &21\\
&$q = $ roots of unity of order 12 &17\\
&otherwise&23\\
\hline
$p^{\pm 1} = iq$ & $q = 1$ or $q = -i$ & 20\\
&$q = -1$ & 14\\
&$q = $ roots of unity of order 3&17\\
&$q = i$&14\\
&$q = $ roots of unity of order 6&23\\
&$q \in \{e^{2 \pi i/8}, e^{5 \cdot 2 \pi i/8}\}$&18\\
&$q \in \{e^{2 \pi i/12}, e^{5\cdot 2 \pi i/12}\}$&17\\
&$q \in \{e^{7\cdot 2 \pi i/12}, e^{11 \cdot 2 \pi i/12}\}$&23\\
&$q \in \{e^{7\cdot 2 \pi i/24}, e^{11\cdot 2 \pi i/24}, e^{19\cdot 2 \pi i/24}, e^{23\cdot 2 \pi i/24}\}$&23\\
&otherwise&24\\
\hline
$p$ is a root of unity of order 3 & $q = -1$ & 11\\
&$q = $ roots of unity of order 3& 15\\
&$q = $ roots of unity of order 6 & 13\\
&$q = $ roots of unity of order 12 & 17\\
&otherwise&19\\
\hline
$p$ is a root of unity of order 6&$q = 1$&22\\
&$q = -1$&14\\
&$q = $ roots of unity of order 3&13\\
&$q = $ roots of unity of order 6&20\\
&$q = $ roots of unity of order 12&23\\
&otherwise&24\\
\hline
$p^{\pm 1} = e^{2 \pi i/6}q$&$q = 1$&22\\
&$q = -1$&11\\
&$q = e^{2 \pi i/3}$&11\\
&$q = e^{2\cdot 2 \pi i/3}$&13\\
&$q = e^{2 \pi i/6}$&13\\
&$q = e^{5 \cdot 2\pi i/6}$&22\\
&$q \in \{e^{2 \pi i/9}, e^{4\cdot 2\pi i/9}, e^{7\cdot 2\pi i/9}\}$&22\\
&$q \in \{e^{2 \pi i/18}, e^{7\cdot 2 \pi i/18}, e^{13\cdot 2\pi i/18}\}$&22\\
&$q \in \{e^{2 \pi i/24}, e^{7 \cdot 2 \pi i/24}, e^{13\cdot2 \pi i/24}, e^{19\cdot 2\pi i/24}\}$&23\\
&otherwise&24\\
\hline
otherwise && 25\\
\hline
\end{longtable}

The following table consolidates the information in the previous tables and counts the number of irreducible finite-dimensional representations of $H_{c_1, c_2}(F_4, \hfr)$ for all parameters $c_1, c_2$ in terms of the $KZ$ parameters $p$ and $q$, up to the same symmetries mentioned above, i.e. up to exchanging and inverting $p$ and $q$, for which the category $\oscr_{c_1, c_2}(F_4, \hfr)$ is not semisimple.  The first column specifies conditions on the parameters $p$ and $q$.  The remaining columns give the counts of the number of irreducible representations in $\oscr_{c_1, c_2}(F_4, \hfr)$, for any parameters $c_1, c_2$ satisfying $p = e^{-2 \pi i c_1}$ and $q = e^{-2 \pi i c_2}$, with support labeled by the parabolic subgroup appearing at the top of the column.  In particular, the last column, labeled by $F_4$, counts finite-dimensional representations.  This last column is obtained from the others by subtracting their sum from 25, the number of irreducible complex representations of the Coxeter group $F_4$.  For those conditions on $p$ and $q$ specifying a certain finite set of points, we give a defining ideal for these points; for example, $(\Phi_2(p), \Phi_3(q))$ specifies that $p = -1$ and that $q$ is a primitive cube root of unity.  The remaining conditions specify certain curves with certain exceptional points removed.

\begin{longtable}{llllllllllllr}
\caption{Modules of Given Support in $H_{c_1, c_2}(F_4)$}
\label{F4supportstable}\\
\endfirsthead
\caption{continued}\\
\endhead
condition & 1 & $A_1'$ & $A_1$ & $A_1' \times A_1$ & $A_2'$ & $A_2$ & $B_2$ & $B_3$ & $C_3$ & $A_1' \times A_2$ & $A_2' \times A_1$ & $F_4$\\
\hline
$(\Phi_1(p), \Phi_2(q))$&9&.&3&.&.&.&4&.&3&.&.&6\\
$(\Phi_1(p), \Phi_4(q))$&20&.&.&.&.&.&.&2&.&.&.&3\\
$(\Phi_1(p), \Phi_6(q))$&22&.&.&.&.&.&.&.&.&.&.&3\\

$(\Phi_2(p), \Phi_2(q))$&8&3&3&1&.&.&2&2&2&.&.&4\\
$(\Phi_2(p), \Phi_3(q))$&11&6&.&.&.&3&.&1&.&1&.&3\\
$(\Phi_2(p), \Phi_4(q))$&14&8&.&.&.&.&.&.&.&.&.&3\\
$(\Phi_2(p), \Phi_6(q))$&14&9&.&.&.&.&.&.&.&.&.&2\\

$(\Phi_3(p), \Phi_3(q))$&15&.&.&.&4&4&.&.&.&.&.&2\\
$(\Phi_3(p), \Phi_6(q))$&13&.&.&.&3&.&2&.&2&.&.&5\\
$(\Phi_3(p), \Phi_{12}(q))$&17&.&.&.&5&.&.&1&.&.&.&2\\

$(\Phi_4(p), \Phi_4(q))$&19&.&.&.&.&.&4&.&.&.&.&2\\

$(\Phi_6(p), \Phi_{6}(q))$&20&.&.&.&.&.&.&1&1&.&.&3\\
$(\Phi_6(p), \Phi_{12}(q))$&23&.&.&.&.&.&.&.&.&.&.&2\\
\hline
$(\Phi_8(p), \Phi_8(q),$\\
$\Phi_1(pq)\Phi_1(pq^{-1}))$&24&.&.&.&.&.&.&.&.&.&.&1\\
\hline
$(\Phi_8(p), \Phi_8(q),$\\
$\Phi_2(pq)\Phi_2(pq^{-1}))$&18&.&.&.&.&.&4&.&.&.&.&3\\
\hline
$(\Phi_9(p), \Phi_{18}(q),$\\
$\Phi_2(p^2q)\Phi_2(p^2q^{-1}))$&22&.&.&.&.&.&.&.&2&.&.&1\\
\hline
$(\Phi_{10}(p), \Phi_{10}(q),$\\
$\Phi_2(pq^2)\Phi_2(p^{-1}q^2))$&21&.&.&.&.&.&.&2&2&.&.&.\\
\hline
$(\Phi_{12}(p), \Phi_{12}(q),$\\
$\Phi_1(pq)\Phi_1(p^{-1}q))$&24&.&.&.&.&.&.&.&.&.&.&1\\
\hline
$(\Phi_{24}(p), \Phi_{24}(q),$\\
$\Phi_4(pq), \Phi_6(p^{-1}q))$&23&.&.&.&.&.&.&.&.&.&.&2\\
\hline
$\Phi_2(p) = 0$ and\\
$(q^6 - 1)\Phi_4(q) \neq 0$&15&9&.&.&.&.&.&.&.&.&.&1\\
\hline
$\Phi_3(p) = 0$ and\\
$\Phi_2(q)\Phi_3(q)\Phi_6(q)\cdot$\\
$\Phi_{12}(q) \neq 0$&19&.&.&.&6&.&.&.&.&.&.&.\\
\hline
$\Phi_6(p) = 0$ and\\
$(q^6 - 1)\Phi_{12}(q) \neq 0$&24&.&.&.&.&.&.&.&.&.&.&1\\
\hline
$\Phi_2(pq) = 0$ and\\
$(q^6 - 1)\Phi_8(q) \neq 0$&19&.&.&.&.&.&4&.&.&.&.&2\\
\hline
$\Phi_2(pq^2) = 0$ and\\
$(q^{12} - 1)\Phi_9(q)\cdot$\\
$\Phi_{10}(q) \neq 0$&23&.&.&.&.&.&.&2&.&.&.&.\\
\hline
$\Phi_4(pq) = 0$ and\\
$(q^{12} - 1)\Phi_2(pq^{-1})\cdot$\\
$\Phi_6(pq^{-1}) \neq 0$&24&.&.&.&.&.&.&.&.&.&.&1\\
\hline
$\Phi_6(pq) = 0$ and\\
$(q^6 - 1)\Phi_4(pq^{-1})\cdot$\\
$\Phi_2(pq^{-2})\cdot$\\
$\Phi_2(p^{-2}q) \neq 0$&24&.&.&.&.&.&.&.&.&.&.&1\\
\hline
\end{longtable}

Table \ref{F4supportstable} shows in particular that there are exactly 4 distinct irreducible finite-dimensional representations of the rational Cherednik algebra $H_{1/2, 1/2}(F_4, \hfr)$.  In notation compatible with the labeling of the irreducible representations of the Coxeter group $F_4$ in GAP3, it is shown in \cite[Section 5.7.2]{GGJL} that with equal parameters $c_1 = c_2 = 1/2$ the lowest weights $\lambda$ such that the representation $L_{1/2, 1/2}(\lambda)$ is finite-dimensional are among the representations $\phi_{1, 0}, \phi'_{2, 4}, \phi''_{2, 4}$, and $\phi_{9, 2}.$  In particular, we have the following corollary, confirming Norton's conjecture on the classification of the irreducible finite-dimensional representations of $H_{1/2, 1/2}(F_4, \hfr)$:

\begin{corollary} The set of irreducible representations $\lambda$ of the Coxeter group $F_4$ for which the irreducible lowest weight representation $L_{1/2, 1/2}(F_4, \hfr)$ is finite-dimensional is $$\{\phi_{1, 0}, \phi'_{2, 4}, \phi''_{2, 4}, \phi_{9, 2}\}.$$\end{corollary}

\subsection{Type $I$} \label{typeI} The only remaining case in which Theorem \ref{simplified-coxeter-theorem} needs to be verified is the case of the irreducible Coxeter groups of type $I$, i.e. the dihedral groups.  For $d \geq 3$, let $I_2(d)$ denote the dihedral Coxeter group with $2d$ elements.  Chmutova \cite{Chm} has computed character formulas for all irreducible representations $L_c(\lambda)$ in the category $\oscr_c(I_2(d), \hfr)$ for all parameters $c$, and in particular the supports of all irreducible representations in $\oscr_c(I_2(d), \hfr)$ are known in this case, so we will not produce tables counting the number of irreducible representations of given support in this case.

To see that Theorem \ref{simplified-coxeter-theorem} holds in type $I$, we need only consider the case of rank 1 parabolic subgroups, i.e. those of type $A_1$.  Let $d \geq 3$ be an integer and let $W' \subset I_2(d)$ be a parabolic subgroup of type $A_1$.  If the parameter $c_1$ attached to the reflection generating $A_1$ does not satisfy $c_1 \in 1/2 + \ints$, the rational Cherednik algebra $H_{c_1}(W', \hfr_{W'})$ does not admit any nonzero finite-dimensional representations.  Otherwise, there is a unique finite-dimensional representation $L$ of $H_{c_1}(W', \hfr_{W'})$ with lowest weight $\triv$ or $\sgn$, depending on whether $c_1 > 0$ or $c_1 < 0$, respectively.  In either case, $L$ is $N_{I_2(d)}(W')$-equivariant by Remark \ref{lowest-weight-remark}.  This completes the proof of Theorem \ref{simplified-coxeter-theorem}.

\end{document}